\documentclass[10pt]{article}
\usepackage[realmainfile]{currfile}
\usepackage[dvipsnames]{xcolor}
\usepackage{graphicx}
\usepackage{amscd}
\usepackage{enumitem}

\usepackage{bm}
\usepackage[french,english]{babel}
\usepackage{caption}
\usepackage[utf8]{inputenc}
\usepackage[T1]{fontenc}
\usepackage{rotating,amssymb,amsmath,amsfonts,delarray}
\usepackage{color}
\usepackage{placeins}
\usepackage{float}
\usepackage{amsthm}

\usepackage{amsfonts}
\usepackage{color}
\usepackage{hyperref}
\usepackage{booktabs}

\usepackage{lipsum} 

\usepackage{graphicx}
\usepackage{caption}
\usepackage{subcaption}
\usepackage{float}

\usepackage{tikz-qtree}
\hypersetup{
	colorlinks=true,
	citecolor=blue,
	linkcolor=blue,
	filecolor=magenta,      
	urlcolor=cyan,
}

\makeatletter
\renewcommand\@biblabel[1]{\textbullet}
\makeatother

\definecolor{amethyst}{rgb}{0.6, 0.4, 0.8}

\usepackage{geometry}
\usepackage{txfonts}
\usepackage[bb=dsserif]{mathalpha}
\usepackage{authblk}

\newtheorem{deff}{Definition}
\newtheorem{lem}{Lemma}
\newtheorem{prop}{Proposition}
\usepackage{thmtools}
\newtheorem{ex}{Example}
\newtheorem{theorem}{Theorem}
\newtheorem{cor}{Corollary}

\usepackage{tikz}
\usetikzlibrary{trees}
\usetikzlibrary{shapes}
\usepackage{float}
\usepackage{wrapfig}
\usepackage{xcolor}
\definecolor{amethyst}{rgb}{0.6, 0.4, 0.8}

\definecolor{codegreen}{rgb}{0,0.6,0}
\definecolor{codegray}{rgb}{0.5,0.5,0.5}
\definecolor{codepurple}{rgb}{0.58,0,0.82}
\definecolor{backcolour}{rgb}{0.9,0.9,0.9}

\usepackage{pgfplotstable}

\usepackage[ruled,vlined]{algorithm2e}

\usepackage{multicol}
\usepackage{tikz, pgfplots}
\pgfplotsset{compat=1.17}
\usetikzlibrary[arrows.meta,bending, calc, positioning, patterns]
\usetikzlibrary{positioning}

\definecolor{Rcolor}{RGB}{150,160,190}
\newcommand{\Rx}{\fontsize{10pt}{12pt}\selectfont
	\raisebox{.3em}{\hspace{1.2em}%
		\llap{\resizebox{1.09em}{.5em}{\color{black}$\bigcirc$}}%
		\llap{\resizebox{1.199em}{.55em}{\color{darkgray}$\bigcirc$}}%
		\llap{\resizebox{1.19em}{.52em}{\color{gray!50}$\bigcirc$}}%
		\llap{\resizebox{1.1em}{.5em}{\color{gray}$\bigcirc$}}%
		\llap{\resizebox{1.25em}{.55em}{\color{gray}$\bigcirc$}}%
	}%
	\hspace{-.85em}%
	\textbf{%
		\textcolor{black}{\textsf{R}}%
		\hspace{-.025em}\raisebox{.01em}{\llap{\textcolor{Rcolor}{\textsf{R}}}}%
}}%
\newbox\rbox
\savebox\rbox{\scalebox{0.1}{\Rx}}

\makeatletter
\def\R{\scalebox{\f@size}{\usebox\rbox}\xspace}
\makeatletter

\newtheorem{innercustomthm}{Theorem}
\newenvironment{theoremeV}[1]
{\renewcommand\theinnercustomthm{#1}\innercustomthm}
{\endinnercustomthm}

\newcommand{\vecun}[1]{\,{\boldsymbol{1}_{|{#1}\mathrm{dsc}({#1})|}}}

\begin{document}
\sloppy
	\title{Centrality and shape-related comparisons in a tree-structured Markov random field}
	\author{Benjamin Côté$^*$, Hélène Cossette$^\dagger$, Etienne Marceau$^\dagger$ \and
        $*$ \textit{Department of Statistics and Actuarial Science, University of Waterloo, Canada} \\  
		$\dagger$ \textit{École d'actuariat, Université Laval, Canada}
	}
        \date{October 26, 2024\\ This version: August 29, 2025}
	\maketitle
	
\begin{abstract}
 Understanding the effects of the choice of the tree on the joint distribution of a tree-structured Markov random field (MRF) is crucial for fully exploiting 
the intelligibility of such probabilistic graphical models. Tools must be developed in this regard: this is the overarching objective of this paper. Our discussion is two-fold. First, we examine concepts specific to network centrality theory. We put forth a new conception of centrality for MRFs that not only accounts for the tree topology, but also the underlying stochastic dynamics. 
In this vein, we compare synecdochic pairs, random vectors comprising a MRF's component and its sum, using stochastic orders. The resulting orderings are transferred to risk-allocation quantities, which therefore serve as new centrality indices tailored to our stochastic framework. 
Second, we shed light on the influence of the tree's shapes, by establishing convex orderings for MRFs encrypted on trees of different shapes. This results in the design of a new partial order on tree shapes. This analysis is done within the framework of a propagation-based family of tree-structured MRFs with the uncommon property of having fixed Poisson marginal distributions unaffected by the dependence scheme. This work is a first step into the analysis of MRFs' trees' shapes and a stepping stone to extending this analysis to a broader framework.        
\end{abstract}

\bigskip

\section{Introduction}

Networks, or \textit{graphs}, abound in the real world and underlie most complex systems and structures, some of which still elude our complete understanding. The human brain's neurological connections, individuals' interactions in a society, and the Internet are such examples. Section~2 of \cite{newman2003structure} provides multiple more real-world examples of networks.
Networks are constituted of vertices and edges, representing agents and their interrelations. The agents' behavior may be deterministic or random; in the latter case, the network schematization of their interrelations describes components of a random vector and their dependence relations. 
This idea is behind Markov random fields (MRF), also referred to as \textit{Markov networks} or \textit{undirected probabilistic graphical models}. MRFs are widely used in statistical modeling (see \cite{cressie2015statistics}); their main characteristic is that their underlying graph translates directly into conditional independencies among their components. For an introduction to MRFs, one may consult \cite{koller2009probabilistic}, \cite{maathuis2018handbook} and \cite{wainwright2008graphical}.

\bigskip

Trees are a specific class of graphs, most convenient due to their cycle-free nature. 
As highlighted in \cite{tunccel2024closeness}, trees play a signicant role in computer science and data analytics, often serving as a means to break down problems involving intricate graphs. In the context of MRFs, they enable exact belief propagation algorithms and hence provide an approximation method for general graph-based MRFs (see \cite{wainwright2003tree}).  
Tree structures arise naturally in some physical settings (e.g., blood vessels, river systems) but are also often uncovered in seemingly unstructured interactions. As a result, 
the literature on probabilistic graphical models has seen much emphasis put on the recovery of these underlying trees from high dimensional data (\cite{bresler2020learning}, \cite{ chow1968approximating}, \cite{hue2021structure}, \cite{nikolakakis2021predictive}). 

\bigskip

The shape of a tree-based network may be dictated by its physical canvas or arise without forethought
; it may also, more importantly, be the object of a critical business decision (e.g.,~the organizational structure of the company), engineering design, 
or marketing strategy. In all cases, involved parties may want to assess their network's strengths and vulnerabilities arising from its shape.  Which agent has a key role in the organization? Which is better positioned to influence others? Is all the pressure put on too few agents? How well does information -- or a virus -- propagate through the network? Would another tree shape enhance or limit this propagation? While these questions have interested researchers in graph theory and social network studies, they seem to have been overlooked in a stochastic framework, wherein network-based phenomena are random in nature, few examples include \cite{la2017cascading} and \cite{ansari2024comparison}. Let us subsume these interrogations under the two following questions in a MRF setting: 
\begin{align*}\tag{Q1}\label{eq:Q1}
    &\text{Which of the two vertex positions grants more influence on the MRF's aggregate behavior?} \\
    \tag{Q2} \label{eq:Q2}
    &\text{Which of the two shapes of trees yields a riskier aggregate distribution?}
\end{align*} 
 The purpose of this paper is to kindle dependence modelers' interest in such graph-theoretic questions and answer (\ref{eq:Q1}) and (\ref{eq:Q2}) for a specific family of MRFs.  
We also believe these results may be used as a stepping stone to extending the discussion to other MRF families. 

\bigskip

For most MRF families, answering these two questions is nearly impossible as dependence parameters and marginals are intertwined with the tree shape. The marginal distribution of any given component of the MRF, as well as its dependencies with other components, varies depending on the vertex's position within the tree. This interconnection
thus befogs the answer to (\ref{eq:Q1}). Similarly, altering the shape of their underlying tree changes the marginal distributions, which in turn impacts the dependencies and complicates the response to (\ref{eq:Q2}). Consequently, MRFs with these limitations do not allow the investigation of these questions. 
\cite{cote2025tree} puts forth a new construction paradigm for MRFs, fixing marginals and correlations through its parameterization. They introduce a new family of tree-structured MRFs under this paradigm, whose collection of distributions we denote by $\mathbb{MPMRF}$.
 This family will be the vehicle of our discussions, leveraging its fixed parameterization to answer (\ref{eq:Q1}) and (\ref{eq:Q2}).
Our interest in $\mathbb{MPMRF}$ does not solely lie in its nice mathematical properties. The stochastic behavior of MRFs from this family is of relevance in itself. It combines spontaneous events at any vertex and propagation of thus events along edges; it may also be represented in terms of conjunct and idiosyncratic shocks (see Section~2 of \cite{cossette2024risk}). As such behaviors are most prominent for network-based phenomena (e.g., viruses on a computer grid, gossip among a community), this analysis for $\mathbb{MPMRF}$ has the potential to reach out of its scope. 
  
\bigskip

We examine the joint distribution of a given component of a MRF and the sum of all its components to address (\ref{eq:Q1}).
We use the term \textit{synecdochic} as a qualifier to refer to this specific pair of random variables, as we recall a synecdoche establishes a semantic link between a part and its whole. We compare synecdochic pairs using stochastic orders to establish which components of the MRF are more contributing to the sum than others, to answer (\ref{eq:Q1}). These comparisons will invite the notion of a vertex's centrality in a network, a concept notably studied in social sciences (see \cite{borgatti2009network, schoch2015positional} for contextualization; seminal works include \cite{borgatti2005centrality}, \cite{borgatti2024analyzing}, \cite{freeman1979centrality} and \cite{friedkin1991theoretical}). This paper's analysis contributes to transfer frameworks and manners of thought specific to network centrality theory to the legibility of probabilistic graphical models. Our conception of centrality not only accounts for the tree topology, but also the underlying stochastic dynamics. 
We furthermore connect this conception to quantities involved in risk allocation problems, which thus become centrality indices tailored to our stochastic framework. 

 \bigskip

In regard to (\ref{eq:Q2}), we consider MRFs defined on trees of different shapes and investigate how their aggregate distributions' riskiness may be compared using the convex order. Remark~1 of \cite{cote2025tree} states that no supermodular order may be established between members of $\mathbb{MPMRF}$ when underlying trees have different shapes. However, thanks to synecdochic pairs' ordering results derived to answer (\ref{eq:Q1}), it is still possible to establish convex orderings for distributions of sums. This allows new perspectives in the study of MRFs, which we explore for $\mathbb{MPMRF}$ through the design of a new tree-comparing partial order. We then derive multiple results with the goal of facilitating comparisons with this new partial order; notably, we trace Hasse diagrams offering an exhaustive look at relations between trees of $d=\{4,5,6,7,8,9\}$ vertices, and provide means of comparing trees of higher dimensions.
We also discuss the new partial order's connections to other existing posets, and, in that vein, its relation to spectral graph theory.  

\bigskip

We recall some results on $\mathbb{MPMRF}$, in Section~\ref{sect:family}, before addressing (\ref{eq:Q1}) in Section~\ref{subsect:OrderAlloc} and (\ref{eq:Q2}) in Section~\ref{sect:SHA}. In Section~\ref{sect:SHA}, we also discuss how (\ref{eq:Q1}) and (\ref{eq:Q2}) are connected to each other.

 \section{Tree-structured MRFs with Poisson marginal distributions}
	\label{sect:family}
	
 In this section, we provide a summary of key results derived in \cite{cote2025tree} pertaining to the family of distributions $\mathbb{MPMRF}$. We begin with some definitions and notation.

\bigskip

    Consider a simple undirected graph $\mathcal{G}=(\mathcal{V},\mathcal{E})$ where $\mathcal{V}=\{1,2,\ldots, d\}$, $d\in\mathbb{N}_1=\mathbb{N}\backslash\{0\}$, is a set of vertices and $\mathcal{E}\subseteq \mathcal{V}\times\mathcal{V}$ a set of edges.
    For a graph $\mathcal{G}$, a path from a vertex $u$ to a vertex $v$, denoted by $\mathrm{path}(u,v)$, is a set of edges such that $u$ and $v$ participates in only one edge, while every other involved vertex participates in an even number of edges. We say these vertices are \textit{on} the path. A tree, denoted by $\mathcal{T}$, is a graph where one and only one path exists for any pair of vertices. 
    For a tree $\mathcal{T}$, one may choose a vertex $r\in\mathcal{V}$ to be the \textit{root}; the rooted version of the tree is then denoted by $\mathcal{T}_r$. Choosing a root allows to refer to vertices on a tree according to their relative positions to one another, as the following filial relations become well defined. For a rooted tree $\mathcal{T}_r$, the descendants of vertex $v\in\mathcal{V}$, denoted by $\mathrm{dsc}(v)$, is the set of vertices for which $v$ is on their path to the root, $\mathrm{dsc}(v) = \left\{u \in \mathcal{V}:  \exists w\in \mathcal{V}, (v,w) \in \mathrm{path}(u,r) \right\}$, $v\in\mathcal{V}$. The children of $v$, denoted by $\mathrm{ch}(v)$, is the set of its descendants also participating in an edge with it,  $\mathrm{ch}(v) = \left\{ j \in \mathrm{dsc}(v): (v,j) \in \mathcal{E}\right\}$, $v\in\mathcal{V}$.
    The parent of $v$, denoted by $\mathrm{pa}(v)$, is the sole vertex participating with it in an edge that is not its children; it is incidentally the first vertex on its path to the root. The root $r$ has no parent. We recommend Chapter~3 of \cite{saoub2021graph} for an introduction to rooted trees and filial relations between vertices.  

 \bigskip

Following the definition in Section 4.2 of \cite{cressie2015statistics}, a vector of random variables $\boldsymbol{X} = (X_v,\,v\in\mathcal{V})$ is a MRF defined on a tree $\mathcal{T} = (\mathcal{V},\mathcal{E})$ if it satisfies the local Markov property, that is, for any two of its components, say $X_{u}$ and $X_{w}$, such that $(u,w) \not\in \mathcal{E}$, 
		\begin{equation*}
			X_{u}\perp \!\!\!\perp X_{w}\left| \left\{X_{j},\,(u,j)\in \mathcal{E}\right\}\right.,
		\end{equation*}
where $\perp\!\!\!\perp$ denotes conditional independence. 	

\bigskip

     A stochastic representation relying on the binomial thinning operator, denoted by $\circ$, introduced in \cite{steutel1983integer}, characterizes the family of distributions $\mathbb{MPMRF}$. For a count random variable $N$ and a thinning parameter $\alpha\in(0,1)$, binomial thinning is the operation $\alpha\circ N = \sum_{i=1}^N I^{(\alpha)}_i$, where $\{I^{(\alpha)},\, i\in\mathbb{N}_1\}$ is a sequence of independent Bernoulli random variables taking 1 with probability $\alpha$, and with the convention $\sum_{i=1}^0 x_i = 0$. We hereby transcribe Theorem~1 of \cite{cote2025tree} providing the stochastic representation.
    	
	\begin{theorem}[Stochastic representation]
		\label{th:StoDynamics}
		Given a tree $\mathcal{T}=(\mathcal{V},\mathcal{E})$ and a chosen root $r\in\mathcal{V}$, let $\mathcal{T}_r$ be its rooted version. Consider a vector of dependence parameters $\boldsymbol{\alpha} = (\alpha_e, e \in \mathcal{E})$, $\boldsymbol{\alpha}\in(0,1)^{d-1}$ and define $\boldsymbol{L} = (L_v, \, v \in \mathcal{V})$ as a vector of independent random variables such that $L_v \sim$~Poisson$(\lambda(1-\alpha_{(\mathrm{pa}(v),v)}))$, $\lambda >0$, for $v\in\mathcal{V}$, with the convention $\alpha_{(\mathrm{pa}(r),r)} = 0$ since the root has no parent. Let $\boldsymbol{N} = (N_v, \, v \in \mathcal{V})$ be a vector of count random variables whose components are defined by
		\begin{equation}
			N_v = 
            \begin{cases}
                L_r, & v=r \\
                L_v + \alpha_{(\mathrm{pa}(v),v)} \circ N_{\mathrm{pa}(v)}, & v \in \mathrm{dsc}(r)            
            \end{cases}
            , \quad \text{for every } v\in\mathcal{V}.
			\label{eq:StoDynamics}
		\end{equation} 
		Then, $\boldsymbol{N}$ is a tree-structured MRF with Poisson marginal distributions of parameter $\lambda$. 
        Throughout the paper, we take the convention of writing $\boldsymbol{N} \sim \text{MPMRF}(\lambda,\boldsymbol{\alpha},\mathcal{T})$. 
	\end{theorem}

The stochastic representation in (\ref{eq:StoDynamics}) always leads to a valid multivariate Poisson distribution. Moreover,
as stated in Theorem 2 of \cite{cote2025tree}, the choice of the root $r\in\mathcal{V}$ in (\ref{eq:StoDynamics}) has no impact on the joint distribution of $\boldsymbol{N}$. Choosing a root simply offers convenience for notation, proofs and programming purposes. The joint probability mass function (pmf) of $\boldsymbol{N}$ factorizes over cliques of the tree and hence the MRF also 
satisfies the global Markov property, that is 
	\begin{equation*}
		N_{u}\perp \!\!\!\perp N_{w}\left| \left\{N_{j},\, j\in S(u,w)\right\}\right.,\quad u,w\in\mathcal{V},
	\end{equation*}
where $S(u,w)$ is a separator for $u$ and $w$, a set of vertices such that every path from $u$ to $v$ has at least one participating vertex from that set. On a tree, $S(u,w)$ is any vertex on $\mathrm{path}(u,w)$. Chapter 3 of \cite{lauritzen1996graphical} provides a discussion on the Markov properties for probabilistic graphical models. 

\bigskip
We hereby transcribe the expression of the joint probability generating function (pgf) of $\boldsymbol{N}$, which will be useful throughout the paper. For notation purposes, let $v\mathrm{dsc}(v) = \{v\}\cup\mathrm{dsc}(v) \subseteq \mathcal{V}$ and let $\boldsymbol{t}_{v\mathrm{dsc}(v)}= (t_j,\, j \in v\mathrm{dsc}(v))$, for $v\in\mathcal{V}$.
Assuming that $\boldsymbol{N} \sim \text{MPMRF}(\boldsymbol{{\lambda}},\boldsymbol{\alpha},\mathcal{T})$ with $\mathcal{T} = (\mathcal{V},\mathcal{E})$, 
    and, for a chosen root $r\in\mathcal{V}$, letting $\mathcal{T}_r$ be the rooted version of $\mathcal{T}$,  the joint pgf of ${\boldsymbol{N}}$ is given by
	\begin{equation}
		\mathcal{P}_{\boldsymbol{N}}({\boldsymbol{t}}) = \mathrm{E}\left[\prod_{v\in\mathcal{V}} t_v^{N_v}  \right] = \prod_{v \in \mathcal{V}} \mathrm{e}^{\lambda(1-\alpha_{(\mathrm{pa}(v),v)})\left(\eta^{\mathcal{T}_r}_v(\boldsymbol{t}_{v\mathrm{dsc}(v)})-1\right)},
		\quad \boldsymbol{t} = (t_v,\,v\in\mathcal{V}) \in [-1,1]^d, \label{eq:JointPGF}
	\end{equation}
         where  $\eta^{\mathcal{T}_r}_v$, for every $v\in\mathcal{V}$, is a joint pgf recursively defined as such:
        \begin{equation}
\eta^{\mathcal{T}_r}_v(\boldsymbol{t}_{v\mathrm{dsc}(v)})= t_v \prod_{j \in \mathrm{ch}(v)}\left(1 - \alpha_{(v,j)} + \alpha_{(v,j)} 	\eta^{\mathcal{T}_r}_j(\boldsymbol{t}_{j\mathrm{dsc}(j)})\right), \label{eq:jointpgf-h}
	    \end{equation}
        with $\eta_v^{\mathcal{T}_r}(\boldsymbol{t}_{v\mathrm{dsc}(v)})=t_v$ when $\mathrm{ch}(v) = \emptyset$. The filial relations dictating the recursiveness of $\eta_v^{\mathcal{T}_r}$ are taken according to the rooted tree in superscript. 

    \bigskip

Throughout, define the random variables $M$ as the sum of the components of $\boldsymbol{N}$, that is, $M = \sum_{v\in\mathcal{V}} N_v$. 
Consider $\boldsymbol{N} \sim \text{MPMRF}(\lambda,\boldsymbol{\alpha},\mathcal{T})$ with $\mathcal{T} = (\mathcal{V},\mathcal{E})$ and 
$\boldsymbol{N}^{\prime} \sim \text{MPMRF}(\lambda,\boldsymbol{\alpha}^{\prime},\mathcal{T}^{\prime})$ with $\mathcal{T}^{\prime} = (\mathcal{V},\mathcal{E}^{\prime})$, 
with identical dependence parameter for every edge, meaning
$\alpha_e = \alpha$ for every $e \in \mathcal{E}$ and
$\alpha_e^{\prime} = \alpha$ for every $e \in \mathcal{E}^{\prime}$, $\alpha\in(0,1)$. Suppose
$\mathcal{E} \neq \mathcal{E}'$.
Hence, the only difference in the respective distributions of $\boldsymbol{N}$ and $\boldsymbol{N}^{\prime}$ comes from the distinct topologies (shapes) of trees $\mathcal{T}$ and
$\mathcal{T}'$. We are interested in the distributions of $M$ and $M^{\prime}$.
In Section~7 of \cite{cote2025tree}, the authors observe through numerical examples that the underlying tree itself seems to have an impact on the distribution of the sum of the MRFs.
The next challenge is to find a criterion related to the shapes of the trees $\mathcal{T}$ and $\mathcal{T}'$ allowing to compare $M$ and $M^{\prime}$ within the rigorous framework provided by the theory of stochastic orders. 
Before doing so, we need to investigate the contribution of the component $N_v$ and the contribution of the component $N_w$ on the distribution of $M$ with regards to their positions within the tree $\mathcal{T}$; this is done in the next section.

\section{Centrality: a stochastic dynamics perspective}
\label{subsect:OrderAlloc}

To answer the question (\ref{eq:Q1}), we examine the joint distribution of a component of the random vector $\boldsymbol{N}$ and its sum, that is, for a given vertex $v\in\mathcal{V}$, the joint distribution of the synecdochic pair $(N_v,M)$. 
Let us preface this section with an example, which requires the following definition, to motivate this examination.

\begin{deff}[Supermodular Order]
	Two vectors of random variables, $\boldsymbol{X}$ and $\boldsymbol{X}^{\prime}$, are ordered according to the \textit{supermodular order}, denoted by $\boldsymbol{X} \preceq_{sm} \boldsymbol{X}^{\prime}$, if 
	$\mathrm{E}[\zeta(\boldsymbol{X})] \leq \mathrm{E}[\zeta(\boldsymbol{X}^{\prime})]$
	for every supermodular function $\zeta$, given the expectations exist. A supermodular function $\zeta:\mathbb{R}^{d}\mapsto\mathbb{R}$ is such that 
	$\zeta(\boldsymbol{x}) + \zeta(\boldsymbol{x}^{\prime}) \leq  	\zeta(\boldsymbol{x}\wedge\boldsymbol{x}^{\prime}) + \zeta(\boldsymbol{x}\vee\boldsymbol{x}^{\prime})$, 
	holds for all $\boldsymbol{x},\boldsymbol{x}^{\prime}\in\mathbb{R}^d$, with $\wedge$ denoting the componentwise minimum and $\vee$ the componentwise maximum. 
\end{deff}
Recall that, in the bivariate case, Definition~\ref{def:RiskContributing} amounts to $(X_1,X_2)\preceq_{sm}(X_1^{\prime}, X_2^{\prime})$ if $F_{X_1,X_2}(x_1,x_2) \leq F_{X_1^{\prime},X_2^{\prime}}(x_1,x_2)$, for all $(x_1,x_2)\in\mathbb{R}^2$, that is, the supermodular order corresponds to the concordance order. We prefer to employ $\preceq_{sm}$ as it will be necessary for upcoming proofs. 
We recommend Chapters~1 and~2 of \cite{muller2002comparison} and Chapters~3 and~9 of \cite{shaked2007} for an introduction to these stochastic orders and their applications.

\begin{ex}
\label{ex:MotivationNvM}
    Assume $\mathcal{T}=(\mathcal{V},\mathcal{E})$ is the 10-vertex star tree depicted in Figure~\ref{fig:Motivationexample} and consider $\boldsymbol{N}\sim \text{MPMRF}(1, \boldsymbol{\alpha}, \mathcal{T})$, with $\boldsymbol{\alpha} = (0.5\,\boldsymbol{1}_{|\mathcal{E}|})$, where $\boldsymbol{1}_{k}$ is a $k$-long vector of 1's, meaning every edge has the same parameter. The specific shape $\mathcal{T}$ makes the value taken by $N_1$ have a direct influence on the values of all other $N_v$, $v\in\mathcal{V}$. This influence is reflected in the distribution of $M$. 
    As a result, the covariance between $M$ and $N_1$ is greater than for other components: we have $\mathrm{Cov}(N_1,M) = 5.5$, whereas $\mathrm{Cov}(N_i,M) = 3.5$, $i = \{2,\ldots,10\}$. Plotting the joint cdfs of $(N_1,M)$ and $(N_i,M)$, $i = \{2,\ldots,10\}$, we observe that the first curve dominates the others, see Figure~\ref{fig:Motivationexample}. This suggests that the impact of a vertex's position goes beyond a simple ordering of covariances. Taking a more extensive perspective by examining the joint distribution of synecdochic pairs allows us to seize this impact more completely: the curves indeed indicate a supermodular ordering between $(N_1,M)$ and $(N_i,M)$, $i\in\{2,\ldots,10\}$.
    The objective is to provide a criterion 
    based on the vertices' positions in $\mathcal{T}$ for such an ordering. 
\end{ex}

    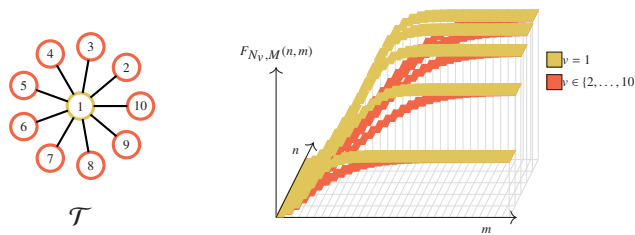
\begin{figure}[H]
\centering
\begin{tikzpicture}[every node/.style={text=Black, circle, draw = Red!75, inner sep=0mm, outer sep = 0mm, minimum size=3.5mm, fill = White, very thick}, node distance = 3mm, scale = 0.1, thick]
    \node[draw = Goldenrod!85!Black] (1) {\tiny 1};
    \node (2) at (40:8cm) {\tiny 2};
    \node (3) at (80:8cm) {\tiny 3};
    \node (4) at (120:8cm) {\tiny 4};
    \node (5) at (160:8cm) {\tiny 5};
    \node (6) at (200:8cm) {\tiny 6};
    \node (7) at (240:8cm) {\tiny 7};
    \node (8) at (280:8cm) {\tiny 8};
   \node (9) at (320:8cm) {\tiny 9};
    \node (10) at (360:8cm) {\tiny 10};
                
    \draw (1) --  (2);
    \draw (1) -- (3);
    \draw (1) -- (4);
    \draw (1) -- (5);
    \draw (1) -- (6);
    \draw (1) -- (7);
    \draw (1) -- (8);
    \draw (1) -- (9);
    \draw (1) -- (10);

    \node[draw = none, below = 11mm of 1] {$\mathcal{T}$};
    \node[draw=none, right = 27mm of 6, rectangle]{
   
   \begin{tikzpicture}[xscale = 0.1, yscale = 0.2]

\foreach \x in { 
1,2,3,4,5,6,7,8,9,10,11,12,13,14,15,16,17,18,19,20,21,22,23,24,25,26,27,28,29,30,31
 }
 \draw[draw = Black!15] (\x,0,0) -- (\x,0,-10);

\foreach \z in { 
 0,1,2,3,4,5
 }
 {\draw[draw = Black!15] (0,0,-2*\z) -- (31,0,-2*\z);
}

\foreach \z/\y in {
0/0.367835106160382, 
1/0.73398319271743,
2/0.917057235995953,
3/0.978081917088794,
4/0.993338087362004,
5/0.993338087362004
}
{
  \draw[draw = Black!15] (31,0,-2*\z) -- (31,10*\y,-2*\z);
}

\foreach \x/\z/\y in {
0/4/0.00408677143846408,
1/4/0.0224908787863064,
2/4/0.0640535239099339,
3/4/0.127302572782171,
4/4/0.201932983280819,
5/4/0.278716265075843,
6/4/0.356657938075743,
7/4/0.44105821864534,
8/4/0.535872465612035,
9/4/0.638014546039292,
10/4/0.737947104940994,
11/4/0.824889016886348,
12/4/0.89195420828704,
13/4/0.938022126136546,
14/4/0.966390847996816,
15/4/0.982163404995365,
16/4/0.990134810333901,
17/4/0.99381997960586,
18/4/0.995387239378234,
19/4/0.996003578886632,
20/4/0.996228750260273,
21/4/0.996305494575097,
22/4/0.996329989363024,
23/4/0.996337336172276,
24/4/0.996339413431646,
25/4/0.996339968696597,
26/4/0.99634010938971,
27/4/0.996340143263231,
28/4/0.996340151029836,
29/4/0.996340152729173,
30/4/0.996340153084666
}
{
\draw[draw = Black!15] (\x,0,-10) -- (\x,10*\y,-10);
}

 \draw[->] (0,0,0) -- (32,0,0)  node[very near end, below, draw = none, fill = none, rectangle]{\tiny $m$};
\draw[->] (0,0,0) -- (0,10,0) node[above, draw = none, fill = none, rectangle]{\tiny $F_{N_v,M}(n,m)$};

\foreach \m/\n/\Fmn in {
0/0/0.00408677143846407,
1/0/0.0204418391677861,
2/0/0.0532318021759856,
3/0/0.0973669610876397,
4/0/0.142897645906775,
5/0/0.182681551299164,
6/0/0.215362360894578,
7/0/0.243008335079891,
8/0/0.267579159139589,
9/0/0.289436847509258,
10/0/0.307980054776971,
11/0/0.322797322492913,
12/0/0.334151953349543,
13/0/0.342746708014556,
14/0/0.349298728248209,
15/0/0.354318144847109,
16/0/0.35812393444783,
17/0/0.36094576224552,
18/0/0.362990368403408,
19/0/0.36445010619157,
20/0/0.365486839115278,
21/0/0.366222186975098,
22/0/0.366741611515268,
23/0/0.367105064740486,
24/0/0.367356264551142,
25/0/0.367527967172871,
26/0/0.367644463033814,
27/0/0.367723156708904,
28/0/0.367776110883981,
29/0/0.367811553336145,
30/0/0.367835106160382,
0/1/0.00408677143846407,
1/1/0.0204498211432518,
2/1/0.0533356146264278,
3/1/0.0980144833336249,
4/1/0.145491109386335,
5/1/0.190231102161631,
6/1/0.232546488713531,
7/1/0.275240145030142,
8/1/0.319595366354666,
9/1/0.364477170786332,
10/1/0.408048977954552,
11/1/0.449433396065742,
12/1/0.48877068273785,
13/1/0.526240654325494,
14/1/0.561373071730904,
15/1/0.59321871100168,
16/1/0.620973439183155,
17/1/0.644400336565999,
18/1/0.663797123307139,
19/1/0.679701550658207,
20/1/0.692634671649384,
21/1/0.703016141355326,
22/1/0.711200983468717,
23/1/0.717531846427209,
24/1/0.722350589952665,
25/1/0.725976597450808,
26/1/0.728682487482826,
27/1/0.730685617432813,
28/1/0.732154334495827,
29/1/0.733219368827665,
30/1/0.73398319271743,
0/2/0.00408677143846407,
1/2/0.0204538121309847,
2/2/0.0533875208516489,
3/2/0.0983382444566176,
4/2/0.146787841126115,
5/2/0.194005877592865,
6/2/0.241138552623007,
7/2/0.291356050005267,
8/2/0.345603469962204,
9/2/0.401997332424869,
10/2/0.458083439543343,
11/2/0.512751432852157,
12/2/0.566080047432003,
13/2/0.617987627480964,
14/2/0.667410243472251,
15/2/0.712668994078965,
16/2/0.752398191550817,
17/2/0.786127623726238,
18/2/0.814200500759004,
19/2/0.837327272891526,
20/2/0.856208587916437,
21/2/0.871413118545439,
22/2/0.883430669445441,
23/2/0.89274523727057,
24/2/0.899847752653427,
25/2/0.905200912589776,
26/2/0.909201499707332,
27/2/0.912166847794768,
28/2/0.91434344630175,
29/2/0.915923276573425,
30/2/0.917057235995953,
0/3/0.00408677143846407,
1/3/0.020455142460229,
2/3/0.0534048229267226,
3/3/0.0984461648309484,
4/3/0.147220085039375,
5/3/0.195264136069943,
6/3/0.244002573926165,
7/3/0.296728018330309,
8/3/0.354272837831384,
9/3/0.414504052971048,
10/3/0.474761593406273,
11/3/0.533857445114295,
12/3/0.591849835663387,
13/3/0.64856995186612,
14/3/0.702755967386033,
15/3/0.752485755104727,
16/3/0.796206442340038,
17/3/0.833370052779651,
18/3/0.864334959909626,
19/3/0.889869180302632,
20/3/0.910733226672121,
21/3/0.927545444275477,
22/3/0.940840564771016,
23/3/0.951149700885024,
24/3/0.95901347355368,
25/3/0.964942350969432,
26/3/0.969374503782167,
27/3/0.972660591248753,
28/3/0.975073150237058,
29/3/0.976824579155345,
30/3/0.978081917088794,
0/4/0.00408677143846407,
1/4/0.0204554750425401,
2/4/0.0534091484454911,
3/4/0.0984731449245311,
4/4/0.14732814601769,
5/4/0.195578700689212,
6/4/0.244718579251955,
7/4/0.298071010411569,
8/4/0.356440179798678,
9/4/0.417630733107593,
10/4/0.478931131872005,
11/4/0.539133948179829,
12/4/0.598292282721233,
13/4/0.656215532962409,
14/4/0.711592398364479,
15/4/0.762439945361168,
16/4/0.807158505037343,
17/4/0.845180660043004,
18/4/0.876868574697281,
19/4/0.903004657155408,
20/4/0.924364386361042,
21/4/0.941578525707987,
22/4/0.955193038602409,
23/4/0.965750816788637,
24/4/0.973804903778744,
25/4/0.979877710564346,
26/4/0.984417754800876,
27/4/0.98778402711225,
28/4/0.990255576220885,
29/4/0.992049904800825,
30/4/0.993338087362004
}
{
\filldraw[draw = Red!75, fill = Red!75] (\m,10*\Fmn,-2*\n) -- (\m+1,10*\Fmn,-2*\n) -- (\m+1,10*\Fmn,-2*\n-2) -- (\m,10*\Fmn,-2*\n-2);
}; 
        \foreach \m/\n/\Fmn in {
0/4/0.00408677143846408,
1/4/0.0224908787863064,
2/4/0.0640535239099339,
3/4/0.127302572782171,
4/4/0.201932983280819,
5/4/0.278716265075843,
6/4/0.356657938075743,
7/4/0.44105821864534,
8/4/0.535872465612035,
9/4/0.638014546039292,
10/4/0.737947104940994,
11/4/0.824889016886348,
12/4/0.89195420828704,
13/4/0.938022126136546,
14/4/0.966390847996816,
15/4/0.982163404995365,
16/4/0.990134810333901,
17/4/0.99381997960586,
18/4/0.995387239378234,
19/4/0.996003578886632,
20/4/0.996228750260273,
21/4/0.996305494575097,
22/4/0.996329989363024,
23/4/0.996337336172276,
24/4/0.996339413431646,
25/4/0.996339968696597,
26/4/0.99634010938971,
27/4/0.996340143263231,
28/4/0.996340151029836,
29/4/0.996340152729173,
30/4/0.996340153084666,
0/3/0.00408677143846408,
1/3/0.0224905462039953,
2/3/0.0640487014664234,
3/3/0.127268940395964,
4/3/0.201782177989141,
5/3/0.278225469721774,
6/3/0.355415705988205,
7/3/0.438496458638595,
8/3/0.531413859220338,
9/3/0.631272517505269,
10/3/0.728861228571961,
11/3/0.813720869027514,
12/3/0.879164672658038,
13/3/0.924113949580084,
14/3/0.951792347874434,
15/3/0.96718068730466,
16/3/0.974957802825698,
17/3/0.978553125576956,
18/3/0.98008216841539,
19/3/0.980683477364101,
20/3/0.98090315722805,
21/3/0.98097802983211,
22/3/0.981001927207006,
23/3/0.981009094829863,
24/3/0.981011121425133,
25/3/0.981011663147174,
26/3/0.981011800408772,
27/3/0.981011833456114,
28/3/0.98101184103329,
29/3/0.981011842691179,
30/3/0.981011843038002,
0/2/0.00408677143846408,
1/2/0.022489215874751,
2/2/0.0640294116923811,
3/2/0.127134410851135,
4/2/0.201178956822433,
5/2/0.276262288305496,
6/2/0.350446777638056,
7/2/0.428249418611616,
8/2/0.513579433653553,
9/2/0.604304403369174,
10/2/0.692517723095831,
11/2/0.769048277592178,
12/2/0.82800653014203,
13/2/0.868481243354234,
14/2/0.89339834738491,
15/2/0.90724981654184,
16/2/0.914249772792886,
17/2/0.917485709461341,
18/2/0.918861884564012,
19/2/0.919403071273977,
20/2/0.919600785099159,
21/2/0.919668170860162,
22/2/0.919689678582935,
23/2/0.919696129460209,
24/2/0.919697953399083,
25/2/0.919698440949484,
26/2/0.919698564485019,
27/2/0.919698594227643,
28/2/0.919698601047105,
29/2/0.919698602539205,
30/2/0.919698602851346,
0/1/0.00408677143846408,
1/1/0.0224852248870181,
2/1/0.0639715423702545,
3/1/0.126730822216648,
4/1/0.199369293322307,
5/1/0.270372744056664,
6/1/0.335539992587607,
7/1/0.397508298530679,
8/1/0.460076156953198,
9/1/0.523400060960892,
10/1/0.583487206667439,
11/1/0.635030503286169,
12/1/0.674532102594008,
13/1/0.701583124676687,
14/1/0.718216345916335,
15/1/0.727457204253381,
16/1/0.732125682694448,
17/1/0.734283461114497,
18/1/0.73520103300988,
19/1/0.735561853003605,
20/1/0.735693668712486,
21/1/0.735738593944318,
22/1/0.735752932710723,
23/1/0.735757233351246,
24/1/0.735758449320934,
25/1/0.735758774356414,
26/1/0.735758856713763,
27/1/0.735758876542233,
28/1/0.735758881088549,
29/1/0.735758882083284,
30/1/0.735758882291378,
 0/0/0.00408677143846407,
1/0/0.0224772429115524,
2/0/0.0638558037260011,
3/0/0.125923644947674,
4/0/0.195749966322056,
5/0/0.258593655559,
6/0/0.305726422486708,
7/0/0.336026058368806,
8/0/0.353069603552486,
9/0/0.361591376144326,
10/0/0.365426173810654,
11/0/0.366994954674152,
12/0/0.367583247497964,
13/0/0.367786887321591,
14/0/0.367852342979185,
15/0/0.367871979676464,
16/0/0.367877502497573,
17/0/0.367878964420808,
18/0/0.367879329901617,
19/0/0.367879416462861,
20/0/0.367879435939141,
21/0/0.36787944011263,
22/0/0.367879440966298,
23/0/0.36787944113332,
24/0/0.367879441164636,
25/0/0.367879441170273,
26/0/0.367879441171249,
27/0/0.367879441171411,
28/0/0.367879441171438,
29/0/0.367879441171442,
30/0/0.367879441171442
        }
{
\filldraw[draw = Goldenrod!85!black, fill = Goldenrod!85!black] (\m,10*\Fmn,-2*\n) -- (\m+1,10*\Fmn,-2*\n) -- (\m+1,10*\Fmn,-2*\n-2) -- (\m,10*\Fmn,-2*\n-2) -- (\m,10*\Fmn,-2*\n);
};              
\draw[->] (0,0,0) -- (0,0,-13) node[very near end, left, draw = none, rectangle, fill = none]{\tiny $n$}; 
\filldraw[draw = Black, fill = Goldenrod!85!black] (36, 10) -- (36, 11) -- (38, 11) -- (38,10) node[midway, right, draw = none, fill = none]{\tiny $v = 1$} -- (36,10) ; 
\filldraw[draw = Black, fill = Red!75] (36, 8.5) -- (36, 9.5) -- (38, 9.5) -- (38,8.5) node[midway, right, draw = none, fill = none]{\tiny $v \in \{2,\ldots,10\}$} -- (36,8.5) ; 

    \end{tikzpicture}
 };
 \end{tikzpicture}
        \caption{(Left) Tree $\mathcal{T}$ from Example~\ref{ex:MotivationNvM}. (Right) Joint cdfs of the random vectors $(N_v,M)$, $v\in\mathcal{V}=\{1,\ldots,10\}$. } 
        \label{fig:Motivationexample}
    \end{figure}
    
The following definition describes the synecdochic pair's ordering observed in Example~\ref{ex:MotivationNvM}. It clarifies what is intended by \textit{influence on the overall dynamics} in (\ref{eq:Q1}). 

\begin{deff}
\label{def:RiskContributing}
    Consider the vector of random variables $\boldsymbol{X} = (X_v, v \in \mathcal{V})$ and the sum of its components $S = \sum_{v\in\mathcal{V}} X_v$. For $u,w\in\mathcal{V}$, 
    we say that $X_u$ is \textit{less contributing} than $X_{w}$ to $S$ if $(X_v,S) \preceq_{sm} (X_{w},S)$. 
\end{deff}

Hence, in respect to Definition~\ref{def:RiskContributing}, 
we aim to compare the synecdochic pairs for distinct vertices $v,w\in\mathcal{V}$ with regards to their positions within the tree $\mathcal{T}$. The key result of this section is given in Theorem~\ref{th:OrderingAlloc}, in which we provide the conditions to establish such comparisons through stochastic ordering. We then discuss how orderings of synecdochic pairs affect common quantities pertaining to risk allocation, and how this relates to the theory of network centrality. 

\bigskip

As discussed in Section~3 of \cite{cote2025tree}, the joint pgf of $\boldsymbol{N}$, given in (\ref{eq:JointPGF}), highlights an alternative stochastic representation, which proves to be essential to the analysis of the joint distribution of synecdochic pairs. We formalize this result in Lemma~\ref{th:GandH} for easier referencing later on. 

\begin{lem}
\label{th:GandH}
Given a tree $\mathcal{T}$, assume $\boldsymbol{N}=(N_v,\, v\in\mathcal{V})\sim \text{MPMRF}(\lambda,\boldsymbol{\alpha}, \mathcal{T})$ and, for a chosen root $r\in\mathcal{V}$, let $\mathcal{T}_r$ be the rooted version of $\mathcal{T}$. For $v\in\mathcal{V}$, we define $\boldsymbol{G}_v^{\mathcal{T}} = (G_{v,j}^{\mathcal{T}_r}, j\in v\mathrm{dsc}(v))$ as a vector of discrete random variables whose joint pgf is $\eta_{v}^{\mathcal{T}_r}$ as given by (\ref{eq:jointpgf-h}). Let $H_v^{\mathcal{T}_r} = \sum_{j\in v\mathrm{dsc}(v)} G_{v,j}^{\mathcal{T}_r}$; hence, its joint pgf is given by $\eta_{v}^{\mathcal{T}_r}(t\vecun{v})$. Then, $\boldsymbol{N}$ and $M$ admit the following alternative stochastic representations:  
\begin{equation}
    \boldsymbol{N} = \sum_{v\in\mathcal{V}}\sum_{k=1}^{L_v} \boldsymbol{G}_{v,(k)}^{\mathcal{T}_r}, \quad \text{and} \quad   M = \sum_{v\in\mathcal{V}}\sum_{k=1}^{L_v} H_v^{\mathcal{T}_r},
    \label{eq:alternateconstructM}
\end{equation}
 with sums taken componentwise according to the subscript $j$, and $\{\boldsymbol{G}_{v,(k)}^{\mathcal{T}_r}, k\in\mathbb{N}_1\}$ being a sequence of independent vectors of random variables having the same joint distribution as $\boldsymbol{G}_{v}^{\mathcal{T}_r}$, $v\in\mathcal{V}$, and independent of $\boldsymbol{L}$ defined as in Theorem~\ref{th:StoDynamics}.
\end{lem}
\begin{proof}
  See equation (19) of \cite{cote2025tree} and comments therein. 
\end{proof}

As for the sequences of joint pgfs $\eta_{v}^{\mathcal{T}_r}$, $r,v\in\mathcal{V}$ in (\ref{eq:jointpgf-h}), 
we specify in superscript the rooted tree on which $\boldsymbol{G}_v^{\mathcal{T}_r}$ and $H_v^{\mathcal{T}_r}$ define themselves, for the sake of clarity. A component $G_{v,j}^{\mathcal{T}_r}$, $j\in v\mathrm{dsc}(v)$, of $\boldsymbol{G}_v^{\mathcal{T}_r}$, $v,r\in\mathcal{T}_r$,  should be interpreted as such: it is the number of events having taken place at vertex $v$ that have propagated to vertex $j$ downward the tree $\mathcal{T}$ according to a rooting in $r$. 
The stochastic representation of $\boldsymbol{N}$ in (\ref{eq:alternateconstructM}) then amounts to counting events after having classified them according to their origin in terms of propagation. Similarly, the random variable $H_v^{\mathcal{T}_r}$ should be interpreted as the total number of events, occurring anywhere on the tree, originating from $v$ according to a rooting in $r$. It thus summarizes the impact vertex $v$ has on the other vertices down the tree, or if $v=r$, to all vertices in $\mathcal{V}$. 
Therefore, its distribution is the key to comparing synecdochic pairs. 
This elicits the criterion presented in the following theorem, 
characterizing, for the family $\mathbb{MPMRF}$, the statement that a component contributes less than another to the sum $M$, as per Definition~\ref{def:RiskContributing}.

\begin{theorem}
	\label{th:OrderingAlloc}
Given a tree $\mathcal{T}$, assume $\boldsymbol{N}=(N_v,\,v\in\mathcal{V}) \sim \text{MPMRF}({\lambda},\boldsymbol{\alpha},\mathcal{T})$ with $\lambda>0$ and $\boldsymbol{\alpha}\in(0,1)^{d-1}$, and let $M=\sum_{v\in\mathcal{V}}N_v$.   
    We define the two discrete random variables $H_v^{\mathcal{T}_v}$ and $H_w^{\mathcal{T}_w}$ as in Lemma~\ref{th:GandH} with pgfs given by $ \mathcal{P}_{H_v^{\mathcal{T}_v}}(t) = \eta_{v}^{\mathcal{T}_v}(t\vecun{v})$ and $ \mathcal{P}_{H_w^{\mathcal{T}_w}}(t) = \eta_{w}^{\mathcal{T}_w}(t\vecun{w})$ as provided in \eqref{eq:jointpgf-h}, $v,w\in\mathcal{V}$, $t\in[-1,1]$.
    If 
    \begin{equation} \label{eq:Criterion}
        H_v^{\mathcal{T}_v}\preceq_{st} H_w^{\mathcal{T}_w},    
    \end{equation}
where $\preceq_{st}$ denotes the usual stochastic order, then $(N_v,M) \preceq_{sm} (N_w,M)$, meaning that $N_v$ is less contributing than $N_w$ to $M$ within the MRF $\boldsymbol{N}$.
\end{theorem}
\begin{proof}
	The proof is provided in \ref{sect:proofConvexOrderMShape}. 
\end{proof}

	Theorem~\ref{th:OrderingAlloc} states that, despite having identical marginal distributions, some random variables constituting the MRF $\boldsymbol{N}$ have more key roles in the distribution of $M$ -- are more contributing to $M$, as per Definition~\ref{def:RiskContributing} -- than others. As discussed in Chapter~9 of \cite{shaked2007}, the supermodular order compares the strength of positive dependence within random vectors. The observations we have made in Example~\ref{ex:MotivationNvM} are then confirmed by this result. Let us provide a further example to illustrate the result of Theorem~\ref{th:OrderingAlloc}.

\begin{figure}[H]
\centering
	\begin{tikzpicture}[every node/.style={text=Black, circle, draw = Blue!50!black, inner sep=0.2mm,   minimum size=3mm, fill = White, very thick}, node distance = 1mm, scale=0.1, thick]
					
					\node (1b) {\tiny 1};
					\node [right=of 1b] (2b) {\tiny 2};
					\node [right=of 2b] (3b) {\tiny 3};
					\node [above =of 3b] (4b) {\tiny 4};
					\node [below =of 3b] (5b) {\tiny 5};
					\node [right =of 3b] (6b) {\tiny 6};
					
					\draw (1b) --  (2b);
					\draw (2b) -- (3b);
					\draw (3b) -- (4b);
					\draw (3b) -- (5b);            
					\draw (3b) -- (6b);

					\node[draw=none, below = of 5b](t2){$\mathcal{T}^{\mathcolor{White}{\prime}}$};
					\end{tikzpicture}
					\caption{Trees $\mathcal{T}$ from Example~\ref{ex:exCentralite-1}.}
					\label{fig:M-centralite}
\end{figure}

\begin{ex}
\label{ex:exCentralite-1}
 Consider $\boldsymbol{N}$ defined on tree $\mathcal{T}$, depicted in Figure~\ref{fig:M-centralite}, and with dependence parameters $\boldsymbol{\alpha}$ such that $\alpha_e = \alpha$ for every $e\in\mathcal{E}$. We aim to show that $N_2$ is less contributing than $N_3$ to $M$. For that matter, from Theorem~\ref{th:OrderingAlloc}, it suffices to show $H_2^{\mathcal{T}_2} \preceq_{st} H_3^{\mathcal{T}_3}$. We have 
\begin{align*}
    \eta_2^{\mathcal{T}_2}(t) &= t \; (1-\alpha + \alpha t) \; (1-\alpha + \alpha t (1-\alpha+\alpha t)^3); \\
    \eta_3^{\mathcal{T}_3}(t) &= t \; (1-\alpha + \alpha t)^3 \; (1-\alpha + \alpha t (1-\alpha+\alpha t)), \quad \text{for every }t\in[-1,1]. 
\end{align*}
From these pgfs, it is easily seen that indeed $H_2^{\mathcal{T}_2} \preceq_{st} H_3^{\mathcal{T}_3}$. Moreover, note that, clearly, $N_4$, $N_5$ and $N_6$ are equally contributing to $M$ since $\eta_4^{\mathcal{T}_4} = \eta_5^{\mathcal{T}_5} = \eta_6^{\mathcal{T}_6}$. 
\end{ex}

It may seem challenging to apply the result of Theorem \ref{th:OrderingAlloc} because one needs to verify the condition $H_v^{\mathcal{T}_v}\preceq_{st} H_w^{\mathcal{T}_w}$. However, thanks to the recursive nature of the sequences of pgfs $\{\eta_j^{\mathcal{T}_r},\,j\in\mathcal{V}\}$, $r\in\mathcal{V}$, provided in (\ref{eq:jointpgf-h}), it is often not necessary to develop $\mathcal{P}_{H_v^{\mathcal{T}_v}}(t)$  and $ \mathcal{P}_{H_w^{\mathcal{T}_w}}(t)$ completely to establish stochastic dominance.
Other results pertaining to the criterion in \eqref{eq:Criterion} are explored later in Section~\ref{sect:SHA}, and further examples showing first order stochastic dominance ordering between $H_v^{\mathcal{T}_v}$ and $H_w^{\mathcal{T}_w}$ are provided therein. 
	The condition $H_v^{\mathcal{T}_v}\preceq_{st}H_w^{\mathcal{T}_w}$ in \eqref{eq:Criterion} considers two different roots, $v$ and $w$. Therefore, both random variables could not be part of the same stochastic representation given by (\ref{eq:alternateconstructM}). The choice of these roots is not anecdotal: for a root $r\in\mathcal{V}$, one has $r\mathrm{dsc}(r) = \mathcal{V}$. 
 The pgf given by
	$\eta_{v}^{\mathcal{T}_v}(t\vecun{v})$, because of its recursive definition (\ref{eq:jointpgf-h}), thus encapsulates all the information needed on the shape of the tree, translating it into the distribution of $H_v^{\mathcal{T}_v}$; likewise for $H_w^{\mathcal{T}_w}$, $v,w\in\mathcal{V}$. 
	
	\bigskip
    
	 The fact that the dependence within a synecdochic pair varies in strength depending on the vertices' relative position in the tree relates to the notion of \textit{centrality} of a vertex to the graph. Heuristically, the result of Theorem~\ref{th:OrderingAlloc} suggests a vertex's centrality should be understood as such in the stochastic context provided by the family $\mathbb{MPMRF}$: events occurring at the most central vertices are those which have the most possibilities to propagate to others, and reciprocally, these vertices also have the most chance of welcoming propagated events. This interpretation comes from that of $H_v^{\mathcal{T}
 _v}$, $v\in\mathcal{V}$, discussed above, arising from Lemma~\ref{th:GandH}, and the criterion in (\ref{eq:Criterion}). Accordingly, Theorem~\ref{th:OrderingAlloc} indicates that if $M$ has taken a high value, the random variables associated with these central vertices probably also had a high outcome since they are the most determinant to $M$. In \cite{freeman1979centrality}, the author explores the notion of centrality and highlights the fuzziness of what constitutes a vertex's centrality. This is why we reason centrality following the dynamics at play on the graph, as advised and discussed in \cite{friedkin1991theoretical} and \cite{borgatti2005centrality}.
 Centrality is generally understood regarding centrality indices; a discussion in the following subsection will confirm the coherence of our heuristic interpretation of the concept for the family $\mathbb{MPMRF}$. 

\bigskip

The study of synecdochic pairs directly connects to risk allocation theory; one may consult \cite{overbeeck2000allocation}, \cite{tasche2007capital} and \cite{dhaene2012optimal} for an introduction on risk allocation concepts. We henceforth present a few relevant quantities. 
The expected allocation of $N_v$ to $M$ for a total outcome $k\in\mathbb{N}$ is given by 
\begin{equation}
    \mathrm{E}[N_v\mathbb{1}_{\{M=k\}}] = \lambda \; \eta_v^{\mathcal{T}_v}(t\boldsymbol{1}_d) \; \mathcal{P}_M(t),\quad  v\in\mathcal{V}, \quad t\in[-1,1],
    \label{eq:expealloc}
\end{equation}
see Corollary~4 of~\cite{cote2025tree}.
Expected allocations allow to perform risk management tasks such as risk-sharing and risk allocation, as they are notably employed in the computation of the conditional mean risk-sharing rule introduced in \cite{denuit2012convex}, given by
 $   \mathrm{E}[N_v|M=k] = \mathrm{E}[N_v\mathbb{1}_{\{M=k\}}]/p_M(k)$, with $k\in\mathbb{N}$ such that $p_M(k)>0$, and $p_M$ denotes the pmf of $M$.
A related quantity is the cumulative expected allocation, defined as $\mathrm{E}[N_v\mathbb{1}_{\{M\leq k\}}]$, for $k\in\mathbb{N}$. 
Expected cumulative allocations are relevant to compute the contribution to the Tail-Value-at-Risk (TVaR) under Euler's principle, as discussed in \cite{blier2022generating}. The risk measure TVaR at a confidence level $\kappa$ is given by $\mathrm{TVaR}_\kappa(X) = \tfrac{1}{1-\kappa}\int_{\kappa}^{1}\mathrm{VaR}_u(X)\,\mathrm{d}u$, with $\mathrm{VaR}_\kappa(X) = \mathrm{inf}\{x\in\mathbb{R}:F_X(x)\geq\kappa\}$, $\kappa\in[0,1)$. Since $M$ takes on values in $\mathbb{N}$, the contribution of $N_v$, $v\in\mathcal{V}$, is given by 
\begin{align}
	\mathcal{C}^{\mathrm{TVaR}}_{\kappa}(N_v;\, M) &= \frac{1}{1-\kappa}\left( \mathrm{E}[N_v\mathbb{1}_{\{M>\mathrm{VaR}_{\kappa}(M)\}}] + \frac{F_M(\mathrm{VaR}_{\kappa}(M))-\kappa}{p_M(\mathrm{VaR}_{\kappa}(M))} \mathrm{E}[N_v\mathbb{1}_{\{M=\mathrm{VaR}_{\kappa}(M)\}}]\right) \notag\\
	&= \frac{1}{1-\kappa}\left( \mathrm{E}[N_v] -  \mathrm{E}[N_v\mathbb{1}_{\{M\leq\mathrm{VaR}_{\kappa}(M)\}}] + \frac{F_M(\mathrm{VaR}_{\kappa}(M))-\kappa}{p_M(\mathrm{VaR}_{\kappa}(M))} \mathrm{E}[N_v\mathbb{1}_{\{M=\mathrm{VaR}_{\kappa}(M)\}}]\right),\notag
\end{align}
for $\kappa \in [0,1)$; see Section~2 in \cite{mausser2018long}. One easily verifies that $\sum_{v\in\mathcal{V}} \mathcal{C}^{\mathrm{TVaR}}_{\kappa}(N_v;\, M) = \mathrm{TVaR}_{\kappa}(M)$. Contributions to TVaR represent the portion of the aggregate risk enclosed within each component of $\boldsymbol{N}$. 
On a related note, the covariance $\mathrm{Cov}(N_v,M)$ pertains to the contribution of $N_v$ to the standard deviation of $M$ under Euler's principle, and rather quantifies the portion of enclosed variability. One may consult \cite{tasche2007capital} for further insights on contributions to TVaR and to standard deviation under Euler's rule. 

 \bigskip
	
	The ordering of synecdochic pairs provided by Theorem~\ref{th:OrderingAlloc} manifests itself through an ordering of covariances, expected cumulative allocations and contributions to TVaR under Euler's rule, as a consequence of the following proposition. 
	
	\begin{prop}
		\label{th:EffectsofSupermodularity}
		If $N_v$ is less contributing than $N_w$ to $M$, with respect to Definition~\ref{def:RiskContributing},
  then the following holds:
		\begin{enumerate}
			\item[(a)] $\mathrm{Cov}(N_v,M)\leq\mathrm{Cov}(N_w,M)$;
			\item[(b)] $\mathrm{E}[N_v\mathbb{1}_{\{M\leq k\}}] \geq \mathrm{E}[N_w\mathbb{1}_{\{M\leq k\}}]$, for all $k\in\mathbb{N}$;
			\item[(c)] $\mathcal{C}^{\mathrm{TVaR}}_{\kappa}(N_v;M)\leq\mathcal{C}^{\mathrm{TVaR}}_{\kappa}(N_w;M)$, for all $\kappa\in[0,1)$.
		\end{enumerate}
	\end{prop}
	
	\begin{proof}
		The statements result from the supermodular ordering between $(N_v,M)$ and $(N_w,M)$ according to Definition~\ref{def:RiskContributing}. Theorem~3.8.2 of \cite{muller2002comparison} grants the validity of part (a). It also states that the inequality $\mathrm{E}[N_v\mathbb{1}_{\{M\geq k\}}]\leq\mathrm{E}[N_w\mathbb{1}_{\{M\geq k\}}]$ holds for every $k\in\mathbb{N}$, since $x\mapsto\mathbb{1}_{\{x\geq k\}}$ is nondecreasing. Part (b) then ensues 
		given $\mathrm{E}[N_v]=\mathrm{E}[N_w]$, and part (c) follows from Lemma~\ref{th:LemmaContrib} (see Appendix~\ref{sect:LemmaContrib}). 
  \end{proof}

		Based on the above discussion on centrality, $\mathrm{Cov}(N_v,M)$, $\mathrm{E}[N_v\mathbb{1}_{\{M\leq k\}}]$, and $\mathcal{C}^{\mathrm{TVaR}}_{\kappa}(N_v,M)$ all act as centrality indices, as they induce an ordering of vertices within a tree, coherent with our understanding of centrality 
         given Theorem~\ref{th:OrderingAlloc}. In the seminal paper \cite{freeman1979centrality}, the author, in a effort to dispel the reigning confusion around centrality, divided the notion into three concepts: degree, betweenness and closeness, and introduced the now-classical eponymous centrality indices. Freeman's closeness, in our tree-based context where only one path exists for each pair of vertices, is given by 
        \begin{equation}
        c_v^{\mathrm{clo}} = \sum_{j\in\mathcal{V}} |\mathrm{path}(v,j)|, \quad v\in\mathcal{V}.
        \label{eq:Freemancloseness}
        \end{equation}
       Having a \textit{centrality} index increase when a vertex is farther in terms of path distances is counter-intuitive, as pointed out in \cite{borgatti2006graph}. Authors have suggested functional transformations for $|\mathrm{path}(v,j)|$ in (\ref{eq:Freemancloseness}) to remedy to this issue, such as the multiplicative inverse, yielding the harmonic closeness discussed in \cite{rochat2009closeness}, or an exponential transformation. The latter was proposed in \cite{burt1991STRUCTURE}, p.~66, and converts Freeman's closeness to
        \begin{equation}
            c_v^{\text{clo,exp.transf}} = \sum_{j\in\mathcal{V}}\alpha^{|\mathrm{path}(v,j)|},\quad v\in\mathcal{V},
            \label{eq:Burtcloseness}
        \end{equation}
        for a chosen parameter $\alpha\in(0,1)$. In our context, the centrality index $\mathrm{Cov}(N_v,M)$, from Proposition~\hyperref[th:EffectsofSupermodularity]{\ref{th:EffectsofSupermodularity}(a)}, is actually the index in (\ref{eq:Burtcloseness}), up to a multiplicative constant, with $\alpha$ being the dependence parameter for the edges. Indeed, provided $\mathrm{Cov}(N_v,N_j)=\lambda\prod_{e\in\mathrm{path}(v,j)} \alpha_e$ given Theorem~5 of \cite{cote2025tree}, we have
        \begin{equation*}
        \mathrm{Cov}(N_v,M) = \sum_{j\in\mathcal{V}} \mathrm{Cov}(N_v,N_j) = \lambda\sum_{j\in\mathcal{V}} \prod_{e\in\mathrm{path}(v,j)} \alpha_e = \lambda\sum_{j\in\mathcal{V}} \alpha^{|\mathrm{path}(v,j)|} = \lambda c_v^{\text{clo,exp.transf}},
        \end{equation*}
        where the fourth equality comes from the fact that $\alpha_e=\alpha$ for every $e\in\mathcal{E}$ under the assumptions of Theorem~\ref{th:OrderingAlloc}.
        The quantity $\mathrm{Cov}(N_v,M)$ is thus a closeness index in terms of Freeman's categorization. 

\bigskip

Note that the ordering of vertices induced by index $\mathrm{Cov}(N_v,M)$ may not be the  same as that of indices $\mathrm{E}[N_v\mathbb{1}_{M\leq k}]$, as in Proposition~\hyperref[th:EffectsofSupermodularity]{\ref{th:EffectsofSupermodularity}(b)}, and $\mathcal{C}_{\kappa}^{\mathrm{TVaR}}(N_v,M)$, as in
Proposition~\hyperref[th:EffectsofSupermodularity]{\ref{th:EffectsofSupermodularity}(c)}, since the condition in (\ref{eq:Criterion}) is sometimes not satisfied for two vertices $v,w\in\mathcal{V}$. 
Because of the intricacy of their respective analytical expressions, we cannot proceed as for $\mathrm{Cov}(N_v,M)$ and connect
$\mathrm{E}[N_v\mathbb{1}_{M\leq k}]$ and $\mathcal{C}_{\kappa}^{\mathrm{TVaR}}(N_v;M)$ to other existing centrality indices.   
Nevertheless, the typology proposed in \cite{borgatti2006graph} suggests they be interpreted as closeness indices likewise. Indeed, as seen in \ref{eq:expealloc}, expected allocations indirectly takes into account the length of paths through the product of dependence parameters $\alpha$ in $\eta_v^{\mathcal{T}_v}(t\boldsymbol{1}_d)$, $t\in[-1,1]$, $v\in\mathcal{V}$. Moreover, the rooting in $v$ suggests the paths considered by expected allocation all emanate from vertex $v$ rather than flowing through it. We therefore deduce that the indices $\mathrm{E}[N_v\mathbb{1}_{\{M\leq k\}}]$ and $\mathcal{C}^{\mathrm{TVaR}_{\kappa}}(N_v;M)$ have a "length" Walk Property of and a "radial" Walk Position. 
Their cross-classification thus falls within the same category as Freeman's closeness and other closeness indices (see Table~1 of \cite{borgatti2006graph}). 

\bigskip

One could want to verify if the three centrality indices put forth in Proposition~\ref{th:EffectsofSupermodularity} agree with axiomatic systems on centrality indices, such as \cite{sabidussi1966centrality}'s and \cite{ruhnau2000eigenvector}'s.  
However, these approaches were deemed reductive, as discussed earlier, by the authors of \cite{freeman1979centrality}, \cite{friedkin1991theoretical} and \cite{borgatti2005centrality}: one needs to take into account the underlying dynamics occurring on the graph. There are no universal centrality indices. As a matter of fact, axiomatic systems seem to consistently overlook closeness as a concept of centrality; one may refer to Table~3.2 in \cite{schoch2015positional} and the discussion thereabout. One should rather ask if the centrality indices are suitable to the context of the family $\mathbb{MPMRF}$. According to \cite{borgatti2005centrality}, the adequacy of a centrality index entails the flow of traffic -- events, in our context -- on the graph. The propagation dynamics unveiled by the stochastic representation in Theorem~\ref{th:StoDynamics} gives an account of this flow of events. Events may propagate to multiple, but not necessarily all, neighbours at a time. Events follow paths, not walks, as they cannot propagate through the same edge back and forth. $\mathbb{MPMRF}$ thus induces a \textit{serial duplication, paths} flow of traffic. The conclusions of \cite{borgatti2005centrality}, summarized in Table~2 therein, suggest that closeness indices be preferred for contexts involving this type of flow on a tree. (Note that paths, geodesics and trails are the same on a tree.) 
Hence, from this reasoning, centrality indices from Proposition~\ref{th:EffectsofSupermodularity} would agree with the stochastic dynamics of $\mathbb{MPMRF}$.

\section{Tree-shape partial order}
\label{sect:SHA}
 Toward answering (\ref{eq:Q2}), forthcoming Theorem~\ref{th:ConvexOrderMShape} exhibits how convex ordering may be established for $M$'s defined on underlying trees with different shapes. After, we devise a partial order induced by this result, allowing to compare trees on the basis of their shapes. Theorem~\ref{th:ConvexOrderMShape} is therefore the cornerstone of this section. 	
\begin{deff}[Convex order]
	Two random variables $X$ and $X^{\prime}$ are said to be ordered according to the convex order, denoted by $X\preceq_{cx}X^{\prime}$, if 
	$ \mathrm{E}[\varphi(X)] \leq \mathrm{E}[\varphi(X^{\prime})]$ 
	for every convex function $\varphi$, given the expectations exist. 
\end{deff}        
Let us also recall the notion of \textit{subtree} of a tree. Precisely, a subtree $\tau$ of a tree $\mathcal{T} = (\mathcal{V},\mathcal{E})$ is defined by the pair $(\mathcal{V}^{\tau},\mathcal{E}^{\tau})$, where $\mathcal{V}^{\tau}$ is a subset of vertices, $\mathcal{V}^{\tau}\subseteq\mathcal{V}$, and $\mathcal{E}^{\tau}$ is a subset of edges, $\mathcal{E}^{\tau}\subseteq\mathcal{E}$, such that $\tau$ is a tree itself.

		\begin{theorem}
			\label{th:ConvexOrderMShape}
            Consider two trees $\mathcal{T}=(\mathcal{V},\mathcal{E})$ and $\mathcal{T}^{\prime}=(\mathcal{V},\mathcal{E}^{\prime})$
            such that $\mathcal{E}\backslash\{(u,v)\} = \mathcal{E}^{\prime}\backslash\{(u,w)\}$, that is the only difference between $\mathcal{T}$ and $\mathcal{T}^{\prime}$ is that $u$ is connected to $w$ in  $\mathcal{T}^{\prime}$ rather than to $v$ as in $\mathcal{T}$. Suppose $\boldsymbol{N}\sim$ MPMRF$(\lambda,\boldsymbol{\alpha},\mathcal{T})$ and $\boldsymbol{N}^{\prime}\sim$ MPMRF$(\lambda,\boldsymbol{\alpha}^{\prime},\mathcal{T}^{\prime})$, with $\lambda>0$ and $\boldsymbol{\alpha},\boldsymbol{\alpha}^{\prime}\in[0,1]^{d-1}$ such that $\alpha_e = \alpha_e^{\prime}$ for every $e \in \mathcal{E}\backslash\{(u,v)\}$ and $\alpha_{(u,v)} = \alpha_{(u,w)}^{\prime}$. Let $M$ and $M^{\prime}$ be sums of the respective components of $\boldsymbol{N}$ and $\boldsymbol{N}^{\prime}$.
 Consider the subtree $\tau^{\dagger}=(\mathcal{V^{\dagger}}=\{u\}\cup\mathrm{dsc}(u), \,\mathcal{E}^{\dagger}=\{(i,j)\in\mathcal{E}:i,j\in\mathcal{V}^{\dagger})$, with descendants of $u$ taken according to rooting $\mathcal{T}_v$ or, equivalently, $\mathcal{T}^{\prime}_w$. Pruning $\mathcal{T}$ or $\mathcal{T}^{\prime}$ of $\tau^{\dagger}$ yields the same residual subtree, which we denote ${\tau}^*$, that is, $\tau^*= (\mathcal{V^*}=\mathcal{V}\backslash\mathcal{V}^{\dagger},\,\mathcal{E}^*=\mathcal{E}\backslash(\{(u,v)\}\cup\mathcal{E}^{\dagger})=\mathcal{E}^{\prime}\backslash(\{(u,w)\}\cup\mathcal{E}^{\dagger}))$. 
   We define the two discrete random variables $H_v^{\tau^*_v}$ and $H_w^{\tau^*_w}$ as in Lemma~\ref{th:GandH} with pgfs given by $ \mathcal{P}_{H_v^{\tau^*_v}}(t) = \eta_{v}^{\tau^*_v}(t\vecun{v})$ and $ \mathcal{P}_{H_w^{\tau^*_w}}(t) = \eta_{w}^{\tau^*_w}(t\vecun{w})$ as provided in \eqref{eq:jointpgf-h}, $v,w\in\mathcal{V}$, $t\in[-1,1]$.
 If $H_v^{\tau^*_v}\preceq_{st} H_w^{\tau^*_w}$, then
			$ M \preceq_{cx} M^{\prime}$.
		\end{theorem}
		\begin{proof}
	The proof is provided in Appendix \ref{sect:proofconvexMshape}.
		\end{proof}

		Figure~\ref{fig:ExTheo9} provides an example of two trees $\mathcal{T}$ and $\mathcal{T}^{\prime}$ such that $\mathcal{E}\backslash\{(u,v)\} = \mathcal{E}^{\prime}\backslash\{(u,w)\}$, as considered by Theorem~\ref{th:ConvexOrderMShape}. For each, we colored the subtree $\tau^{\dagger}$ in yellow (light shade) and the subtree $\tau^{*}$ in blue (dark shade).

		\begin{figure}[H]
			\centering
			\begin{tikzpicture}[every node/.style={text=Black, circle, draw = Blue!50!black, inner sep=0.1mm,   minimum size=2.5mm, fill = White, very thick}, node distance = 1mm, scale=0.1, thick]
				\node (1a) {};
				\node [above left=of 1a] (2a) {};
				\node [below left=of 1a] (3a) {};
				\node [right =of 1a] (4a) {\tiny $v$};
				\node [right =of 4a] (5a) {};
				\node [right =of 5a] (6a) {\tiny $w$};
				\node [above left =of 6a] (7a) {};
				\node [above right =of 6a] (8a) {};
				\node [right =of 6a] (9a) {};
				\node [above =of 6a] (95a) {};
				\node [draw =Goldenrod!85!black, below=of 4a] (10a) {\tiny $u$};
				\node [draw =Goldenrod!85!black, below left=of 10a] (11a) {};
				\node [draw =Goldenrod!85!black, below right=of 10a] (12a) {};
				\node [draw =Goldenrod!85!black, right=of 12a] (13a) {};
				
				\draw (1a) -- (2a);
				\draw (1a) -- (3a);
				\draw (1a) -- (4a);
				\draw (4a) -- (5a);            
				\draw (5a) -- (6a);       
				\draw (6a) -- (7a);
				\draw (6a) -- (8a);
				\draw (6a) -- (9a);
				\draw (6a) -- (95a);
				\draw (4a) -- (10a);
				\draw (10a) -- (11a);
				\draw (10a) -- (12a);
				\draw (12a) -- (13a);
				
				\node [right = 10mm of 9a](1b) {};
				\node [above left=of 1b] (2b) {};
				\node [below left=of 1b] (3b) {};
				\node [right =of 1b] (4b) {\tiny $v$};
				\node [right =of 4b] (5b) {};
				\node [right =of 5b] (6b) {\tiny $w$};
				\node [above left =of 6b] (7b) {};
				\node [above right =of 6b] (8b) {};
				\node [right =of 6b] (9b) {};
				\node [above =of 6b] (95b) {};
				\node [draw =Goldenrod!85!black, below=of 6b] (10b) {\tiny $u$};
				\node [draw =Goldenrod!85!black, below left=of 10b] (11b) {};
				\node [draw =Goldenrod!85!black, below right=of 10b] (12b) {};
				\node [draw =Goldenrod!85!black, right=of 12b] (13b) {};
				
				\draw (1b) -- (2b);
				\draw (1b) -- (3b);
				\draw (1b) -- (4b);
				\draw (4b) -- (5b);            
				\draw (5b) -- (6b);       
				\draw (6b) -- (7b);
				\draw (6b) -- (8b);
				\draw (6b) -- (9b);
				\draw (6b) -- (95b);
				\draw (6b) -- (10b);
				\draw (10b) -- (11b);
				\draw (10b) -- (12b);
				\draw (12b) -- (13b);     
				
				\node[draw=none, below = 8mm of 5a](t1){$\mathcal{T}^{\mathcolor{White}{\prime}}$};
				\node[draw=none, below= 8mm of 5b](t2){$\mathcal{T}^{\prime}$};
			\end{tikzpicture}
			\caption{Two trees differing only by the anchoring position of vertex $u$}
			\label{fig:ExTheo9}
		\end{figure}

		Let $\preceq_{sha}$ denote the tree-shape partial order on the set $\Omega_{\mathcal{T}}$ of all trees defined such that $\mathcal{T}\preceq_{sha}\mathcal{T}^{\prime}$ if $M\preceq_{cx} M^{\prime}$ according to Theorem \ref{th:ConvexOrderMShape}, where the respective distributions of $M$ and $M^{\prime}$ are defined on $\mathcal{T}$ and $\mathcal{T}^{\prime}$, and have the same vector of dependence parameters $\boldsymbol{\alpha}$ such that $\alpha_e = \alpha$, for every $e\in\mathcal{E}\cup\mathcal{E}^{\prime}$, $\alpha\in(0,1)$. Recall that a partial order consists of a homogeneous binary relation that is reflexive, antisymmetric and transitive. The relation $\preceq_{sha}$ directly inherits reflexivity and transitivity from $\preceq_{cx}$; for antisymmetry, one needs to suppose that no two tree shapes yield the same distribution of $M$. We conjecture this is the case, as even the cospectral, so-called \textit{topological twin} trees presented in \cite{mckay1977spectral} render different distributions of $M$. Else, one could work with the quotient space given by the equivalence relation on trees yielding the same distribution of $M$, which we would also denote by $\Omega_{\mathcal{T}}$ for convenience. Therefore, the order $\preceq_{sha}$ satisfying the three properties of partial orders on $\Omega_{\mathcal{T}}$ makes $(\Omega_{\mathcal{T}}, \preceq_{sha})$ a partially ordered set (poset); one may consult Chapter~2 of \cite{nguyen2018first} for an introduction to posets. 
		
		\bigskip
		
		The poset $(\Omega_{\mathcal{T}}, \preceq_{sha})$ offers a framework to compare, based on the shape of their dependence structure, MRFs of the proposed family that are on the same basis regarding their magnitude and the strength of their dependence. Theorem~\ref{th:ConvexOrderMShape} induces $\preceq_{sha}$. Indeed, it is this result that enables the ordering of trees according to their shape. Leaning on the transitivity of partial orders, the following corollary exhibits that trees differing by more than one edge may be compared by iterative applications of Theorem~\ref{th:ConvexOrderMShape}.

		\begin{cor}
			\label{th:Star-to-Series}
			Let $\{\mathcal{T}^{<k>},\, k \in \{1,...,d-2\} \}\subseteq \Omega_{\mathcal{T}}$ be a sequence of trees resulting from the iterative deconstruction of a star to a series tree, as depicted in Figure \ref{fig:Star-to-Series}, and assume $d\geq4$. Hence, $\mathcal{T}^{<1>}$ is a $d$-vertex star tree and $\mathcal{T}^{<d-2>}$ is a $d$-vertex series tree. We have the following ordering of trees:
			\begin{equation*}
				\mathcal{T}^{<d-2>}\preceq_{sha} \mathcal{T}^{<d-3>} \preceq_{sha} \cdots \preceq_{sha} \mathcal{T}^{<2>} \preceq_{sha} \mathcal{T}^{<1>}.
			\end{equation*}
		\end{cor}

		\begin{proof}
			We process in an iterative manner, meaning we establish the order relation at each step as we go from the star to the series tree, changing one edge at a time, as in Figure~\ref{fig:Star-to-Series}.
			Without loss of generality, we work with the vertices labelled as in Figure~\ref{fig:Star-to-Series}.

			\begin{figure}[H]
				\centering
				\begin{tikzpicture}[every node/.style={text = black, draw=none, fill = none, text opacity = 1, rectangle, inner sep=0mm, outer sep= 2mm}, node distance = 5mm, scale=0.1, very thick]
					\node(a0){
						\begin{tikzpicture}[every node/.style={circle, draw = Blue!50!black, inner sep=0.2mm,   minimum size=2.5mm, fill = White, very thick, outer sep = 0.1mm}, node distance = 1.5mm, scale=0.1, thick]
							\node(1){\tiny 1};
							\node[right = of 1](2){\tiny 2};
							\node[above right = of 1](3){\tiny 3};
							\node[above = of 1](4){\tiny 4};
							\node[above left = of 1](5){\tiny 5};
							\node[left = of 1](6){\tiny 6};
							\node[below left = of 1](7){\tiny 7};
							\node[draw = White, below = of 1](8){...};
							\node[below right = of 1](9){\tiny $d$};
							
							\draw (1) --  (2);
							\draw (1) -- (3);
							\draw (1) -- (4);
							\draw (1) -- (5);
							\draw (1) -- (6);
							\draw (1) -- (7);
							\draw (1) -- (8);
							\draw (1) -- (9);
						\end{tikzpicture}
					};
					\node[below = 0mm of a0]{\footnotesize{Star ($\mathcal{T}^{<1>}$)}};
					\node[right= of a0](a1){
						\begin{tikzpicture}[every node/.style={circle, draw = Blue!50!black, inner sep=0.2mm,   minimum size=2.5mm, fill = White, very thick, outer sep = 0.1mm}, node distance = 1.5mm, scale=0.1, thick]
							\node(1){\tiny 1};
							\node[right = of 1](2){\tiny 2};
							\node[right = of 2](3){\tiny 3};
							\node[above = of 1](4){\tiny 4};
							\node[above left = of 1](5){\tiny 5};
							\node[left = of 1](6){\tiny 6};
							\node[below left = of 1](7){\tiny 7};
							\node[draw = White, below = of 1](8){...};
							\node[below right = of 1](9){\tiny $d$};
							
							\draw (1) --  (2);
							\draw (2) -- (3);
							\draw (1) -- (4);
							\draw (1) -- (5);
							\draw (1) -- (6);
							\draw (1) -- (7);
							\draw (1) -- (8);
							\draw (1) -- (9);
						\end{tikzpicture}
					};
					\node[below = 0mm of a1]{\footnotesize{$\mathcal{T}^{<2>}$}};
					\node[right= of a1](a2){
						\begin{tikzpicture}[every node/.style={circle, draw = Blue!50!black, inner sep=0.2mm,   minimum size=2.5mm, fill = White, very thick, outer sep = 0.1mm}, node distance = 1.5mm, scale=0.1, thick]
							\node(1){\tiny 1};
							\node[White, above = of 1](phantom){};
							\node[right = of 1](2){\tiny 2};
							\node[right = of 2](3){\tiny 3};
							\node[right = of 3](4){\tiny 4};
							\node[above left = of 1](5){\tiny 5};
							\node[left = of 1](6){\tiny 6};
							\node[below left = of 1](7){\tiny 7};
							\node[draw = White, below = of 1](8){...};
							\node[below right = of 1](9){\tiny $d$};
							
							\draw (1) --  (2);
							\draw (2) -- (3);
							\draw (3) -- (4);
							\draw (1) -- (5);
							\draw (1) -- (6);
							\draw (1) -- (7);
							\draw (1) -- (8);
							\draw (1) -- (9);
						\end{tikzpicture}
					};
					\node[below = 0mm of a2]{\footnotesize{$\mathcal{T}^{<3>}$}};
					\node[right= of a2](dots){\dots
					};
					\node[right= of dots](a3){
						\begin{tikzpicture}[every node/.style={circle, draw = Blue!50!black, inner sep=0.2mm,   minimum size=2.5mm, fill = White, very thick, outer sep = 0.1mm}, node distance = 1.5mm, scale=0.1, thick]
							\node(1){\tiny 1};
							\node[White, above = of 1](phantom){};
							\node[White, below = of 1](phantom2){};
							\node[right = of 1](2){\tiny 2};
							\node[draw = White, right = of 2](3){...};
							\node[rounded rectangle, right = of 3](4){\tiny $d$-2};
							\node[rounded rectangle, below left = of 1](8){\tiny $d$-1};
							\node[below right = of 1](9){\tiny $d$};
							
							\draw (1) --  (2);
							\draw (2) -- (3);
							\draw (3) -- (4);
							\draw (1) -- (8);
							\draw (1) -- (9);
						\end{tikzpicture}
					};
					\node[below = 0mm of a3]{\footnotesize{$\mathcal{T}^{<d-3>}$}};
					\node[right= of a3](a4){
						\begin{tikzpicture}[every node/.style={circle, draw = Blue!50!black, inner sep=0.2mm,   minimum size=2.5mm, fill = White, very thick, outer sep = 0.1mm}, node distance = 1.5mm, scale=0.1, thick]
							\node(1){\tiny 1};
							\node[White, above = of 1](phantom){};
							\node[White, below = of 1](phantom2){};
							\node[draw = White, right = of 1](3){...};
							\node[rounded rectangle, right = of 3](4){\tiny $d$-2};
							\node[rounded rectangle, right = of 4](8){\tiny $d$-1};
							\node[left = of 1](9){\tiny $d$};
							
							\draw (1) --  (2);
							\draw (2) -- (3);
							\draw (3) -- (4);
							\draw (4) -- (8);
							\draw (1) -- (9);
						\end{tikzpicture}
					};
					\node[below = 0.5mm of a4]{\footnotesize{Series tree ($\mathcal{T}^{<d-2>}$)}};
					\draw[very thick] (a0) -- (a1);	
					\draw[very thick] (a1) -- (a2);	
					\draw[very thick] (a2) -- (dots);	
					\draw[very thick] (dots) -- (a3);	
					\draw[very thick] (a3) -- (a4);	
				\end{tikzpicture}
				\caption{Iterative deconstruction of a star to a series tree.}
				\label{fig:Star-to-Series}
			\end{figure}
			
			We first look at $\mathcal{T}^{<1>}$ and $\mathcal{T}^{<2>}$. They differ only from the anchorage of vertex $3$, connected to $1$ in the first case and to $2$ in the second. In order to apply Theorem~\ref{th:ConvexOrderMShape}, we establish a stochastic dominance relation between $H_{1}^{\tau^*_1}$ and $H_{2}^{\tau^*_2}$, defined on subtree $\tau^*$ with vertices $\mathcal{V}^*=\mathcal{V}\backslash\{3\}$, in a context where all dependence parameters are equal to $\alpha$, $\alpha\in(0,1)$. 
            Beforehand, we define $\psi$ by 
	           \begin{equation}
    	       \psi(y, t,\alpha,\chi) = t(1 - \alpha + \alpha y)^{\chi},
	           \label{eq:psi}
    	       \end{equation}
	        with $\psi^{\{k\}}$ denoting the $k$th composition of $\psi$ on $y$, $k\in \mathbb{N}_1$, and $\psi^{\{0\}}(y, t,\alpha,\chi)=t$.
            Also, for Bernoulli random variables $\{I_j,\, j\in\mathbb{N}_1\}$ and random variables $X_1$ and $X_2$, note that $X_1+X_2\times\prod_{j=1}^{k}I_j\preceq_{st}X_1+X_2$ for any $k\in\mathbb{N}_1$. Therefore, for random variables $X_3$ and $X_4$ with respective pgfs given by $\mathcal{P}_{X_3}(t) = \psi^{\{k\}}(t\mathcal{P}_{X_2}(t),t,\alpha,1)
			$ and
			$\mathcal{P}_{X_4}(t) = \psi^{\{k\}}(t,t,\alpha,1)\times\mathcal{P}_{X_2}(t)$, $k\in\mathbb{N}_1$, we establish
			\begin{equation}
				X_3\preceq_{st} X_4 
				\label{eq:psiXstoX}
			\end{equation} 
			by identification of the pgfs. Given (\ref{eq:jointpgf-h}), we have 
				\begin{align}
     &\eta^{\tau^*_1}_1(t\vecun{1})  = t(1-\alpha + \alpha t)^{d-2} = \psi(t,t,\alpha,1) \times (1-\alpha + \alpha t)^{d-3}; 
					\label{eq:proofStar-to-Series-1}\\
     &\eta^{\tau^*_2}_2(t\vecun{2}) = t(1-\alpha + \alpha t(1-\alpha+\alpha t)^{d-3}) =  \psi(t(1-\alpha + \alpha t)^{d-3}, t, \alpha, 1),
					\label{eq:proofStar-to-Series-2}
				\end{align}
				for $t\in[-1,1]$. Therefore, as in (\ref{eq:psiXstoX}), it follows that $H_{2}^{\tau^*_2}\preceq_{st} H_{1}^{\tau^*_1}$, leading to $M^{<2>}\preceq_{cx}M^{<1>}$ by Theorem \ref{th:ConvexOrderMShape} allowing to conclude $\mathcal{T}^{<2>}\preceq_{sha}\mathcal{T}^{<1>}$. Shifting the focus to $\mathcal{T}^{<2>}$ and $\mathcal{T}^{<3>}$, we establish a stochastic dominance relation between $H_{1}^{\tau^{**}_1}$ and $H_{3}^{\tau^{**}_3}$, with subtree $\tau^{**}$ having vertices $\mathcal{V}^{**}=\mathcal{V}\backslash\{4\}$, in order to apply Theorem~\ref{th:ConvexOrderMShape}. Indeed, the two structures only differ by the anchoring position of vertex $4$. From (\ref{eq:jointpgf-h}), we obtain
				\begin{align}
&\eta^{\tau^{**}_1}_1(t\vecun{1})  = t(1-\alpha + \alpha t)^{d-4}(1-\alpha t(1-\alpha t)) = (1-\alpha + \alpha t)^{d-4}\psi^{\{2\}}(t,t,\alpha,1);
					\label{eq:proofStar-to-Series-3}\\
&\eta^{\tau^{**}_3}_3(t\vecun{3}) = t(1-\alpha + \alpha t( 1-\alpha t((1-\alpha + \alpha t)^{d-4})))  = \psi^{\{2\}}(t(1-\alpha + \alpha t)^{d-4},t,\alpha,1),
					\label{eq:proofStar-to-Series-4}
				\end{align}
				for $t\in[-1,1]$. Given (\ref{eq:psiXstoX}), we draw the same conclusions as above and establish $H_{3}^{\tau^{**}_3}\preceq_{st}H_{1}^{\tau^{**}_1}$, allowing to conclude that $M^{<3>}\preceq_{cx}M^{<2>}$ by Theorem \ref{th:ConvexOrderMShape}. Hence, $\mathcal{T}^{<3>}\preceq_{sha}\mathcal{T}^{<2>}$.
				We repeat these steps until the star has been entirely rearranged as a series tree. The comparing pattern of (\ref{eq:proofStar-to-Series-1}) against (\ref{eq:proofStar-to-Series-2}), and of (\ref{eq:proofStar-to-Series-3}) against (\ref{eq:proofStar-to-Series-4}), indeed reissues at each step, due to the recursiveness of $\eta_{v}^{\mathcal{T}_v}$. 
			\end{proof}

					In Appendix \ref{sect:ShapesBonanza}, we provide Hasse diagrams for posets $(\Omega_{\mathcal{T}}^{(d)}, \preceq_{sha})$, where $\Omega_{T}^{(d)}$ is the finite set of all trees comprising $d$ vertices, $d\in\{4,\ldots,9\}$. The diagrams were devised from multiple applications of Theorem~\ref{th:ConvexOrderMShape}, considering every combination of trees. Two trees may not be comparable if they either differ by the positions of more than one vertex or if the condition $H_v^{\tau^*_v}\preceq_{st} H_w^{\tau^*_w}$ is not satisfied. 
		Notice that the Hasse diagram is not upward plannar starting with $d=8$. This coincides with the first occurences of shapes not being comparable because the condition $H_v^{\tau^*_v}\preceq_{st} H_w^{\tau^*_w}$ is not met.  Also, $(\Omega_{\mathcal{T}}, \preceq_{sha})$ is not a lattice: some pair of trees have more than one supremum, and others more than one infimum, as observed in the Hasse diagram of $(\Omega_{\mathcal{T}}^{(9)},\preceq_{sha})$. The star acts as maximal element and the series tree as minimal element of the subset of $d$-vertex trees, $d\in\{4,\ldots,9\}$. 

			\bigskip
			
			The criterion $H_v^{\tau^*_v}\preceq_{st} H_w^{\tau^*_w}$ cannot be omitted  in Theorem~\ref{th:ConvexOrderMShape}: even if two trees differ by the anchoring position of only one vertex, they may not be comparable under $\preceq_{sha}$. This is illustrated in Example \ref{ex:NumEx-MnotConvex}.

			\begin{ex}\label{ex:NumEx-MnotConvex} 
				Consider $\mathcal{T}$ and $\mathcal{T}^{\prime}$, the two trees depicted in Figure \ref{fig:CounterEx}. They differ only by the anchorage of vertex $4$: $\mathcal{E}\backslash\{(2,4)\} = \mathcal{E}^{\prime}\backslash\{(3,4)\}$. Suppose $\lambda=1$ and $\alpha_e = 0.9$ for every $e \in \mathcal{E}\cup\mathcal{E}^{\prime}$. Subtree $\tau^*$ is obtained by pruning $4$ of either trees.  
				Theorem \ref{th:ConvexOrderMShape} cannot be applied as $H_{2}^{\tau^*_2}$ and $H_{3}^{\tau^*_3}$ are not comparable according to the stochastic dominance order. Indeed, one finds 
				$\mathrm{Pr}({H_{2}^{\tau^*_2}}\leq 1) = 0.01 > 0.001 = \mathrm{Pr}({H_{3}^{\tau^*_3}}\leq 1)$ and $\mathrm{Pr}({H_{2}^{\tau^*_2}}\leq2) = 0.0190 < 0.0199 = \mathrm{Pr}({H_{3}^{\tau^*_3}}\leq2)$.
				After computing the values of the pmfs of $M$ and $M^{\prime}$, we obtain
			$		\mathrm{TVaR}_{0.23}(M) = 15.4342 < 15.4355 = \mathrm{TVaR}_{0.23}(M^{\prime}) $
				and
			$		\mathrm{TVaR}_{0.97}(M) = 42.3553 > 42.2257 = \mathrm{TVaR}_{0.97}(M^{\prime})$,	
    which disrupt the possibility of convex ordering between $M$ and $M^{\prime}$. Therefore, $\mathcal{T}$ and $\mathcal{T}^{\prime}$ are not comparable under $\preceq_{sha}$.  	
    \end{ex}

     \begin{figure}[H]
    \centering
    \centering
					\begin{tikzpicture}[every node/.style={text=Black, circle, draw = Blue!50!black, inner sep=0.2mm,   minimum size=3mm, fill = White, very thick}, node distance = 1mm, scale=0.1, thick]
						\node (1a) {\tiny $1$};
						\node [below left=of 1a] (2a) {\tiny $2$};
						\node [below right=of 1a] (3a) {\tiny $3$};
						\node [draw =Goldenrod!85!black, below left =of 2a] (4a) {\tiny $4$};
						\node [below =of 2a] (5a) {\tiny $5$};
						\node [below =of 3a] (6a) {\tiny $6$};
						\node [below right =of 3a] (7a) {\tiny $7$};
						\node [below =of 5a] (10a) {\tiny $10$};
						\node [right =of 10a] (11a) {\tiny $11$};
						\node [right=of 11a] (12a) {\tiny $12$};
						\node [left=of 10a] (9a) {\tiny $9$};
						\node [left=of 9a] (8a) {\tiny $8$};
						
						\draw (1a) -- (2a);
						\draw (1a) -- (3a);
						\draw (2a) -- (4a);
						\draw (2a) -- (5a);            
						\draw (3a) -- (6a);       
						\draw (3a) -- (7a);
						\draw (5a) -- (8a);
						\draw (5a) -- (9a);
						\draw (5a) -- (10a);
						\draw (5a) -- (11a);
						\draw (5a) -- (12a);
						
						\node [right = 24mm of 1a](1b) {\tiny $1$};
						\node [below left=of 1b] (2b) {\tiny $2$};
						\node [below right=of 1b] (3b) {\tiny $3$};
						\node [draw =Goldenrod!85!black, right =of 3b] (4b) {\tiny $4$};
						\node [below =of 2b] (5b) {\tiny $5$};
						\node [below =of 3b] (6b) {\tiny $6$};
						\node [below right =of 3b] (7b) {\tiny $7$};
						\node [below =of 5b] (10b) {\tiny $10$};
						\node [right =of 10b] (11b) {\tiny $11$};
						\node [right=of 11b] (12b) {\tiny $12$};
						\node [left=of 10b] (9b) {\tiny $9$};
						\node [left=of 9b] (8b) {\tiny $8$};
						
						\draw (1b) -- (2b);
						\draw (1b) -- (3b);
						\draw (3b) -- (4b);
						\draw (2b) -- (5b);            
						\draw (3b) -- (6b);       
						\draw (3b) -- (7b);
						\draw (5b) -- (8b);
						\draw (5b) -- (9b);
						\draw (5b) -- (10b);
						\draw (5b) -- (11b);
						\draw (5b) -- (12b);
						
						\node[draw=none, below = of 10a](t1){$\mathcal{T}^{\mathcolor{White}{\prime}}$};
						\node[draw=none, below= of 10b](t2){$\mathcal{T}^{\prime}$};
\end{tikzpicture}
\caption{Trees $\mathcal{T}$ and $\mathcal{T}^{\prime}$ of Example \ref{ex:NumEx-MnotConvex}.}
        \label{fig:CounterEx}
    \end{figure}

			Since the criterion to apply Theorem~\ref{th:ConvexOrderMShape} is the same as the one to apply Theorem~\ref{th:OrderingAlloc}, but on $\tau^*$, one ties the ordering of synecdochic pairs established above to $\preceq_{sha}$ through the following theorem. 

   \begin{theorem}
   \label{th:connection}
    Given a tree $\mathcal{T}=(\mathcal{V},\mathcal{E})$, let $\mathcal{T}^{\text{+}i}$, denote the tree obtained by connecting an additional vertex to $i$, for $i\in\mathcal{V}$. If $\mathcal{T}^{\text{+}v} \preceq_{sha} \mathcal{T}^{\text{+}w}$, for vertices $v,w\in \mathcal{V}$, then $N_v$ is less contributing than $N_w$ to $M$ with respect to Definition~\ref{def:RiskContributing}, that is to say $v$ is less central than $w$.
   \end{theorem} 

\begin{proof}
The residual subtree $\tau^*$ common to both $\mathcal{T}^{\text{+}v}$ and $\mathcal{T}^{\text{+}w}$ is $\mathcal{T}$. Hence, $\mathcal{T}^{\text{+}v} \preceq_{sha} \mathcal{T}^{\text{+}w}$ means $H_v^{\mathcal{T}_v}\preceq_{st}H_w^{\mathcal{T}_w}$, as the order $\preceq_{sha}$ was induced by Theorem~\ref{th:ConvexOrderMShape}. The result follows from Theorem~\ref{th:OrderingAlloc}.
\end{proof}

Results facilitating comparison under $\preceq_{sha}$, such as Corollary~\ref{th:Star-to-Series} and upcoming Corollaries~\ref{th:Tool1}, \ref{th:Tool2} and \ref{th:Tool3} as well as the Hasse diagrams in Appendix \ref{sect:ShapesBonanza}, thus also serve as tools in the analysis of the joint distribution of synecdochic pairs and related centrality indices. This is exhibited in the following example. 

\begin{figure}[H]
\centering
	\begin{tikzpicture}[every node/.style={text=Black, circle, draw = Blue!50!black, inner sep=0.2mm,   minimum size=3mm, fill = White, very thick}, node distance = 1mm, scale=0.1, thick]
                        \node [right = 24mm of 7b](2a) {\tiny 2};
					\node [below left=of 2a] (1a) {\tiny 1};
					\node [right=of 2a] (3a) {\tiny 3};
					\node [above =of 3a] (4a) {\tiny 4};
					\node [below =of 3a] (5a) {\tiny 5};
					\node [right =of 3a] (6a) {\tiny 6};
					\node [draw =Goldenrod!85!black, above left=of 2a] (7a) {\tiny 7};
					
					\draw (1a) --  (2a);
					\draw (2a) -- (3a);
					\draw (3a) -- (4a);
					\draw (3a) -- (5a);            
					\draw (3a) -- (6a);       
					\draw (2a) -- (7a);
					
					\node [right = 5mm of 6a](1b) {\tiny 1};
					\node [right=of 1b] (2b) {\tiny 2};
					\node [right=of 2b] (3b) {\tiny 3};
					\node [above =of 3b] (4b) {\tiny 4};
					\node [below =of 3b] (5b) {\tiny 5};
					\node [right =of 3b] (6b) {\tiny 6};
					
					\draw (1b) --  (2b);
					\draw (2b) -- (3b);
					\draw (3b) -- (4b);
					\draw (3b) -- (5b);            
					\draw (3b) -- (6b);       
					
					\node[right= 5mm of 6b] (1c) {\tiny 1};
					\node [right=of 1c] (2c) {\tiny 2};
					\node [right=of 2c] (3c) {\tiny 3};
					\node [above =of 3c] (4c) {\tiny 4};
					\node [below =of 3c] (5c) {\tiny 5};
					\node [above right=of 3c] (6c) {\tiny 6};
					\node [draw = Goldenrod!85!black, below right=of 3c] (7c) {\tiny 7};
					
					\draw (1c) --  (2c);
					\draw (2c) -- (3c);
					\draw (3c) -- (4c);
					\draw (3c) -- (5c);            
					\draw (3c) -- (6c);       
					\draw (3c) -- (7c);
					
					\node[draw=none, above = of 6a](phantom1){};
					\node[draw=none, below = of 2c](phantom2){};
					\draw[Black!50, very thick, ->] (2b) to [out = 120, in=45] (phantom1);
					\draw[Black!50, very thick, ->] (3b) to [out=315, in=225] (phantom2);
					
					\node[draw=none, below = of 5a](t1){$\mathcal{T}^{\prime}$};
					\node[draw=none, below = of 5b](t2){$\mathcal{T}^{\mathcolor{White}{\prime}}$};
					\node[draw=none, below = of 5c](t3){$\mathcal{T}^{\prime\prime}$};
					\end{tikzpicture}
					\caption{Trees $\mathcal{T}$, $\mathcal{T}^{\prime}$ and $\mathcal{T}^{\prime\prime}$ from Example~\ref{ex:exAlloc}.}
					\label{fig:MnotConvex}
				\end{figure}

			\begin{ex}
				\label{ex:exAlloc}  
				Consider $\boldsymbol{N}$ defined on tree $\mathcal{T}$, depicted in Figure~\ref{fig:MnotConvex}, and with dependence parameters $\boldsymbol{\alpha}$ such that $\alpha_e = \alpha$ for every $e\in\mathcal{E}$. We aim to show that $N_2$ is less contributing than $N_3$ to $M$, just as in Example~\ref{ex:exCentralite-1}, but this time without deriving the expressions of $\eta_2^{\mathcal{T}_2}$ and $\eta_3^{\mathcal{T}_3}$. If one were to add a vertex by connecting it to $2$, they would obtain tree $\mathcal{T}^{\prime}$; if one were to add a vertex by connecting it to $3$, they would obtain tree $\mathcal{T}^{\prime\prime}$. Now, referring to Figure~\ref{fig:ShapesBonanza2}, we see $\mathcal{T}^{\prime}\preceq_{sha}\mathcal{T}^{\prime\prime}$. Hence, returning to $\mathcal{T}$, one deduces that $N_2$ is less contributing than $N_3$ to $M$ given Theorem~\ref{th:connection}.
			\end{ex}

			We provide additional corollaries to Theorem \ref{th:ConvexOrderMShape}, which, similarly to Corollary~\ref{th:Star-to-Series}, prove useful to compare trees differing by more than one edge in accordance to $\preceq_{sha}$. Each one presents a different sequence of iterative manipulations of trees, for which successive applications of Theorem~\ref{th:ConvexOrderMShape} allow to establish the ordering along the sequence.

			\begin{cor}
				\label{th:Tool1}
				Let $\mathcal{T}^{(1)}=(\mathcal{V},\mathcal{E}^{(1)})$ be a tree such that, to a specific vertex $w\in\mathcal{V}$, are connected subtrees as well as $d_{\mathrm{ray}}$ single vertices, which we call \textit{ray-like} vertices, $d_{\mathrm{ray}}\in\mathbb{N}_1$.   
				Define $\{\mathcal{T}^{(k)},\, k\in\{2,\ldots,d_{\mathrm{ray}}-1\}\}\subseteq\Omega_{\mathcal{T}}$ as the sequence of trees obtained by successively moving ray-like vertices to form a series subtree connected to $w$. Assume $d \geq 4$. We have the following ordering of trees: 
				\begin{equation*}
					\mathcal{T}^{(d_{\mathrm{ray}}-1)}\preceq_{sha}\mathcal{T}^{(d_{\mathrm{ray}}-2)}\preceq_{sha}\cdots\preceq_{sha}\mathcal{T}^{(2)}\preceq_{sha}\mathcal{T}^{(1)}.
				\end{equation*} 
			\end{cor}
			
			\begin{proof}
				Aiming to compare $\mathcal{T}^{(k)}$ and $\mathcal{T}^{(k+1)}$, $k\in\{1,\ldots,d_{\mathrm{ray}}-2\}$, let $v$ be the vertex to which the ray-like vertex from $\mathcal{T}^{(k)}$ anchors itself in $\mathcal{T}^{(k+1)}$. We show $H_{v}^{\tau^*_v} \preceq_{st} H_{w}^{\tau^*_w}$ when dependence parameters are all $\alpha$, $\alpha \in[0,1]$, with $\tau^*$ being $\mathcal{T}^{(k)}$ pruned of that moving ray-like vertex. With $\psi$ as in (\ref{eq:psi}), we obtain from (\ref{eq:jointpgf-h})
				\begin{equation}
					\eta_{v}^{\tau^*_v}(t\vecun{v})= \psi^{\{k\}}\left(t(1-\alpha+\alpha t)^{d_{\mathrm{ray}}-k-1}\prod_{j\in\mathrm{sub}}(1-\alpha+\alpha\eta_{j}^{\tau^*_v}(t\vecun{j}), t, \alpha, 1\right),
					\label{eq:proofTool1-1}
				\end{equation}
				and
				\begin{equation}
					\eta_{w}^{\tau^*_w}(t\vecun{w}) = \psi^{\{k\}}(t, t, \alpha, 1) \times (1-\alpha+\alpha t)^{d_{\mathrm{ray}}-k-1}\prod_{j\in\mathrm{sub}}(1-\alpha+\alpha\eta_{j}^{\tau^*_w}(t\vecun{j})),
					\label{eq:proofTool1-2}
				\end{equation}
				with $\mathrm{sub}$ denoting the set of vertices by which the generic subtrees are connected to $v$. For all $j\in\mathrm{sub}$, we have $\eta_{j}^{\tau^*_v}(t\vecun{j}) = \eta_{j}^{\tau^*_w}(t\vecun{j})$, $t\in[-1,1]$, as the children of $j$ are the same under both rooting. Therefore, from (\ref{eq:psiXstoX}), (\ref{eq:proofTool1-1}) and (\ref{eq:proofTool1-2}), we deduce $H_{v}^{\tau^*_v} \preceq_{st} H_{w}^{\tau^*_w}$; hence, $\mathcal{T}^{(k+1)}\preceq_{sha}\mathcal{T}^{(k)}$, $k\in\{1,\ldots,d_{\mathrm{ray}}-2\}$.
			\end{proof}

\begin{figure}[H]
					\centering
					\begin{tikzpicture}[every node/.style={text = black, draw=none, fill = none, text opacity = 1, rectangle, inner sep=0mm, outer sep = 1.5mm}, node distance = 5mm, scale=0.1, thick]
						\node(a0){
							\begin{tikzpicture}[every node/.style={circle, draw = Blue!50!black, inner sep=0.1mm,   minimum size=2mm, fill = White, very thick, outer sep = 0.1mm}, node distance = 1mm, scale=0.1, thick]
								\node[minimum size=1mm, draw = Blue!50!black](1){\tiny $w$};
								\node[draw = Red!75, minimum size=2mm, above = of 1](2){};
								\node[draw = Red!75, minimum size=2mm, above left = of 2](3){};
								\node[draw = Red!75, minimum size=2mm, above right = of 2](4){};
								\node[draw = Red!75, minimum size=2mm, above right = of 1](5){};
								\node[draw = Red!75, minimum size=2mm, above right= of 5](6){};
								\node[draw = White, below = of 1](10){...};
								\node[draw = Goldenrod!85!black, below right = of 1](11){};
								\node[draw = Goldenrod!85!black, right = of 1](12){};
								\node[draw = Goldenrod!85!black, below left = of 1](9){};
								\node[draw = Goldenrod!85!black, above left = of 1](7){};
								\node[draw = Goldenrod!85!black, left = of 1](8){};
								
								\draw (1) --  (2);
								\draw (2) -- (3);
								\draw (2) -- (4);
								\draw (1) -- (5);
								\draw (5) -- (6);
								\draw (1) -- (7);
								\draw (1) -- (8);
								\draw (1) -- (9);
								\draw (1) -- (10);
								\draw (1) -- (11);
								\draw (1) -- (12);
							\end{tikzpicture}
						};
						\node[below = 0mm of a0, outer sep = 0mm]{$\mathcal{T}^{(1)}$};
						\node[ right= of a0](a1){
							\begin{tikzpicture}[every node/.style={circle, draw = Blue!50!black, inner sep=0.1mm,   minimum size=2mm, fill = White, very thick, outer sep = 0.1mm}, node distance = 1mm, scale=0.1, thick ]
								\node[minimum size=1mm, draw = Blue!50!black](1){\tiny $w$};
								\node[draw = Red!75, minimum size=2mm, above = of 1](2){};
								\node[draw = Red!75, minimum size=2mm, above left = of 2](3){};
								\node[draw = Red!75, minimum size=2mm, above right = of 2](4){};
								\node[draw = Red!75, minimum size=2mm, above right = of 1](5){};
								\node[draw = Red!75, minimum size=2mm, above right= of 5](6){};
								\node[draw = White, below = of 1](10){...};
								\node[draw = Goldenrod!85!black, below right = of 1](11){};
								\node[draw = Goldenrod!85!black, right = of 1](12){};
								\node[draw = Goldenrod!85!black, below left = of 1](9){};
								\node[draw = Goldenrod!85!black, right = of 12](7){};
								\node[draw = Goldenrod!85!black, left = of 1](8){};
								
								\draw (1) --  (2);
								\draw (2) -- (3);
								\draw (2) -- (4);
								\draw (1) -- (5);
								\draw (5) -- (6);
								\draw (12) -- (7);
								\draw (1) -- (8);
								\draw (1) -- (9);
								\draw (1) -- (10);
								\draw (1) -- (11);
								\draw (1) -- (12);
							\end{tikzpicture}
						};
						\node[below = 0mm of a1, outer sep = 0mm]{$\mathcal{T}^{(2)}$};
						\node[right= of a1](a2){
							\begin{tikzpicture}[every node/.style={circle, draw = Blue!50!black, inner sep=0.1mm,   minimum size=2mm, fill = White, very thick, outer sep = 0.1mm}, node distance = 1mm, scale=0.1 , thick]
								\node[minimum size=1mm, draw = Blue!50!black](1){\tiny $w$};
								\node[draw = Red!75, minimum size=2mm, above = of 1](2){};
								\node[draw = Red!75, minimum size=2mm, above left = of 2](3){};
								\node[draw = Red!75, minimum size=2mm, above right = of 2](4){};
								\node[draw = Red!75, minimum size=2mm, above right = of 1](5){};
								\node[draw = Red!75, minimum size=2mm, above right= of 5](6){};
								\node[draw = White, below = of 1](10){...};
								\node[draw = Goldenrod!85!black, below right = of 1](11){};
								\node[draw = Goldenrod!85!black, right = of 1](12){};
								\node[draw = Goldenrod!85!black, below left = of 1](9){};
								\node[draw = Goldenrod!85!black, right = of 12](7){};
								\node[draw = Goldenrod!85!black, right = of 7](8){};
								
								\draw (1) --  (2);
								\draw (2) -- (3);
								\draw (2) -- (4);
								\draw (1) -- (5);
								\draw (5) -- (6);
								\draw (12) -- (7);
								\draw (7) -- (8);
								\draw (1) -- (9);
								\draw (1) -- (10);
								\draw (1) -- (11);
								\draw (1) -- (12);
							\end{tikzpicture}
						};
						\node[ below = 0mm of a2, outer sep = 0mm]{$\mathcal{T}^{(3)}$};
						\node[ right= of a2](dots){\dots
						};
						\node[right= of dots](a3){
							\begin{tikzpicture}[every node/.style={circle, draw = Goldenrod!85!black, inner sep=0.1mm,   minimum size=2mm, fill = White, very thick, outer sep = 0.1mm}, node distance = 1mm, scale=0.1 , thick]
								\node[minimum size=1mm, draw = Blue!50!black](1){\tiny $w$};
								\node[draw = Red!75, minimum size=2mm, above = of 1](2){};
								\node[draw = Red!75, minimum size=2mm, above left = of 2](3){};
								\node[draw = Red!75, minimum size=2mm, above right = of 2](4){};
								\node[draw = Red!75, minimum size=2mm, above right = of 1](5){};
								\node[draw = Red!75, minimum size=2mm, above right = of 5](6){};
								\node[draw = White, below = of 1](phantom){};
								\node[draw = Goldenrod!85!black, right = of 1](12){};
								\node[draw = Goldenrod!85!black, right = of 12](7){};
								\node[draw = White, right = of 7](10){...};
								\node[draw = Goldenrod!85!black, right = of 10](8){};
								\node[draw = Goldenrod!85!black, below right = of 1](11){};
								
								\draw (1) --  (2);
								\draw (2) -- (3);
								\draw (2) -- (4);
								\draw (1) -- (5);
								\draw (5) -- (6);
								\draw (12) -- (7);
								\draw (1) -- (11);
								\draw (10) -- (8);
								\draw (7) -- (10);
								\draw (1) -- (12);
							\end{tikzpicture}
						};
						\node[below = 0mm of a3, outer sep = 0mm]{$\mathcal{T}^{(d_{\mathrm{ray}}-2)}$};
						\node[right= of a3](a4){
							\begin{tikzpicture}[every node/.style={circle, draw = Goldenrod!85!black, inner sep=0.1mm,   minimum size=2mm, fill = White, very thick, outer sep = 0.1mm}, node distance = 1mm, scale=0.1 , thick]
								\node[minimum size=1mm, draw = Blue!50!black](1){\tiny $w$};
								\node[draw = Red!75, minimum size=2mm, above = of 1](2){};
								\node[draw = Red!75, minimum size=2mm, above left = of 2](3){};
								\node[draw = Red!75, minimum size=2mm, above right = of 2](4){};
								\node[draw = Red!75, minimum size=2mm, above right = of 1](5){};
								\node[draw = Red!75, minimum size=2mm, above right = of 5](6){};
								\node[draw = White, below = of 1](phantom){};
								\node[draw = Goldenrod!85!black, right = of 1](12){};
								\node[draw = Goldenrod!85!black, right = of 12](7){};
								\node[draw = White, right = of 7](10){...};
								\node[draw = Goldenrod!85!black, right = of 10](8){};
								\node[draw = Goldenrod!85!black, right = of 8](9){};
								
								\draw (1) --  (2);
								\draw (2) -- (3);
								\draw (2) -- (4);
								\draw (1) -- (5);
								\draw (5) -- (6);
								\draw (12) -- (7);
								\draw (10) -- (8);
								\draw (9) -- (8);
								\draw (7) -- (10);
								\draw (1) -- (12);
							\end{tikzpicture}
						};
						\node[below = 0mm of a4, outer sep = 0mm]{$\mathcal{T}^{(d_{\mathrm{ray}}-1)}$};
						\draw[very thick] (a0) -- (a1);	
						\draw[very thick] (a1) -- (a2);	
						\draw[very thick] (a2) -- (dots);	
						\draw[very thick] (dots) -- (a3);	
						\draw[very thick] (a3) -- (a4);	
					\end{tikzpicture}
					\caption{Example of a sequence of trees encompassed by Corollary~\ref{th:Tool1}. Vertex of interest $w$ is in blue (dark shade), generic subtrees connected to it are in orange (mid shade), ray-like vertices and the series subtree they form are in yellow (light shade). }
					\label{fig:Reconstr1}
				\end{figure}

   	We depict in Figure~\ref{fig:Reconstr1} an example of the sequence $\{\mathcal{T}^{(k)},\, k\in\{1,\ldots,d_{\mathrm{ray}}-1\}\}$ induced by the reconstruction of a tree $\mathcal{T}^{(1)}$.
			Corollary~\ref{th:Tool1} should be seen as a generalization of Corollary~\ref{th:Star-to-Series}: vertex $d$ in Figure \ref{fig:Star-to-Series} is a connected subtree while vertices 2 to $d-1$ are the ray-like vertices.

			\begin{cor}
				\label{th:Tool2}
				Let $\mathcal{T}^{[1]}=(\mathcal{V},\mathcal{E}^{[1]})$ be a $d_{\mathrm{se}}$-vertex series subtree to which is connected at one end a subtree $\tau$, $d_{\mathrm{se}}\geq 3$. Let $\{l_k,\, k\in\{1,\ldots,d_{\mathrm{se}}\}\}$ label the successive vertices of the series subtree, with $l_1$ connected to $\tau$.
				Define $\{\mathcal{T}^{[k]},\, k\in\{2,\ldots,k_{\max}\}\}\subseteq\Omega_{\mathcal{T}}$, $k_{\max}=\lfloor\frac{d_{\mathrm{se}}}{2}\rfloor$, as the sequence of trees obtained by anchoring $\tau$ to $l_k$ rather than $l_1$, $k\in\{1,\ldots,k_{\max}\}$. 
				We have the following ordering of trees:   
				\begin{equation*}
					\mathcal{T}^{[1]}\preceq_{sha}\mathcal{T}^{[2]}\preceq_{sha}\cdots \preceq_{sha}\mathcal{T}^{[k_{\max}-1]}\preceq_{sha}\mathcal{T}^{[k_{\max}]}.
				\end{equation*} 
			\end{cor}
			
			\begin{proof}			
		Let $\tau^{*}$ denote the series subtree of interest. 
				We show $H_{l_k}^{\tau^{*}_{l_k}}\preceq_{st} H_{l_{k+1}}^{\tau^{*}_{l_{k+1}}}$ for all $k\in\{1,\ldots,k_{\max}-1\}$ when dependence parameters are all $\alpha$, $\alpha\in(0,1)$. We remark that, rooting on $l_{k+1}$, $\mathrm{ch}(l_k)=\{l_{k-1}\}$; meanwhile rooting on $l_{k}$, $\mathrm{ch}(l_k)=\{l_{k-1}, l_{k+1}\}$. Also, from (\ref{eq:jointpgf-h}), $\eta_{l_{k+2}}^{\tau^{*}_{l_{k}}}(t\vecun{l_{k+2}}) = \eta_{l_{k+2}}^{\tau^{*}_{l_{k+1}}}(t\vecun{l_{k+2}})$ as the children of $l_{k+2}$ are the same under both rootings, leading to
				\begin{align}
					\eta_{l_k}^{\tau^{*}_{l_k}}(t\vecun{l_k}) 
					&= \eta_{l_k}^{\tau^{*}_{l_{k+1}}}(t\vecun{l_k}) \times (1-\alpha+\alpha t(1-\alpha+\alpha \eta_{l_{k+2}}^{\tau^{*}_{l_{k}}}(t\vecun{l_{k+2}})))\notag\\
					&= (1-\alpha)\eta_{l_k}^{\tau^{*}_{l_{k+1}}}(t\vecun{l_k}) +\alpha t\eta_{l_k}^{\tau^{*}_{l_{k+1}}}(t\vecun{l_k})(1-\alpha+\alpha \eta_{l_{k+2}}^{\tau^{*}_{l_{k+1}}}(t\vecun{l_{k+2}})).
					\label{eq:proofTool2-1}
				\end{align}
				Next, directly from (\ref{eq:jointpgf-h}), one obtains
					\begin{align}
				&\eta_{l_{k+1}}^{\tau^{*}_{l_{k+1}}}(t\vecun{l_{k+1}})
						=t(1-\alpha+\alpha\eta_{l_k}^{\tau^{*}_{l_{k+1}}}(t\vecun{l_k}))(1-\alpha+\alpha\eta_{l_{k+2}}^{\tau^{*}_{l_{k+1}}}(t\vecun{l_{k+2}}))\notag\\
				&\quad\quad= (1-\alpha)t(1-\alpha + \alpha\eta_{l_{k+2}}^{\tau^{*}_{l_{k+1}}}(t\vecun{l_{k+2}}) +\alpha t\eta_{l_k}^{\tau^{*}_{l_{k+1}}}(t\vecun{l_k})(1-\alpha+\alpha \eta_{l_{k+2}}^{\tau^{*}_{l_{k+1}}}(t\vecun{l_{k+2}})).
			\label{eq:proofTool2-2}
				\end{align}
				We note $\eta_{l_k}^{\tau^{*}_{l_{k+1}}}(t\vecun{l_k}) = \eta_{l_{d_{\mathrm{se}}+1-k}}^{\tau^{*}_{l_{k+1}}}(t\vecun{l_{d_{\mathrm{se}}+1-k}})$ by the symmetry of the series subtree. Since $l_{d_{\mathrm{se}}+1-k}\in\{l_{k+2}\}\cup\mathrm{dsc}(l_{k+2})$ under rooting $l_{k+1}$, for $k\leq k_{\max}$, the pgfs in (\ref{eq:proofTool2-1}) and (\ref{eq:proofTool2-2}) indicate that $H_{l_k}^{\tau^*_{l_k}}\preceq_{st} H_{l_{k+1}}^{\tau^*_{l_{k+1}}}$. Hence, $\mathcal{T}^{[l_k]}\preceq_{sha} \mathcal{T}^{[l_{k+1}]}$ according to Theorem~\ref{th:ConvexOrderMShape}.  
			\end{proof}
	\begin{figure}[H]
					\centering
					\begin{tikzpicture}[every node/.style={text = black, draw=none, fill = none, text opacity = 1, rectangle, inner sep=0mm, outer sep=1mm}, node distance = 5mm, scale=0.1, thick]
						\node(a0){
							\begin{tikzpicture}[every node/.style={circle, draw = Blue!50!black, inner sep=0.1mm,   minimum size=2mm, fill = White, very thick, outer sep = 0.1mm}, node distance = 1mm, scale=0.1 , thick]
								\node(1){\tiny $l_1$};
								\node[right = of 1](2){\tiny $l_2$};
								\node[right = of 2](3){\tiny $l_3$};
								\node[draw= White, right = of 3](4){...};
								\node[rounded rectangle, right = of 4](5){\tiny ${l_{d_{se}}}$};
								
								\node[draw = Goldenrod!85!black, below = of 1](6){};
								\node[draw = Goldenrod!85!black, below left = of 6](7){};
								\node[draw = Goldenrod!85!black, below right = of 6](8){};
								
								\draw (1) --  (2);
								\draw (2) -- (3);
								\draw (3) -- (4);
								\draw (4) -- (5);
								\draw (1) -- (6);
								\draw (6) -- (7);
								\draw (6) -- (8);
							\end{tikzpicture}
						};
						\node[below = 0mm of a0]{$\mathcal{T}^{[1]}$};
						\node[right= of a0](a1){
							\begin{tikzpicture}[every node/.style={circle, draw = Blue!50!black, inner sep=0.1mm,   minimum size=2mm, fill = White, very thick, outer sep = 0.1mm}, node distance = 1mm, scale=0.1 , very thick]
								\node(1){\tiny $l_1$};
								\node[right = of 1](2){\tiny $l_2$};
								\node[right = of 2](3){\tiny $l_3$};
								\node[draw= White, right = of 3](4){...};
								\node[rounded rectangle, right = of 4](5){\tiny ${l_{d_{se}}}$};
								
								\node[draw = Goldenrod!85!black, below = of 2](6){};
								\node[draw = Goldenrod!85!black, below left = of 6](7){};
								\node[draw = Goldenrod!85!black, below right = of 6](8){};
								
								\draw (1) --  (2);
								\draw (2) -- (3);
								\draw (3) -- (4);
								\draw (4) -- (5);
								\draw (2) -- (6);
								\draw (6) -- (7);
								\draw (6) -- (8);
							\end{tikzpicture}
						};
						\node[below = 0mm of a1]{$\mathcal{T}^{[2]}$};
						\node[right= of a1](a2){
							\begin{tikzpicture}[every node/.style={circle, draw = Blue!50!black, inner sep=0.1mm,   minimum size=2mm, fill = White, very thick, outer sep = 0.1mm}, node distance = 1mm, scale=0.1 , thick]
								\node(1){\tiny $l_1$};
								\node[right = of 1](2){\tiny $l_2$};
								\node[right = of 2](3){\tiny $l_3$};
								\node[draw= White, right = of 3](4){...};
								\node[rounded rectangle, right = of 4](5){\tiny ${l_{d_{se}}}$};
								
								\node[draw = Goldenrod!85!black, below = of 3](6){};
								\node[draw = Goldenrod!85!black, below left = of 6](7){};
								\node[draw = Goldenrod!85!black, below right = of 6](8){};
								
								\draw (1) --  (2);
								\draw (2) -- (3);
								\draw (3) -- (4);
								\draw (4) -- (5);
								\draw (3) -- (6);
								\draw (6) -- (7);
								\draw (6) -- (8);
							\end{tikzpicture}
						};
						\node[below = 0mm of a2]{$\mathcal{T}^{[3]}$};
						\node[right= of a2](dots){\dots
						};
						\node[right= of dots](a3){
							\begin{tikzpicture}[every node/.style={circle, draw = Blue!50!black, inner sep=0.1mm,   minimum size=2mm, fill = White, very thick, outer sep = 0.1mm}, node distance = 1mm, scale=0.1 , thick]
								\node(1){\tiny $l_1$};
								\node[draw = White, right = of 1](2){...};
								\node[rounded rectangle, right = of 2](3){\tiny ${l_{k_{\max}}}$};
								\node[draw= White, right = of 3](4){...};
								\node[right = of 4](5){};
								\node[rounded rectangle, right = of 4](5){\tiny ${l_{d_{se}}}$};
								
								\node[draw = Goldenrod!85!black, below = of 3](6){};
								\node[draw = Goldenrod!85!black, below left = of 6](7){};
								\node[draw = Goldenrod!85!black, below right = of 6](8){};
								
								\draw (1) --  (2);
								\draw (2) -- (3);
								\draw (3) -- (4);
								\draw (4) -- (5);
								\draw (3) -- (6);
								\draw (6) -- (7);
								\draw (6) -- (8);
							\end{tikzpicture}
						};
						\node[below = 0.3mm of a3]{$\mathcal{T}^{[k_{\max}]}$};
						\draw[very thick] (a0) -- (a1);	
						\draw[very thick] (a1) -- (a2);	
						\draw[very thick] (a2) -- (dots);	
						\draw[very thick] (dots) -- (a3);	
					\end{tikzpicture}
					
					\caption{Example of a sequence of trees encompassed by Corollary~\ref{th:Tool2}. The series subtree is in blue (dark shade) while the generic subtree $\tau$ is in yellow (light shade). }
					\label{fig:Reconstr2}
				\end{figure}
	We depict in Figure~\ref{fig:Reconstr2} an example of the sequence $\{\mathcal{T}^{[k]},\, k\in\{1,\ldots,k_{\max}\}\}$ induced by a tree $\mathcal{T}^{[1]}$.
   
			\begin{cor}
				\label{th:Tool3}
				Consider $d_{\mathrm{beam}},d_{\mathrm{ray}}\in\mathbb{N}_1$, with $d_{\mathrm{beam}}+d_{\mathrm{ray}}\leq d$, assuming $d\geq 4$. Let $\{\mathcal{T}^{\{m;n\}},\, m,n\in\mathbb{N}:m+n=d_{\mathrm{ray}}\}\subseteq\Omega_{\mathcal{T}}$ be a sequence of trees such that to a $d_{\mathrm{beam}}$-vertex series subtree, which we call the \textit{beam subtree} and whose vertices are labelled $\{l_k, \, k\in\{1,\ldots,d_{\mathrm{beam}}\}\}$, we connect $m$ vertices to an end and $n$ vertices to the other end. Additional subtrees may be connected to the beam subtree, in a symmetric manner: for a subtree connected to $l_k$, an identical subtree is connected to $l_{d_{\mathrm{beam}}-k}$, $k\in\{1,\ldots,\lfloor\frac{d_{\mathrm{beam}}}{2} \rfloor\}$. For $m_1,n_1,m_2,n_2\in\mathbb{N}$ such that $m_1+n_1=m_2+n_2$ , if $|m_1-n_1|\leq|m_2-n_2|$, then $\mathcal{T}^{\{m_1;n_1\}}\preceq_{sha}\mathcal{T}^{\{m_2;n_2\}}$. 
			\end{cor}
			
			\begin{proof} 
				Let $\tau^*=\mathcal{T}^{\{a^{*}; b\}}$, with $a^{*}=a-1$. We establish stochastic dominance ordering between $H_{l_1}^{\tau^*_{l_1}}$ and $H_{l_{d_{\mathrm{beam}}}}^{\tau^*_{l_{d_{\mathrm{beam}}}}}$ with all dependence parameters being $\alpha$, $\alpha\in(0,1)$. Let us denote by $\tau^{\prime\prime}$ the beam subtree with the additional subtrees connected to it. From (\ref{eq:jointpgf-h}), we have
				\begin{align}
					&\eta_{l_1}^{\tau^*_{l_1}}(t\vecun{l_1}) = (1-\alpha + \alpha t)^{a^{*}}\eta_{l_1}^{\tau^{\prime\prime}_{l_1}}(t,t,\ldots, \underbrace{t(1-\alpha+\alpha t)^b}_{\text{position }l_{d_{\mathrm{beam}}}},\ldots, t, t);
					\label{eq:proofTool3-1}\\
				&\eta_{l_{d_{\mathrm{beam}}}}^{\tau^*_{l_{d_{\mathrm{beam}}}}}(t\vecun{l_{d_{\mathrm{beam}}}}) = (1-\alpha + \alpha t)^b\eta_{l_{d_{\mathrm{beam}}}}^{\tau^{\prime\prime}_{l_{d_{\mathrm{beam}}}}}((t,t,\ldots,  \underbrace{t(1-\alpha+\alpha t)^{a^{*}}}_{\text{position }l_1},\ldots,t,t),
					\label{eq:proofTool3-2}
				\end{align}
				for $t\in[-1,1]$. By the symmetry of $\tau^{\prime\prime}$, one has $\eta_{l_1}^{\tau^{\prime\prime}_{l_1}}(t\vecun{l_1}) = \eta_{l_{d_{\mathrm{beam}}}}^{\tau^{\prime\prime}_{l_{d_{\mathrm{beam}}}}}(t\vecun{l_{d_{\mathrm{beam}}}})$. Therefore, from (\ref{eq:proofTool3-1}) and (\ref{eq:proofTool3-2}), we infer that $H_{l_1}^{\tau^*_{l_1}}\preceq_{st}H_{l_{d_{\mathrm{beam}}}}^{\tau^*_{l_{d_{\mathrm{beam}}}}}$ as long as $a^*\leq b$, and vice-versa. Consequently, $\mathcal{T}^{\{a^{*}+1;b\}}\preceq_{sha} \mathcal{T}^{\{a^{*}; b+1\}}$ if $a^{*}\leq b$, and $\mathcal{T}^{\{a^{*}; b+1\}}\preceq_{sha} \mathcal{T}^{\{a^{*}+1;b\}}$ if $a^{*}\geq b$, according to Theorem~\ref{th:ConvexOrderMShape}. The result follows from the transitivity of $\preceq_{sha}$.
			\end{proof}

\begin{figure}[H]
					\centering
					\begin{tikzpicture}[every node/.style={text = black, draw=none, fill = none, text opacity = 1, rectangle, inner sep=0mm, outer sep=1mm}, node distance = 5mm, scale=0.1, thick]
						\node(a0){
							\begin{tikzpicture}[every node/.style={circle, draw = Goldenrod!85!black, inner sep=0.1mm,   minimum size=2mm, fill = White, very thick, outer sep = 0.1mm}, node distance = 1mm, scale=0.1 , thick]
								\node [draw=  Blue!50!black, minimum size=2.5mm](1) {\tiny $l_1$};
								\node [draw=  Blue!50!black, minimum size=2.5mm, right=of 1] (2) {\tiny $l_2$};
								\node [draw=  Blue!50!black, minimum size=2.5mm, right=of 2] (3) {\tiny $l_3$};
								\node [draw=  Blue!50!black, minimum size=2.5mm, right=of 3] (6) {\tiny $l_4$};
								\node [draw=  Blue!50!black, minimum size=2.5mm, right=of 6] (7) {\tiny $l_5$};
								\node [draw = White, below left = of 1](11){...};
								\node [above = of 1](8){};
								\node [above left = of 1](9){};
								\node [left = of 1](10){};
								\node [below = of 1](12){};
								\node[draw = Red!75, minimum size=2mm, above = of 2](sub1){};
								\node[draw = Red!75, minimum size=2mm, above = of sub1](sub2){};
								\node[draw = Red!75, minimum size=2mm, above = of 3](sub3){};
								\node[draw = Red!75, minimum size=2mm, above = of 6](sub4){};
								\node[draw = Red!75, minimum size=2mm, above = of sub4](sub5){};

								\draw (1) --  (2);
								\draw (2) -- (3);
								\draw (3) -- (6);
								\draw (6) -- (7);
								\draw (1) -- (8);
								\draw (1) -- (9);
								\draw (1) -- (10);
								\draw (1) -- (11);
								\draw (1) -- (12);
								\draw (2) -- (sub1);
								\draw (sub1) -- (sub2);
								\draw (3) -- (sub3);
								\draw (6) -- (sub4);
								\draw (sub4) -- (sub5);
							\end{tikzpicture}
						};
						\node[below = 1mm of a0]{$\mathcal{T}^{\{d_{\mathrm{ray}};0\}}$};
						\node[right= of a0](a1){
							\begin{tikzpicture}[every node/.style={circle, draw = Goldenrod!85!black, inner sep=0.1mm,   minimum size=2mm, fill = White, very thick, outer sep = 0.1mm}, node distance = 1mm, scale=0.1 , thick]
								\node [draw=  Blue!50!black](1) {\tiny $l_1$};
								\node [draw=  Blue!50!black, right=of 1] (2) {\tiny $l_2$};
								\node [draw=  Blue!50!black, right=of 2] (3) {\tiny $l_3$};
								\node [draw=  Blue!50!black, right=of 3] (6) {\tiny $l_4$};
								\node [draw=  Blue!50!black, right=of 6] (7) {\tiny $l_5$};
								\node [draw = White, below left=of 1](11){...};
								\node [below = of 7](8){};
								\node [above left = of 1](9){};
								\node [left = of 1](10){};
								\node [below = of 1](12){};
								
								\node[draw = Red!75, minimum size=2mm, above = of 2](sub1){};
								\node[draw = Red!75, minimum size=2mm, above = of sub1](sub2){};
								\node[draw = Red!75, minimum size=2mm, above = of 3](sub3){};
								\node[draw = Red!75, minimum size=2mm, above = of 6](sub4){};
								\node[draw = Red!75, minimum size=2mm, above = of sub4](sub5){};
								
								\draw (1) --  (2);
								\draw (2) -- (3);
								\draw (3) -- (6);
								\draw (6) -- (7);
								\draw (7) -- (8);
								\draw (1) -- (9);
								\draw (1) -- (10);
								\draw (1) -- (11);
								\draw (1) -- (12);
								\draw (2) -- (sub1);
								\draw (sub1) -- (sub2);
								\draw (3) -- (sub3);
								\draw (6) -- (sub4);
								\draw (sub4) -- (sub5);
							\end{tikzpicture}
						};
						\node[below = 1mm of a1]{$\mathcal{T}^{\{d_{\mathrm{ray}}-1;1\}}$};
						\node[right= of a1](a2){
							\begin{tikzpicture}[every node/.style={circle, draw = Goldenrod!85!black, inner sep=0.1mm,   minimum size=2mm, fill = White, very thick, outer sep = 0.1mm}, node distance = 1mm, scale=0.1 , thick]
								\node [draw=  Blue!50!black, minimum size=2.5mm](1) {\tiny $l_1$};
								\node [draw=  Blue!50!black, minimum size=2.5mm, right=of 1] (2) {\tiny $l_2$};
								\node [draw=  Blue!50!black, minimum size=2.5mm, right=of 2] (3) {\tiny $l_3$};
								\node [draw=  Blue!50!black, minimum size=2.5mm, right=of 3] (6) {\tiny $l_4$};
								\node [draw=  Blue!50!black, minimum size=2.5mm, right=of 6] (7) {\tiny $l_5$};
								\node [draw = White, below left =of 1](11){...};
								\node [below = of 7](8){};
								\node [below right = of 7](9){};
								\node [left = of 1](10){};
								\node [below = of 1](12){};
								
								\node[draw = Red!75, minimum size=2mm, above = of 2](sub1){};
								\node[draw = Red!75, minimum size=2mm, above = of sub1](sub2){};
								\node[draw = Red!75, minimum size=2mm, above = of 3](sub3){};
								\node[draw = Red!75, minimum size=2mm, above = of 6](sub4){};
								\node[draw = Red!75, minimum size=2mm, above = of sub4](sub5){};

								\draw (1) --  (2);
								\draw (2) -- (3);
								\draw (3) -- (6);
								\draw (6) -- (7);
								\draw (7) -- (8);
								\draw (7) -- (9);
								\draw (1) -- (10);
								\draw (1) -- (11);
								\draw (1) -- (12);
								\draw (2) -- (sub1);
								\draw (sub1) -- (sub2);
								\draw (3) -- (sub3);
								\draw (6) -- (sub4);
								\draw (sub4) -- (sub5);
							\end{tikzpicture}
						};
						\node[below = 1mm of a2]{$\mathcal{T}^{\{d_{\mathrm{ray}}-2;2\}}$};
						\node[right= of a2](dots){\dots
						};
						
						\node[right= of dots](a4){
							\begin{tikzpicture}[every node/.style={circle, draw = Goldenrod!85!black, inner sep=0.1mm,   minimum size=2mm, fill = White, very thick, outer sep = 0.1mm}, node distance = 1mm, scale=0.1 , thick]
								\node [draw=  Blue!50!black, minimum size=2.5mm](1) {\tiny $l_1$};
								\node [draw=  Blue!50!black, minimum size=2.5mm, right=of 1] (2) {\tiny $l_2$};
								\node [draw=  Blue!50!black, minimum size=2.5mm, right=of 2] (3) {\tiny $l_3$};
								\node [draw=  Blue!50!black, minimum size=2.5mm, right=of 3] (6) {\tiny $l_4$};
								\node [draw=  Blue!50!black, minimum size=2.5mm, right=of 6] (7) {\tiny $l_5$};
								\node [draw = White, below right=of 7](11){...};
								\node [above = of 7](8){};
								\node [above right = of 7](9){};
								\node [right = of 7](10){};
								\node [below = of 7](12){};
								
								\node[draw = Red!75, minimum size=2mm, above = of 2](sub1){};
								\node[draw = Red!75, minimum size=2mm, above = of sub1](sub2){};
								\node[draw = Red!75, minimum size=2mm, above = of 3](sub3){};
								\node[draw = Red!75, minimum size=2mm, above = of 6](sub4){};
								\node[draw = Red!75, minimum size=2mm, above = of sub4](sub5){};

								\draw (1) --  (2);
								\draw (2) -- (3);
								\draw (3) -- (6);
								\draw (6) -- (7);
								\draw (7) -- (8);
								\draw (7) -- (9);
								\draw (7) -- (10);
								\draw (7) -- (11);
								\draw (7) -- (12);
								\draw (2) -- (sub1);
								\draw (sub1) -- (sub2);
								\draw (3) -- (sub3);
								\draw (6) -- (sub4);
								\draw (sub4) -- (sub5);
							\end{tikzpicture}
						};
						\node[below = 1mm of a4]{$\mathcal{T}^{\{0;d_{\mathrm{ray}}\}}$};

						\draw[very thick] (a0) -- (a1);	
						\draw[very thick] (a1) -- (a2);	
						\draw[very thick] (a2) -- (dots);	
						\draw[very thick] (dots) -- (a4);		
					\end{tikzpicture}
					\caption{Example of a sequence of trees encompassed by Corollary~\ref{th:Tool3}.  The moving vertices are in yellow (light shade), the beam subtree is in blue (dark shade), and the additional subtrees attached to it are in orange (mid shade).}
					\label{fig:Reconstr3}
				\end{figure}
   
							We depict in Figure~\ref{fig:Reconstr3} an example of a sequence $\{\mathcal{T}^{\{m;n\}},\, m,n\in\mathbb{N}:m+n=d_{\mathrm{ray}}\}$ with $d_{\mathrm{beam}}=5$, as encompassed by Corollary~\ref{th:Tool3}.
			Combining the use of Corollaries~\ref{th:Tool1}, \ref{th:Tool2} and \ref{th:Tool3} offers wide possibilities of comparison, as exhibited in the following example. 
			
			\begin{ex}
				Assuming a set of $d=9$ vertices, let $\mathcal{T}=$ \begin{tikzpicture}[every node/.style={text=White, text opacity=0, circle, draw = Blue!50!black, inner sep=0mm,   minimum size=1mm, fill = White}, node distance = 1mm, scale=0.1, thick]
					\node (1) {};
					\node [right=of 1] (2) {};
					\node [right =of 2] (3) {};
					\node [right =of 3] (4) {};
					\node [right=of 4] (5) {};
					\node [above right= of 5] (6) {};
					\node [below right =of 5] (7){};
					\node [right=of 6](8){};
					\node[right=of 7](9){};
					
					\draw (1) --  (2);
					\draw (2) -- (3);
					\draw (3) -- (4);
					\draw (4) -- (5);
					\draw (5) -- (6);
					\draw (5) -- (7);
					\draw (6) -- (8);
					\draw (7) -- (9);
				\end{tikzpicture} and $\mathcal{T}^{\prime}=$ 
				\begin{tikzpicture}[every node/.style={text=White, text opacity=0, circle, draw = Blue!50!black, inner sep=0mm,   minimum size=1mm, fill = White}, node distance = 1mm, scale=0.1, thick]
					\node (1) {};
					\node [right=of 1] (2) {};
					\node [right=of 2] (3) {};
					\node [above left=of 1] (4) {};
					\node [below left=of 1] (5) {};
					\node [above= of 2] (6) {};
					\node [below=of 2] (7){};
					\node [above right=of 3](8){};
					\node[below right=of 3](9){};
					
					\draw (1) --  (2);
					\draw (2) -- (3);
					\draw (1) -- (4);
					\draw (1) -- (5);
					\draw (2) -- (6);
					\draw (2) -- (7);
					\draw (3) -- (8);
					\draw (3) -- (9);
				\end{tikzpicture}. By first applying  Corollary~\ref{th:Tool3}, then Corollary~\ref{th:Tool1}, and finally Corollary~\ref{th:Tool2}, as shown in Figure~\ref{fig:ExampleTools}, we establish $\mathcal{T}\preceq_{sha}\mathcal{T}^{\prime}$. Note that this ordering agrees with the Hasse diagram depicted in Figure~\ref{fig:ShapesBonanza2}.
				\hfill $\square$
			\end{ex}
			
			\begin{figure}[H]
				\centering
				\begin{tikzpicture}[every node/.style={text = black, draw=none, fill = none, text opacity = 1, rectangle, inner sep=0mm, outer sep=1mm}, node distance = 1mm, scale=0.1, thick]
					\node(a0){
						\begin{tikzpicture}[every node/.style={circle, draw = Blue!50!black, inner sep=0mm,   minimum size=2mm, fill = White, outer sep = 0.1mm, very thick}, scale=0.1 , thick]
							\node (1) {};
							\node [right = of 1](2) {};
							\node [right = of 2](3) {};
							\node [right = of 3](4) {};
							\node [below = of 1](5) {};
							\node [below left = of 5] (6) {};
							\node [below right= of 5] (7) {};
							\node [below = of 6](8){};
							\node [below = of 7](9){};
							
							\draw (1) --  (2);
							\draw (2) -- (3);
							\draw (3) -- (4);
							\draw (1) -- (5);
							\draw (5) -- (6);
							\draw (5) -- (7);
							\draw (6) -- (8);
							\draw (7) -- (9);
						\end{tikzpicture}
					};
					\node[below = 1mm of a0]{$\mathcal{T}$};
					\node[right= of a0](sha0){$\substack{\preceq_{sha}\\\quad\\\scriptstyle{\text{Corollary~\ref{th:Tool3}}}}$
					};
					\node[right= of sha0](a1){
						\begin{tikzpicture}[every node/.style={circle, draw = Blue!50!black, inner sep=0mm,   minimum size=2mm, fill = White, outer sep = 0.1mm, very thick}, scale=0.1 , thick]
							\node (1) {};
							\node [right = of 1](2) {};
							\node [right = of 2](3) {};
							\node [right = of 3](4) {};
							\node [below = of 1](5) {};
							\node [below left = of 5] (6) {};
							\node [below right= of 5] (7) {};
							\node [below left= of 7](8){};
							\node [below right= of 7](9){};
							
							\draw (1) --  (2);
							\draw (2) -- (3);
							\draw (3) -- (4);
							\draw (1) -- (5);
							\draw (5) -- (6);
							\draw (5) -- (7);
							\draw (7) -- (8);
							\draw (7) -- (9);
						\end{tikzpicture}
					};
					\node[right= of a1](sha1){$\substack{\preceq_{sha}\\\quad\\\scriptstyle{\text{Corollary~\ref{th:Tool1}}}}$
					};
					\node[right= of sha1](a2){
						\begin{tikzpicture}[every node/.style={circle, draw = Blue!50!black, inner sep=0mm,   minimum size=2mm, fill = White, outer sep = 0.1mm, very thick},  scale=0.1, thick]
							\node (1) {};
							\node [right = of 1](2) {};
							\node [right = of 2](3) {};
							\node [below = of 1](5) {};
							\node [below left = of 5] (6) {};
							\node [below right= of 5] (7) {};
							\node [below left= of 7](8){};
							\node [below right= of 7](9){};
							\node [left = of 5](4) {};
							
							\draw (1) --  (2);
							\draw (2) -- (3);
							\draw (5) -- (4);
							\draw (1) -- (5);
							\draw (5) -- (6);
							\draw (5) -- (7);
							\draw (7) -- (8);
							\draw (7) -- (9);
						\end{tikzpicture}
					};
					\node[right= of a2](sha2){$\substack{\preceq_{sha}\\\quad\\\scriptstyle{\text{Corollary~\ref{th:Tool2}}}}$
					};
					\node[right= of sha2](a3){
						\begin{tikzpicture}[every node/.style={circle, draw = Blue!50!black, inner sep=0mm,   minimum size=2mm, fill = White, outer sep = 0.1mm, very thick},  scale=0.1 , thick]
							\node (1) {};
							\node [right = of 1](2) {};
							\node [right = of 2](3) {};
							\node [below = of 2](5) {};
							\node [below left = of 5] (6) {};
							\node [below right= of 5] (7) {};
							\node [below left= of 7](8){};
							\node [below right= of 7](9){};
							\node [left = of 5](4) {};
							
							\draw (1) --  (2);
							\draw (2) -- (3);
							\draw (5) -- (4);
							\draw (2) -- (5);
							\draw (5) -- (6);
							\draw (5) -- (7);
							\draw (7) -- (8);
							\draw (7) -- (9);
						\end{tikzpicture}
					};
					\node[below = 1mm of a3]{$\mathcal{T}^{\prime}$};
				\end{tikzpicture}
				\caption{A combined use of Corollaries~\ref{th:Tool1}, \ref{th:Tool2} and \ref{th:Tool3} to order using $\preceq_{sha}$}
				\label{fig:ExampleTools}
			\end{figure}

			Bases of comparison for trees already exist in the (spectral) graph theory literature; however, they have not been explored, to our knowledge, in the context of MRFs. In this section, we look at a few of them and provide discussion and examples showing their distinctiveness with $\preceq_{sha}$. 
			
			\bigskip
			
			The \textit{degree} of a vertex is the number of edges connected to it. Let $\boldsymbol{\pi}^{\mathcal{T}}$ denote the decreasingly ordered vector of degrees of a tree $\mathcal{T}$. Trees may be gathered in categories following their respective vector $\boldsymbol{\pi}^{\mathcal{T}}$. Ordering between two trees of different categories, say $\mathcal{T}$ and $\mathcal{T}^{\prime}$, is established if $\sum_{i=1}^k\pi_{i}^{\mathcal{T}}\leq\sum_{i=1}^k\pi_{i}^{\mathcal{T}^{\prime}}$, for all $k\in\{1,\ldots,d\}$. Such a comparison is called \textit{majorization}; see \cite{marshall2011inequalities}. Thus categorizing and comparing trees have been prominently studied in the context of spectral graph theory, notably in \cite{liu2009majorization} and subsequent work. Majorization of vectors of degrees does not produce the order given by $\preceq_{sha}$. On one hand, trees having the same $\boldsymbol{\pi}^{\mathcal{T}}$ may still be ordered according to $\preceq_{sha}$; $\mathcal{T}^{[2]}$ and $\mathcal{T}^{[3]}$ of Figure~\ref{fig:Reconstr2} are an example. On the other hand, a tree whose vector of degrees majorizes another's may be incomparable according to $\preceq_{sha}$; take $\mathcal{T}$ and $\mathcal{T}^{\prime}$ of Figure~\ref{fig:CounterEx} as an example. 
			
			\bigskip
			
			Another ground of comparison of trees is the algebraic connectivity, introduced in \cite{fiedler1973algebraic}. Algebraic connectivity corresponds to the second smallest eigenvalue of a graph's Laplacian matrix, defined as the identity matrix minus the adjacency matrix; one may consult \cite{grone1987algebraic} for further insight in the context of trees.
			Comparing with respect to $\preceq_{sha}$ is distinct than ordering trees according to algebraic connectivity as investigated in \cite{grone1990ordering}. A glance at Table~1 therein and Figure~\ref{fig:ShapesBonanza2} for $d=7$ allows to visualize the distinction.

			Spectral graph theory provides other means of comparison of trees,
			notably employing the spectral radius or the Estrada index.
			Denote the set of eigenvalues of the adjacency matrix of a graph by $\{\mu_i,\, i\in\{1,\ldots, d\}\}$. Then, the spectral radius of that graph is given by $\rho_{\mathrm{spec}}(\mathcal{T}) = \max_{i=1}^d \mu_i$ and its Estrada index by $EE(\mathcal{T})=\sum_{i=1}^d \mathrm{e}^{\mu_i}$. One may consult \cite{gutman2011estrada} for a review on the Estrada index.
			Exhaustive computations showed that, for $d\in\{4,\ldots,8\}$, $\mathcal{T}\preceq_{sha}\mathcal{T}^{\prime}$ yields $\rho_{\mathrm{spec}}(\mathcal{T})\leq\rho_{\mathrm{spec}}(\mathcal{T}^{\prime})$ and $EE(\mathcal{T}) \leq EE(\mathcal{T}^{\prime})$. This remains mostly true for larger number of vertices, but exceptions surface, as exhibited in Example~\ref{ex:SpectrRadius-Estrada}.  Moreover, referring to Figure~\ref{fig:ShapesBonanza2}, we observe \begin{tikzpicture}[every node/.style={text=White, text opacity=0, circle, draw = Blue!50!black, inner sep=0mm,   minimum size=0.7mm, fill = White}, node distance = 0.7mm, scale=0.1, thick]
				\node (1) {};
				\node [right=of 1] (2) {};
				\node [right =of 2] (3) {};
				\node [right =of 3] (4) {};
				\node [right=of 4] (5) {};
				\node [right= of 5] (6) {};
				\node [above right =of 6] (7){};
				\node [below right=of 6](8){};
				\node[above=of 5](9){};
				
				\draw (1) --  (2);
				\draw (2) -- (3);
				\draw (3) -- (4);
				\draw (4) -- (5);
				\draw (5) -- (6);
				\draw (6) -- (7);
				\draw (6) -- (8);
				\draw (5) -- (9);
			\end{tikzpicture} \,$\preceq_{sha}$\,\begin{tikzpicture}[every node/.style={text=White, text opacity=0, circle, draw = Blue!50!black, inner sep=0mm,   minimum size=0.7mm, fill = White}, node distance = 0.7mm, scale=0.1, thick]
				\node (1) {};
				\node [right=of 1] (2) {};
				\node [right =of 2] (3) {};
				\node [above right =of 3] (4) {};
				\node [right=of 4] (5) {};
				\node [below right= of 3] (6) {};
				\node [right =of 6] (7){};
				\node [above left=of 1](8){};
				\node[below left=of 1](9){};
				
				\draw (1) --  (2);
				\draw (2) -- (3);
				\draw (3) -- (4);
				\draw (4) -- (5);
				\draw (3) -- (6);
				\draw (6) -- (7);
				\draw (1) -- (8);
				\draw (1) -- (9);
			\end{tikzpicture}, while both trees' adjacency matrix have the same eigenvalues: they are cospectral. Tools from spectral graph theory, even apart from the spectral radius and the Estrada index, are therefore not sufficient to entirely describe the comparison dynamics at play in $(\Omega_{\mathcal{T}}, \preceq_{sha})$. 
			
			\begin{ex}
				\label{ex:SpectrRadius-Estrada}
				Consider trees $\mathcal{T}$ and $\mathcal{T}^{\prime}$ depicted in Figure~\ref{fig:SpectrRadius-Estrada}. They differ only by the anchorage of vertex $9$, connected to $3$ in one case and to $7$ in the other. From (\ref{eq:jointpgf-h}), for $\tau^*$ the subtree obtained by pruning $9$, we have
				\begin{align*}
					\eta_{3}^{\tau^*_3}(t\,\boldsymbol{1}_8) &= t(1-\alpha+\alpha t[1-\alpha+\alpha t])(1-\alpha+\alpha t[1-\alpha+\alpha t\{1-\alpha+\alpha t\}\{1-\alpha+\alpha t(1-\alpha+\alpha t)\}]) ;\\
					\eta_{7}^{\tau^*_7}(t\,\boldsymbol{1}_8) &= t\{1-\alpha+\alpha t\}\{1-\alpha+\alpha t(1-\alpha+\alpha t)(1-\alpha+\alpha t[1-\alpha+\alpha t\{1-\alpha+\alpha t(1-\alpha+\alpha t)\}])\}  ,
				\end{align*}
				which, when developed, indicates $H_{7}^{\tau^*_7}\preceq_{st}H_{3}^{\tau^*_3}$. Hence, $\mathcal{T}\preceq_{sha}\mathcal{T}^{\prime}$ by Theorem~\ref{th:ConvexOrderMShape}, although 
				$\rho_{\mathrm{spec}}(\mathcal{T})= 2.0840 > 2.0743 =\rho_{\mathrm{spec}}(\mathcal{T}^{\prime})$ and $EE(\mathcal{T})= 19.4594 > 19.4591 =EE(\mathcal{T}^{\prime})$.
			\end{ex}
			
			\begin{figure}[H]
				\centering
				\begin{tikzpicture}[every node/.style={text=Black, circle, draw = Blue!50!black, inner sep=0mm,   minimum size=3mm, fill = White}, node distance = 1.5mm, scale=0.1, very thick]
					\node (1a) {\tiny $1$};
					\node [right=of 1a] (2a) {\tiny $2$};
					\node [above right=of 2a] (3a) {\tiny $3$};
					\node [right=of 3a] (4a) {\tiny $4$};
					\node [right=of 4a] (5a) {\tiny $5$};
					\node [above=of 5a] (6a) {\tiny $6$};
					\node [right=of 5a] (7a) {\tiny $7$};
					\node [below right=of 7a] (8a) {\tiny $8$};
					\node [draw =Goldenrod!85!black, above right =of 7a] (9a) {\tiny $9$};

					\draw (1a) -- (2a);
					\draw (2a) -- (3a);
					\draw (3a) -- (4a);
					\draw (4a) -- (5a);
					\draw (5a) -- (6a);
					\draw (5a) -- (7a);
					\draw (7a) -- (8a);
					\draw (7a) -- (9a);

					\node [right= 35mm of 1a](1b) {\tiny $1$};
					\node [right=of 1b] (2b) {\tiny $2$};
					\node [above right=of 2b] (3b) {\tiny $3$};
					\node [right=of 3b] (4b) {\tiny $4$};
					\node [right=of 4b] (5b) {\tiny $5$};
					\node [above=of 5b] (6b) {\tiny $6$};
					\node [right=of 5b] (7b) {\tiny $7$};
					\node [below right=of 7b] (8b) {\tiny $8$};
					\node [draw =Goldenrod!85!black, above left =of 3b] (9b) {\tiny $9$};
					
					\draw (1b) -- (2b);
					\draw (2b) -- (3b);
					\draw (3b) -- (4b);
					\draw (4b) -- (5b);
					\draw (5b) -- (6b);
					\draw (5b) -- (7b);
					\draw (7b) -- (8b);
					\draw (3b) -- (9b);
					
					\node[draw=none, below = 4mm of 4a](t1){$\mathcal{T}^{\mathcolor{White}{\prime}}$};
					\node[draw=none, below= 4mm of 4b](t2){$\mathcal{T}^{\prime}$};
				\end{tikzpicture}
				\caption{Trees $\mathcal{T}$ and $\mathcal{T}^{\prime}$ from Example~\ref{ex:SpectrRadius-Estrada}}
				\label{fig:SpectrRadius-Estrada}
			\end{figure}

			In \cite{csikvari2010poset}, the author introduces a transformation on trees, the \textit{generalized tree shift}, and designs a poset according to this transformation. This idea was further explored in \cite{csikvari2013poset}. According to this poset, a tree $\mathcal{T}$ precedes another tree $\mathcal{T}^{\prime}$, which we denote $\mathcal{T}\preceq_{gts}\mathcal{T}^{\prime}$, if $\mathcal{T}^{\prime}$ is the image of a generalized tree shift on $\mathcal{T}$. This gives the poset $(\Omega_{\mathcal{T}},\preceq_{gts})$. Roughly, $(\Omega_{\mathcal{T}}, \preceq_{gts})$ is coarser than $(\Omega_{\mathcal{T}}, \preceq_{sha})$, that is, many comparisons with respect to $\preceq_{sha}$ cannot be established according to $\preceq_{gts}$.  
			Corollaries~\ref{th:Tool2} and~\ref{th:Tool3} provide evidence of that: in sequences of trees encompassed by either, intermediary trees  cannot be obtained from one another through a generalized tree shift. Conversely,         
			most generalized tree shifts are reproducible through successive applications of Theorem~\ref{th:ConvexOrderMShape}, and thus, in most cases, $\mathcal{T}\preceq_{gts}\mathcal{T}^{\prime}$ implies $\mathcal{T}\preceq_{sha}\mathcal{T}^{\prime}$. 
			Evidence suggests this is not true in general, however. Figure~\ref{fig:GTS} depicts two trees, $\mathcal{T}$ and $\mathcal{T}^{\prime}$, such that $\mathcal{T}^{\prime}$ is the result of a generalized tree shift on $\mathcal{T}$, meaning $\mathcal{T}\preceq_{gts}\mathcal{T}^{\prime}$. However, for every $\tau^*$ obtained by pruning one of the subtrees in orange (mid shade) or yellow (light shade) from $\mathcal{T}$, the condition $H_v^{\tau^*_v}\preceq_{st}H_w^{\tau^*_w}$ is not met, suggesting incomparability according to $\preceq_{sha}$ since no iterative reconstruction seems to permit comparison. It still seems reasonable to believe $M\preceq_{cx}M^{\prime}$, with both respectively defined on $\mathcal{T}$ and $\mathcal{T}^{\prime}$, and their dependence parameters identical for every edge. Figure~\ref{fig:GTS} hence provides an example illustrating that the need for trees to differ only by one edge to meet the conditions of Theorem~\ref{th:ConvexOrderMShape} may be limiting. Further results on convex orderings could possibly refine the poset $(\Omega_{\mathcal{T}},\preceq_{sha})$ by permitting comparisons unreachable with the sole result of Theorem~\ref{th:ConvexOrderMShape} and iterative reconstructions. Recall that the condition $H_v^{\tau^*_v}\preceq_{st}H_w^{\tau^*_w}$ is simply sufficient, not necessary. In that prospect, the implication $\mathcal{T}\preceq_{gts}\mathcal{T}^{\prime}\implies\mathcal{T}\preceq_{sha}\mathcal{T}^{\prime}$ could become valid. 
			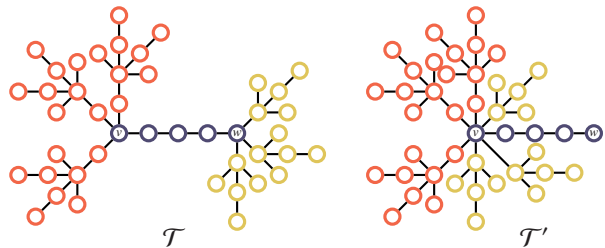
\begin{figure}[H]
				\centering
				\begin{tikzpicture}[every node/.style={text=Black, circle, draw = Blue!50!black, inner sep=0mm,  minimum size=2mm, fill = White, very thick}, node distance = 1.5mm, scale=0.1, thick]
					\node (1a) {\tiny $v$};
					\node [right=of 1a] (2a) {};
					\node [right=of 2a] (mida) {};
					\node [right=of mida] (midmida) {};
					\node [right=of midmida] (3a) {\tiny $w$};
					\node [draw =Red!75, above left = of 1a] (4a) {};
					\node [draw =Red!75, above left =of 4a] (5a) {};
					\node [draw =Red!75, above =of 5a] (6a) {};
					\node [draw =Red!75, above left =of 5a] (7a) {};
					\node [draw =Red!75, left =of 5a] (8a) {};
					\node [draw =Red!75, below left =of 5a] (9a) {};
					\node [draw =Red!75, above = of 1a] (10a) {};
					\node [draw =Red!75, above = of 10a] (11a) {};
					\node [draw =Red!75, right = of 11a] (12a) {};
					\node [draw =Red!75, above right = of 11a] (13a) {};
					\node [draw =Red!75, above = of 11a] (14a) {};
					\node [draw =Red!75, above left = of 11a] (15a) {};
					\node [draw =Red!75, below left = of 1a] (16a) {};
					\node [draw =Red!75, below left = of 16a] (17a) {};
					\node [draw =Red!75, left = of 17a] (18a) {};
					\node [draw =Red!75, below left = of 17a] (19a) {};
					\node [draw =Red!75, below = of 17a] (20a) {};
					\node [draw =Red!75, above left = of 17a] (21a) {};
					\node [draw =   Goldenrod!85!black, below = of 3a] (22a) {};
					\node [draw =   Goldenrod!85!black, below = of 22a] (23a) {};
					\node [draw =   Goldenrod!85!black, below left = of 22a] (24a) {};
					\node [draw =   Goldenrod!85!black, below right = of 22a] (25a) {};
					\node [draw =   Goldenrod!85!black, below = of 23a] (26a) {};
					\node [draw =   Goldenrod!85!black, below right = of 3a] (27a) {};
					\node [draw =   Goldenrod!85!black, below right = of 27a] (28a) {};
					\node [draw =   Goldenrod!85!black, right = of 27a] (29a) {};
					\node [draw =   Goldenrod!85!black, above right = of 27a] (30a) {};
					\node [draw =   Goldenrod!85!black, right = of 29a] (31a) {};
					\node [draw =   Goldenrod!85!black, above right = of 3a] (32a) {};
					\node [draw =   Goldenrod!85!black, above right = of 32a] (33a) {};
					\node [draw =   Goldenrod!85!black, above = of 32a] (34a) {};
					\node [draw =   Goldenrod!85!black, right = of 32a] (35a) {};
					\node [draw =   Goldenrod!85!black, above right = of 33a] (36a) {};
					\node [draw =Red!75, above right = of 13a] (37a) {};
					\node [draw =Red!75, above = of 14a] (38a) {};
					\node [draw =Red!75, left = of 18a] (39a) {};
					\node [draw =Red!75, below left = of 19a] (40a) {};
					\node [draw =Red!75, above left = of 7a] (41a) {};
					\node [draw =Red!75, left = of 8a] (42a) {};
					
					\draw (1a) -- (2a);
					\draw (2a) -- (mida);
					\draw (mida) -- (midmida);
					\draw (midmida) -- (3a);
					\draw (1a) -- (4a);
					\draw (4a) -- (5a);            
					\draw (5a) -- (6a);       
					\draw (5a) -- (7a);
					\draw (5a) -- (8a);
					\draw (5a) -- (9a);
					\draw (1a) -- (10a);
					\draw (10a) -- (11a);
					\draw (11a) -- (12a);
					\draw (11a) -- (13a);
					\draw (11a) -- (14a);
					\draw (11a) -- (15a);
					\draw (1a) -- (16a);
					\draw (16a) -- (17a);
					\draw (17a) -- (18a);
					\draw (17a) -- (19a);
					\draw (17a) -- (20a);
					\draw (17a) -- (21a);
					\draw (3a) -- (22a);
					\draw (22a) -- (23a);
					\draw (22a) -- (24a);
					\draw (22a) -- (25a);
					\draw (23a) -- (26a);
					\draw (3a) -- (27a);
					\draw (27a) -- (28a);
					\draw (27a) -- (29a);
					\draw (27a) -- (30a);
					\draw (29a) -- (31a);
					\draw (3a) -- (32a);
					\draw (32a) -- (33a);
					\draw (32a) -- (34a);
					\draw (32a) -- (35a);
					\draw (33a) -- (36a);
					\draw (13a) -- (37a);
					\draw (14a) -- (38a);
					\draw (18a) -- (39a);
					\draw (19a) -- (40a);
					\draw (7a) -- (41a);
					\draw (8a) -- (42a);
					
					\node [right = 45mm of 1a](1b) {\tiny $v$};
					\node [right=of 1b] (2b) {};
					\node [right=of 2b] (midb) {};
					\node [right=of midb] (midmidb) {};
					\node [right=of midmidb] (3b) {\tiny $w$};
					\node [draw =Red!75, above left = of 1b] (4b) {};
					\node [draw =Red!75, above left =of 4b] (5b) {};
					\node [draw =Red!75, above =of 5b] (6b) {};
					\node [draw =Red!75, above left =of 5b] (7b) {};
					\node [draw =Red!75, left =of 5b] (8b) {};
					\node [draw =Red!75, below left =of 5b] (9b) {};
					\node [draw =Red!75, above = of 1b] (10b) {};
					\node [draw =Red!75, above = of 10b] (11b) {};
					\node [draw =Red!75, left = of 11b] (12b) {};
					\node [draw =Red!75, above right = of 11b] (13b) {};
					\node [draw =Red!75, above = of 11b] (14b) {};
					\node [draw =Red!75, above left = of 11b] (15b) {};
					\node [draw =Red!75, below left = of 1b] (16b) {};
					\node [draw =Red!75, below left = of 16b] (17b) {};
					\node [draw =Red!75, left = of 17b] (18b) {};
					\node [draw =Red!75, below left = of 17b] (19b) {};
					\node [draw =Red!75, below = of 17b] (20b) {};
					\node [draw =Red!75, above left = of 17b] (21b) {};
					\node [draw =   Goldenrod!85!black, below = of 1b] (22b) {};
					\node [draw =   Goldenrod!85!black, below = of 22b] (23b) {};
					\node [draw =   Goldenrod!85!black, below left = of 22b] (24b) {};
					\node [draw =   Goldenrod!85!black, below right = of 22b] (25b) {};
					\node [draw =   Goldenrod!85!black, below = of 23b] (26b) {};
					\node [draw =   Goldenrod!85!black, below right = 5mm of 1b] (27b) {};
					\node [draw =   Goldenrod!85!black, below right = of 27b] (28b) {};
					\node [draw =   Goldenrod!85!black, right = of 27b] (29b) {};
					\node [draw =   Goldenrod!85!black, above right = of 27b] (30b) {};
					\node [draw =   Goldenrod!85!black, right = of 29b] (31b) {};
					\node [draw =   Goldenrod!85!black, above right = of 1b] (32b) {};
					\node [draw =   Goldenrod!85!black, above right = of 32b] (33b) {};
					\node [draw =   Goldenrod!85!black, above = of 32b] (34b) {};
					\node [draw =   Goldenrod!85!black, right = of 32b] (35b) {};
					\node [draw =   Goldenrod!85!black, above right = of 33b] (36b) {};
					\node [draw =Red!75, above left = of 15b] (37b) {};
					\node [draw =Red!75, above = of 14b] (38b) {};
					\node [draw =Red!75, left = of 18b] (39b) {};
					\node [draw =Red!75, below left = of 19b] (40b) {};
					\node [draw =Red!75, above left = of 7b] (41b) {};
					\node [draw =Red!75, left = of 8b] (42b) {};
					
					\draw (1b) -- (2b);
					\draw (2b) -- (midb);
					\draw (midb) -- (midmidb);
					\draw (midmidb) -- (3b);
					\draw (1b) -- (4b);
					\draw (4b) -- (5b);            
					\draw (5b) -- (6b);       
					\draw (5b) -- (7b);
					\draw (5b) -- (8b);
					\draw (5b) -- (9b);
					\draw (1b) -- (10b);
					\draw (10b) -- (11b);
					\draw (11b) -- (12b);
					\draw (11b) -- (13b);
					\draw (11b) -- (14b);
					\draw (11b) -- (15b);
					\draw (1b) -- (16b);
					\draw (16b) -- (17b);
					\draw (17b) -- (18b);
					\draw (17b) -- (19b);
					\draw (17b) -- (20b);
					\draw (17b) -- (21b);
					\draw (1b) -- (22b);
					\draw (22b) -- (23b);
					\draw (22b) -- (24b);
					\draw (22b) -- (25b);
					\draw (23b) -- (26b);
					\draw (1b) -- (27b);
					\draw (27b) -- (28b);
					\draw (27b) -- (29b);
					\draw (27b) -- (30b);
					\draw (29b) -- (31b);
					\draw (1b) -- (32b);
					\draw (32b) -- (33b);
					\draw (32b) -- (34b);
					\draw (32b) -- (35b);
					\draw (33b) -- (36b);
					\draw (15b) -- (37b);
					\draw (14b) -- (38b);
					\draw (18b) -- (39b);
					\draw (19b) -- (40b);
					\draw (7b) -- (41b);
					\draw (8b) -- (42b);
					
					\node[draw=none, below = 10mm of mida](t1){$\mathcal{T}^{\mathcolor{White}{\prime}}$};
					\node[draw=none, below= 10mm of midb](t2){$\mathcal{T}^{\prime}$};
				\end{tikzpicture}
				\caption{A generalized tree shift from $\mathcal{T}$ to $\mathcal{T}^{\prime}$}
				\label{fig:GTS}
			\end{figure}

				\section{Conclusion}
				\label{sect:Conclusion}

        The main contributions of this paper revolve around stochastic orderings in the context of the family $\mathbb{MPMRF}$. 
        Within a given tree, we have established
        supermodular orderings of synecdochic pairs, the pair given by a component of a random vector and its sum. This was done to assess if a component of the MRF contributes less or more than another to the sum,  depending on its corresponding vertex's position. Our findings on the ordering of synecdochic pairs have also found repercussions for allocation-related quantities, which invited their consideration as centrality indices. These findings were also key to deriving the result that allowed for establishing convex ordering of sums of MRFs defined on trees of different shapes. To thoroughly explore this latter result, we have designed a new partial order on trees and examined it exhaustively. This examination was done through the conception of Hasse diagrams, tools facilitating comparison along sequences of trees, and a discussion stemming from spectral graph theory on its relation to other poset of trees. Moreover, we have shown that the two stochastic orderings, $(N_v,M)\preceq_{sm}(N_w,M)$ and $M\preceq_{cx}M^{\prime}$, were two sides of the same coin, both linked through the condition $H_{v}^{\mathcal{T}_v}\preceq_{st}H_{w}^{\mathcal{T}_w}$, $v,w\in\mathcal{V}$. Our acquired comprehension of the poset $(\Omega_{\mathcal{T}},\preceq_{sha})$ thus transposes to both. 

        \bigskip

        One could dive deeper into the analysis of the poset $(\Omega_{\mathcal{T}}, \preceq_{sha})$ using graph-theoretic machinery. \cite{li2021poset} furthered the understanding of \cite{csikvari2010poset}'s poset, revealing results relating to the Wiener index and tree invariants' polynomials; likewise, we suspect more could be uncovered about $(\Omega_{\mathcal{T}}, \preceq_{sha})$ in connection with (spectral) graph theory. \cite{barghi2024ranking} discusses, among other things, the ordering of trees induced by some global closeness index
        and notably proves that the series tree [star tree] acts as its minimum [maximum] element. 
        Further analytical results on that ordering could perhaps, given $\preceq_{sha}$'s underlying relation to closeness, connect both in the same fashion that the ordering of synecdochic pairs from Theorem~\ref{th:OrderingAlloc} connected to centrality indices in Proposition~\ref{th:EffectsofSupermodularity}. Ergo, the conclusions of \cite{barghi2024ranking}'s comparative study could then transpose to $\preceq_{sha}$.

\begin{appendix}

				\section{Proof of Theorem \ref{th:OrderingAlloc}} \label{sect:proofConvexOrderMShape}
				The proof comprises two parts. In the first part, we dissect the joint distributions of $(N_v,M)$ and $(N_w,M)$ to provide an alternative stochastic representation of both vectors. In the second part, we rely on
				the mass transfer characterization of supermodular ordering given by \cite{meyer2015beyond} to establish comparison using this alternative representation. For a discussion on the relation tying stochastic orders to mass transfer, see \cite{muller2013duality}. Appendix~C of \cite{ansari2024supermodular} also relies on that characterization and provides further insight.
				Theorem~7 of \cite{cote2025tree} provides the expression of the pgf of $M$.
				Considering first $v$ as a root, we have 
				\begin{align}
					\mathcal{P}_M(t_m) &= \mathrm{e}^{\lambda\left(d-\sum_{e\in\mathcal{E}}\alpha_{e}\right) \left( \left(\sum_{j\in\mathcal{V}}\frac{1-\alpha_{(\mathrm{pa}(j),j)}}{d-\sum_{e\in\mathcal{E}}\alpha_{e}}\eta_{j}^{\mathcal{T}_v}(t\vecun{j})\right)-1\right)}\\
					&= \mathrm{e}^{\lambda\left(\eta_v^{\mathcal{T}_v}(t_m\vecun{v}) + \sum_{j\in\mathcal{V}\backslash\{v\}} (1-\alpha)\eta_{j}^{\mathcal{T}_v}(t_m\vecun{j})\right) - \lambda d (1-\alpha) }\notag\\
					&= \mathrm{e}^{\lambda\left(\sum_{k=0}^{\infty} \left(p_{H_v^{\mathcal{T}_v}}(k) + \sum_{j\in\mathcal{V}\backslash\{v\}} (1-\alpha)p_{H_j^{\mathcal{T}_v}}(k) \right) t_m^{k}\right) - \lambda d (1-\alpha) },\quad t_m\in[-1,1].
                \label{eq:sigmafinal1}
				\end{align}
				For notation purposes, let us define $\varsigma_k = p_{H_v^{\mathcal{T}_v}}(k) + \sum_{j\in\mathcal{V}\backslash\{v\}} (1-\alpha)p_{H_j^{\mathcal{T}_v}}(k)$, $k\in\mathbb{N}$. 
				Meanwhile, selecting $w$ instead as the root, we also have 
				\begin{equation}
					\mathcal{P}_M(t_m) = \mathrm{e}^{\lambda\left(\sum_{k=0}^{\infty} \left(p_{H_w^{\mathcal{T}_w}}(k) + \sum_{j\in\mathcal{V}\backslash\{w\}} (1-\alpha)p_{H_j^{\mathcal{T}_w}}(k) \right) t_m^{k}\right) - \lambda d (1-\alpha) },\quad t_m\in[-1,1].\label{eq:sigmafinal2}
				\end{equation}
				Since the choice of the root has no impact on the joint distribution of $\boldsymbol{N}$ (Theorem~2 of \cite{cote2025tree}) and due to the uniqueness of pgfs, the relation $\varsigma_k = p_{H_w^{\mathcal{T}_w}}(k) + \sum_{j\in\mathcal{V}\backslash\{w\}} (1-\alpha)p_{H_j^{\mathcal{T}_w}}(k)$ is valid for $k\in\mathbb{N}$ by comparing (\ref{eq:sigmafinal1}) and (\ref{eq:sigmafinal2}).
				We derive the joint pgfs of $(N_{v}, M)$ and $(N_{w}, M)$ from the joint pgf of $\boldsymbol{N}$ using Theorem~2.3 of \cite{blier2022generating} as follows
				\begin{align}
					\mathcal{P}_{N_{v}, M}(t_{v},t_m) &= \mathcal{P}_{\boldsymbol{N}}(t_m,t_m,\ldots, \underbrace{t_{v}t_m}_{\text{position } v}, \ldots, \underbrace{t_m}_{\text{position } w}, \ldots,t_m, t_m);
					\label{eq:OnBoitDuThe1}\\
					\mathcal{P}_{N_{w}, M}(t_{w},t_m) &= \mathcal{P}_{\boldsymbol{N}}(t_m,t_m,\ldots, \underbrace{t_m}_{\text{position } v}, \ldots, \underbrace{t_{w}t_m}_{\text{position } w}, \ldots,t_m, t_m),
					\label{eq:OnBoitDuThe2}
				\end{align}
				for $t_v,t_w,t_m\in[-1,1]$. From the expression of the joint pgf in (\ref{eq:JointPGF}), choosing $v$ as the root, relation \eqref{eq:OnBoitDuThe1} grants
				\begin{align}
					\mathcal{P}_{N_{v}, M}(t_{v},t_m)
					&= \mathrm{e}^{\lambda (t_{v} \eta^{\mathcal{T}_v}_v(t_m\vecun{v}) - 1)} \times \prod_{j\in\mathcal{V}\backslash\{v\}}\mathrm{e}^{\lambda(1-\alpha)(\eta_j^{\mathcal{T}_v}(t_m\vecun{j}) - 1)}\notag\\
					&= \mathrm{e}^{\lambda\left(\sum_{k=1}^{\infty}  p_{H_v^{\mathcal{T}_v}}(k) t_v t_m^k + \sum_{k=1}^{\infty}  \sum_{j\in\mathcal{V}\backslash\{v\}} (1-\alpha)p_{H_j^{\mathcal{T}_v}}(k)t_m^k \right)-\lambda d (1-\alpha)} \notag\\
					&= \mathrm{e}^{\lambda\left(\sum_{k=1}^{\infty} p_{H_v^{\mathcal{T}_v}}(k) t_v t_m^k + \sum_{k=1}^{\infty}   (\varsigma_k - p_{H_j^{\mathcal{T}_v}}(k)) t_m^k \right) - \lambda d(1-\alpha)},\quad t_v,t_m\in[-1,1].\label{eq:fgpNvM-1}
				\end{align}
				Similarly, choosing $w$ as the root and following the same steps,  (\ref{eq:OnBoitDuThe2}) yields
				\begin{equation}
					\mathcal{P}_{N_{w}, M}(t_{w},t_m) = \mathrm{e}^{\lambda\left(\sum_{k=1}^{\infty} p_{H_w^{\mathcal{T}_w}}(k) t_w t_m^k +  \sum_{k=1}^{\infty}  (\varsigma_k - p_{H_j^{\mathcal{T}_w}}(k)) t_m^k \right) - \lambda d(1-\alpha)},\quad t_w,t_m\in[-1,1].\label{eq:fgpNwM-1}
				\end{equation}
				
				To alleviate the notation for the remaining of the proof, define $a_k={p}_{H_v^{\mathcal{T}_v}}(k)$ and $b_k={p}_{H_w^{\mathcal{T}_w}}(k)$. Also, let $r=\sum_{k=0}^{\infty}{(b_k-a_k)\mathbb{1}_{\{a_k \leq b_k\}}}=\sum_{k=0}^{\infty}{(a_k-b_k)\mathbb{1}_{\{a_k \geq b_k\}}}$.
    One may note that $r$ corresponds to the total variation distance between $H_v^{\mathcal{T}_v}$ and $H_w^{\mathcal{T}_w}$.
    Consider a random variable $Y$ whose pgf is given by  
    \begin{equation}
				\mathcal{P}_{Y}(t) = \mathrm{e}^{(\lambda d (1-\alpha)-\lambda^{*})(\frac{1}{d(1-\alpha)-(1+r)}\sum_{k=0}^{\infty}(\varsigma_k - a_k - (b_k-a_k)\mathbb{1}_{\{a_k\leq b_k\}}t^k - 1)},\quad t\in[-1,1],
					\label{eq:pgfUpsilon}
				\end{equation}
     where $\lambda^{*} = \lambda(1+r)$.
    From the definitions of $a_k$, $b_k$ and $\varsigma_k$, one directly sees $\varsigma_k - a_k \geq 0$ and $\varsigma_k-b_k\geq 0$ for every $k\in\mathbb{N}$. Thus, the quantity $\varsigma_k - a_k - (b_k-a_k)\mathbb{1}_{\{a_k\leq b_k\}}$, which is equal to $\varsigma_k - b_k - (a_k-b_k)\mathbb{1}_{\{a_k\geq b_k\}}$, is nonnegative for every $k\in\mathbb{N}$. The pgf in (\ref{eq:pgfUpsilon}) is thus, indeed, a valid pgf.
    
    The joint pgf in (\ref{eq:fgpNvM-1}) then becomes  
				\begin{align}
					\mathcal{P}_{N_v,M}(t_v,t_m) 
					&= \mathrm{e}^{\lambda(\sum_{k=0}^{\infty} a_kt_vt^k_m + \sum_{k=0}^{\infty}(\varsigma_k-a_k)t^k_m) - \lambda d(1-\alpha)} \notag \\
					&= \mathrm{e}^{\lambda(\sum_{k=0}^{\infty} a_kt_vt^k_m  + \sum_{k=0}^{\infty}(b_k-a_k)\mathbb{1}_{\{a_k \leq b_k\}}t^k_m + \sum_{k=0}^{\infty}(\varsigma_k-a_k-(b_k-a_k)\mathbb{1}_{\{a_k\leq b_k\}})t^k_m)- \lambda d (1-\alpha)} \notag \\
					&= \mathrm{e}^{\lambda^{*}( \frac{1}{1+r}\left(\sum_{k=0}^{\infty} a_kt_vt^k_m  + \sum_{k=0}^{\infty}(b_k-a_k)\mathbb{1}_{\{a_k \leq b_k\}}t^k_m\right) -1)} \times \mathcal{P}_Y(t_m), \quad t_v,t_m\in[-1,1].
     \label{eq:jointpgf-NvM}
				\end{align}
				Likewise, (\ref{eq:fgpNwM-1}) becomes
				\begin{align}
					\mathcal{P}_{N_w,M}(t_w,t_m)
					&= \mathrm{e}^{\lambda^{*}( \frac{1}{1+r}\left(\sum_{k=0}^{\infty} b_kt_vt^k_m  + \sum_{k=0}^{\infty}(a_k-b_k)\mathbb{1}_{\{a_k \geq b_k\}}t^k_m\right) -1)} \times \mathcal{P}_Y(t_m) \quad t_w,t_m\in[-1,1], 
					\label{eq:jointpgf-NwM}
				\end{align}
				since $\varsigma_k - a_k - (b_k-a_k)\mathbb{1}_{\{a_k\leq b_k\}}=\varsigma_k - b_k - (a_k-b_k)\mathbb{1}_{\{a_k\geq b_k\}}$ for every $k\in\mathbb{N}$. 
				Next, let $(\ddot{I}_v, \ddot{H}_v)$ and $(\ddot{I}_w, \ddot{H}_w)$ be vectors of random variables whose joint pgfs are respectively given by
				\begin{align}
					\mathcal{P}_{\ddot{I}_v, \ddot{H}_v}(t_v,t_m) &= \frac{1}{1+r}\left(\sum_{k=0}^{\infty}a_kt_vt^k_m + \sum_{k=0}^{\infty}(b_k-a_k)\mathbb{1}_{\{a_k\leq b_k\}} t_m^k\right);
					\label{eq:jointIH-v}\\
					\mathcal{P}_{\ddot{I}_w, \ddot{H}_w}(t_w,t_m) &= \frac{1}{1+r}\left(\sum_{k=0}^{\infty}b_kt_wt^k_m + \sum_{k=0}^{\infty}(a_k-b_k)\mathbb{1}_{\{a_k\geq b_k\}} t_m^k\right),
					\label{eq:jointIH-W}
				\end{align}
				for $t_v,t_w,t_m\in[-1,1]$. From (\ref{eq:jointIH-v}) and (\ref{eq:jointIH-W}), we note $\mathcal{P}_{\ddot{I}_v,\ddot{H}_v}(t, 1) = \mathcal{P}_{\ddot{I}_w,\ddot{H}_w}(t, 1)$, for all $t\in[-1,1]$. Hence, $\ddot{I}_v\stackrel{d}{=}\ddot{I}_w$; they indeed both follow a Bernoulli distribution of parameter $\tfrac{1}{1+r}$. We similarly deduce $\ddot{H}_v \stackrel{d}{=} \ddot{H}_w$.
				Finally, let $\ddot{N}$ be a random variable following a Poisson distribution of parameter $\lambda^*$. Its pgf is given by
				\begin{equation}
					\mathcal{P}_{\ddot{N}}(t) = \mathrm{e}^{\lambda^*(t-1)}, \quad t\in[-1,1].
					\label{eq:pgfNtrema}
				\end{equation}
				Using (\ref{eq:pgfUpsilon}), (\ref{eq:jointIH-v}), (\ref{eq:jointIH-W}) and (\ref{eq:pgfNtrema}), we rewrite (\ref{eq:jointpgf-NvM}) and (\ref{eq:jointpgf-NwM}) as
				\begin{align}
					\mathcal{P}_{N_v,M}(t_v,t_m) = \mathcal{P}_{\ddot{N}}\left( \mathcal{P}_{\ddot{I}_v,\ddot{H}_v}(t_v,t_m)\right) \times\mathcal{P}_{Y}(t_m);
     \label{eq:fgpPGF-NvM} \\
					\mathcal{P}_{N_w,M}(t_w,t_m) = \mathcal{P}_{\ddot{N}}\left( \mathcal{P}_{\ddot{I}_w,\ddot{H}_w}(t_w,t_m)\right) \times\mathcal{P}_{Y}(t_m),\label{eq:fgpPGF-NwM}
				\end{align}
				for $t_v,t_w,t_m\in[-1,1]$. Hence, given the uniqueness of joint pgfs, we deduce from (\ref{eq:fgpPGF-NvM}) and (\ref{eq:fgpPGF-NwM})
				
				\begin{equation}
					(N_v,M)\stackrel{d}{=}\left(\sum_{i=1}^{\ddot{N}} \ddot{I}_{v,i}, \sum_{i=1}^{\ddot{N}} \ddot{H}_{v,i} + Y\right); \quad\;\; 
					(N_w,M)\stackrel{d}{=}\left(\sum_{i=1}^{\ddot{N}} \ddot{I}_{w,i}, \sum_{i=1}^{\ddot{N}} \ddot{H}_{w,i} + Y\right),
					\label{eq:alternative-NwM}
				\end{equation}
				where $\{(\ddot{I}_{v,i}, \ddot{H}_{v,i}),\,i\in\mathbb{N}_1\}$ and $\{(\ddot{I}_{w,i}, \ddot{H}_{w,i}),\,i\in\mathbb{N}_1\}$ are sequences of independent vectors of random variables  identically distributed to $(\ddot{I}_v, \ddot{H}_v)$ and $(\ddot{I}_w,\ddot{H}_w)$ respectively, and where $\ddot{N}$ and $Y$ are independent random variables, also independent of the two sequences.  Relations in 
    (\ref{eq:alternative-NwM}) provide the sought alternative representations.

				The second part of the demonstration rests on Theorem~1 of \cite{meyer2015beyond}. 
				Before continuing, let us present it in a bivariate context, in details and without proof, as follows. 
				\begin{theoremeV}{1 of \cite{meyer2015beyond}}
					Consider two vectors of random variables, $\boldsymbol{X}$ and $\boldsymbol{X}^{\prime}$, taking values in $\mathbb{N}^2$ with joint pmfs $p_{\boldsymbol{X}}$ and $p_{\boldsymbol{X}^{\prime}}$. Then, $\boldsymbol{X}\preceq_{sm} \boldsymbol{X}^{\prime}$ if and only if there exist nonnegative coefficients $(\gamma_{_\beta})_{\beta\in\boldsymbol{\beta}\!\!\!\!\boldsymbol{\beta}}$ such that
					\begin{equation}
						p_{\boldsymbol{X}^{\prime}}(\boldsymbol{x}) = p_{\boldsymbol{X}}(\boldsymbol{x}) + \sum_{\beta\in\boldsymbol{\beta}\!\!\!\!\boldsymbol{\beta}}\gamma_{_\beta} \beta(\boldsymbol{x}),\quad\boldsymbol{x}\in\mathbb{N}^2,\label{eq:MEYER}
					\end{equation}
					where $\boldsymbol{\beta}\!\!\!\!\boldsymbol{\beta}$ denotes the set $\left\{\beta^{(i,j)}_{k,l},\,i,j,k,l\in\mathbb{N}\right\}$ of functions $\mathbb{N}^2\mapsto\{-1,0,1\}$ given by
					\begin{equation*}
						\beta^{(i,j)}_{k,l}(i,j) = \beta^{(i,j)}_{k,l}(i+k,j+l) = 1, \quad\quad  \beta^{(i,j)}_{k,l}(i+k,j) = \beta^{(i,j)}_{k,l}(i,j+l) = -1,
					\end{equation*}
					and $\beta^{(i,j)}_{k,l} = 0$ elsewhere. 
				\end{theoremeV}

				Having in mind the identity
    $
					a_j-b_j = \left(\sum_{k=0}^{j} a_k - \sum_{k=0}^{j} b_k\right) - \left(\sum_{k=0}^{j-1} a_k - \sum_{k=0}^{j-1} b_k \right), 
	$			
				and in the spirit of (\ref{eq:MEYER}), we see from (\ref{eq:jointIH-v}) and (\ref{eq:jointIH-W}) the following relation between the pmfs of $(\ddot{I}_v,\ddot{H}_v)$ and $(\ddot{I}_w,\ddot{H}_w)$:
				\begin{equation}
					p_{\ddot{I}_w,\ddot{H}_w}(\boldsymbol{x}) = p_{\ddot{I}_v,\ddot{H}_v}(\boldsymbol{x}) + \sum_{j=0}^{\infty}\left(\sum_{k=0}^j{a_k} - \sum_{k=0}^jb_k\right)\beta_{1,1}^{(0,j)}(\boldsymbol{x}).
					\label{eq:Strulovici}
				\end{equation}
				The coefficients  $\left(\sum_{k=0}^j{a_k} - \sum_{k=0}^jb_k\right)$ in (\ref{eq:Strulovici}) are nonnegative for all $j\in\mathbb{N}$, from the assumption $H_{v}^{\mathcal{T}_v}\preceq_{st}H_{w}^{\mathcal{T}_w}$. The coefficients are indeed differences of their respective cdfs. Theorem~1 of \cite{meyer2015beyond} conveys, as a result of (\ref{eq:Strulovici}), that
	$		
     \left(\ddot{I}_v, \ddot{H}_v\right)\preceq_{sm}\left(\ddot{I}_w, \ddot{H}_w\right).
	$	
				Evoking Lemma~\ref{th:SupermodularSeverity} (see Appendix~\ref{sect:LemmaContrib}), we deduce
			$
     \left(\sum_{i=1}^{\ddot{N}}\ddot{I}_v, \sum_{i=1}^{\ddot{N}}\ddot{H}_v\right)\preceq_{sm}\left(\sum_{i=1}^{\ddot{N}}\ddot{I}_w, \sum_{i=1}^{\ddot{N}}\ddot{H}_w\right), 
     $
				with the same assumptions on the involved random variables as in 
    (\ref{eq:alternative-NwM}). 
				From Corollary 9.A.10 of \cite{shaked2007}, with $Y$ independent of the other random variables, it follows that 
	$		
     \left(\sum_{i=1}^{\ddot{N}}\ddot{I}_v, \sum_{i=1}^{\ddot{N}}\ddot{H}_v+Y\right)\preceq_{sm}\left(\sum_{i=1}^{\ddot{N}}\ddot{I}_w, \sum_{i=1}^{\ddot{N}}\ddot{H}_w + Y\right)
	$.	
		 Finally, referring to 
    (\ref{eq:alternative-NwM}), we conclude $(N_v,M)\preceq_{sm}(N_w,M)$. \hfill$\square$

    \section{Proof of Theorem \ref{th:ConvexOrderMShape}}
    \label{sect:proofconvexMshape}
		Let $\boldsymbol{N}^{*} = (N_i, \, i\in\mathcal{V}^*)$ and $M^*= \sum_{i\in\mathcal{V}^*} N_i$. Given Corollary~1 of \cite{cote2025tree}, $\boldsymbol{N}^*\sim \mathrm{MPMRF}(\lambda, \alpha \boldsymbol{1}_{|\mathcal{E}^{*}|}, \tau^*)$. By Theorem~\ref{th:OrderingAlloc}, since $H_v^{\tau^*_v}\preceq_{st} H_w^{\tau^*_w}$, we have 
			\begin{equation}
				(N_v, M^{*})\preceq_{sm}(N_w, M^{*}).
				\label{eq:proofMconvShape-part1-final}
			\end{equation}
			
			Also, $(N_{w}, M^{*}) \stackrel{d}{=} (N_{w}^{\prime}, M^{*\prime})$ because the only elements of discordance, $N_{u}$ and $N_{u}^{\prime}$, are not part of $M^{*}$ nor $M^{*\prime}$. Hence, given relation (\ref{eq:proofMconvShape-part1-final}),
	$			
    (N_{v}, M^*) \preceq_{sm} (N_{w}^{\prime}, M^{*\prime})$. 
			By Theorem 9.A.14 of \cite{shaked2007}, the relation
			\begin{equation}
				\left(\sum^{N_{v}}_{i=1} I^{(\alpha_{(u,v)})}_i, M^*\right) \preceq_{sm} \left(\sum^{N_{w}^{\prime}}_{i=1} I^{(\alpha_{(u,w)}^{\prime})}_i, M^{*\prime}\right) 
				\label{eq:proofMconvShape-part2-2}
			\end{equation}
			also holds. Recall that $\alpha_{(u,v)} = \alpha_{(u,w)}^{\prime}$. Then, by Corollary 9.A.10 of \cite{shaked2007}, relation in (\ref{eq:proofMconvShape-part2-2}) becomes
			\begin{equation}
				\left(\sum^{N_{v}}_{i=1} I^{(\alpha_{(u,v)})}_i + L_{u}, M^*\right) \preceq_{sm} \left(\sum^{N_{w}^{\prime}}_{i=1} I^{(\alpha_{(u,w)}^{\prime})}_i + L_{u}^{\prime}, M^{*\prime}\right), 
				\label{eq:proofMconvShape-part2-3}
			\end{equation}
			since $L_{u} \stackrel{d}{=} L_{u}^{\prime}$, and both are independent of all other implicated random variables. Rewriting (\ref{eq:proofMconvShape-part2-3}), we have
			\begin{equation}
				(N_{u}, M^*) \preceq_{sm} (N_{u}^{\prime}, M^{*\prime}).
				\label{eq:proofMconvShape-part2-4}
			\end{equation}
			Note that $\{N_j, \, j\in\mathcal{V}^{\dagger}\backslash\{u\}\}\stackrel{d}{=}\{N_j^{\prime}, \, j\in\mathcal{V}^{\dagger}\backslash\{u\}\}$, since their dependence relations are not affected by the change in the anchorage of vertex $u$ given the global Markov property. 
   Similarly, for every ${j} \in \mathcal{V}^{\dagger}$,
	\begin{equation}
		(N_{j} | N_{\mathrm{pa}(j)} = \theta) \stackrel{d}{=} (N_{j}^{\prime} | N_{\mathrm{pa}(j)}^{\prime} = \theta), \quad \theta\in\mathbb{N}, 
		\label{eq:SupermodularN-distr2}
	\end{equation} 
 with the filial relation taken according to a root chosen among $\mathcal{V}^{*}$. Besides, because $(N_j|N_v = \theta)$ follows a binomial distribution of size parameter $\theta$, the order relation
	\begin{equation}
		(N_{j} | N_{\mathrm{pa}(j)} = \theta_1) \preceq_{st} (N_{j} | N_{\mathrm{pa}(j)} = \theta_2), \quad \theta_1, \theta_2 \in \mathbb{N}, \; \theta_1 \leq \theta_2
		\label{eq:SupermodularN-Cond1-StoDomN}
	\end{equation}
	holds for all $j \in \mathcal{V}^{\dagger}$. In the same vein, 
 \begin{equation}
		(N_{j}^{\prime} | N_{\mathrm{pa}(j)}^{\prime} = \theta_1) \preceq_{st} (N_{j}^{\prime} | N_{\mathrm{pa}(j)}^{\prime} = \theta_2), \quad \theta_1, \theta_2 \in \mathbb{N}, \; \theta_1 \leq \theta_2.
		\label{eq:SupermodularN-Cond2-StoDomN}
	\end{equation}
	 Given the global Markov properties of $\boldsymbol{N}$ and $\boldsymbol{N}^{\prime}$, and relations (\ref{eq:proofMconvShape-part2-4}), 
	(\ref{eq:SupermodularN-distr2}),
	(\ref{eq:SupermodularN-Cond1-StoDomN}) and (\ref{eq:SupermodularN-Cond2-StoDomN}), one may then proceed as in the proof of Theorem~6 in \cite{cote2025tree}, repeatedly invoking Theorem 9.A.15 of \cite{shaked2007} using $N_{u}$ and $N_{u}^{\prime}$ as mixing random for the first iteration, $(N_{u}, (N_i : i\in\mathrm{ch}(u)))$ for the second, and so on until all random vriables associated to vertices constituting $\tau^{\dagger}$ are captured. This operation yields 
			\begin{equation*}
				(N_{u}, M^{*}, (N_j:\, j\in\mathcal{V}^{\dagger}\backslash\{u\}))\preceq_{sm}	(N_{u}^{\prime}, M^{*}, (N_j^{\prime}:\, j\in\mathcal{V}^{\dagger}\backslash\{u\})).
			\end{equation*}
			From Theorem 3.1 of \cite{muller1997stop}, we deduce $(N_{u} + M^{*} + \sum_{j\in\mathcal{V}^{\dagger}\backslash\{u\}} N_j) \preceq_{cx} (N_{u}^{\prime} + M^{*\prime} + \sum_{j\in\mathcal{V}^{\dagger}\backslash\{u\}} N_j^{\prime})$, which is the desired result $M \preceq_{cx} M^{\prime}$.  
    \hfill$\square$
				\section{Lemma~\ref{th:SupermodularSeverity} and Lemma~\ref{th:LemmaContrib}}
				\label{sect:LemmaContrib}
				
				\begin{lem}
					\label{th:SupermodularSeverity}
					Consider two vectors of random variables $(X_1,X_2)$ and $(X_1^{\prime}, X_2^{\prime})$ such that $(X_1,X_2)\preceq_{sm}(X_1^{\prime},X_2^{\prime})$. For a random variable $W$ taking values in $\mathbb{N}$, we observe the relation 
					\begin{equation*}
						\left(\sum_{i=1}^{W}X_{1,i}, \sum_{i=1}^{W}X_{2,i}\right)\preceq_{sm}\left(\sum_{i=1}^W X_{1,i}^{\prime}, \sum_{i=1}^W X_{2,i}^{\prime}\right),
					\end{equation*}
					where $\{(X_{1,i}, X_{2,i}),\,i\in\mathbb{N}_1\}$ and $\{(X_{1,i}^{\prime}, X_{2,i}^{\prime}),\,i\in\mathbb{N}_1\}$ are sequences of independent vectors of random variables identically distributed to $(X_1,X_2)$ and $(X_1^{\prime}, X_2^{\prime})$ respectively, and are independent of $W$. 
				\end{lem}
				
				\begin{proof}
					Since $x\mapsto x\mathbb{1}_{\{w\geq k\}}$ is increasing, Corollary~9.A.10 of \cite{shaked2007} implies that the relation 
					\begin{equation*}
						\left(X_1\mathbb{1}_{\{W\geq k\}}, X_2\mathbb{1}_{\{W\geq k\}}\right) \preceq_{sm} \left(X_1^{\prime}\mathbb{1}_{\{W\geq k\}},X_2^{\prime}\mathbb{1}_{\{W\geq k\}}\right)
					\end{equation*}
					holds for all $k\in\mathbb{N}$. Hence, one has
						$(X_{1,i}\mathbb{1}_{\{W\geq i\}}, X_{2,i}\mathbb{1}_{\{W\geq i\}}) \preceq_{sm} (X_{1,i}^{\prime}\mathbb{1}_{\{W\geq i\}},X_{2,i}^{\prime}\mathbb{1}_{\{W\geq i\}}),
	$				
					for all $i\in\mathbb{N}*$.
					Then, we sum every vector of the sequence $\{(X_{1,i}\mathbb{1}_{\{W\geq i\}}, X_{2,i}\mathbb{1}_{\{W\geq i\}}),\, i\in\mathbb{N}*\}$, componentwise, and separately $ \{(X_{1,i}^{\prime}\mathbb{1}_{\{W\geq i\}},X_{2,i}^{\prime}\mathbb{1}_{\{W\geq i\}}),\, i\in\mathbb{N}_1\}$, while we repeatedly apply Theorem~9.A.12 of \cite{shaked2007} to obtain
					\begin{equation*}
						\left(\sum_{i=1}^{\infty} X_{1,i}\mathbb{1}_{\{W\geq i\}}, \sum_{i=1}^{\infty} X_{2,i}\mathbb{1}_{\{W\geq i\}}\right)\preceq_{sm} \left(\sum_{i=1}^{\infty} X_{1,i}^{\prime}\mathbb{1}_{\{W\geq i\}}, \sum_{i=1}^{\infty} X_{2,i}^{\prime}\mathbb{1}_{\{W\geq i\}}\right) ,
					\end{equation*}
					which rewrites $\left(\sum_{i=1}^{W}X_{1,i}, \sum_{i=1}^{W}X_{2,i}\right)\preceq_{sm}\left(\sum_{i=1}^W X_{1,i}^{\prime}, \sum_{i=1}^W X_{2,i}^{\prime}\right)$.

				\end{proof}
				
				\begin{lem}
					\label{th:LemmaContrib}
					Let $\boldsymbol{X}=(X_i,\, i\in\{1,\ldots,d\})$ be a vector of random variables taking on values in $\mathbb{N}^d$ and $Y$ be the sum of its components. For $v,w\in\{1,\ldots,d\}$, if  \begin{equation}
						\mathrm{E}[X_v\mathbb{1}_{\{Y\geq k\}}]\geq\mathrm{E}[X_w\mathbb{1}_{\{Y\geq k\}}],\quad \text{for every } k\in\mathbb{N}, 
						\label{eq:LemmaContrib-1}
					\end{equation} then $\mathcal{C}_{\kappa}^{\mathrm{TVaR}}(X_v; Y) \leq \mathcal{C}_{\kappa}^{\mathrm{TVaR}}(X_w; Y)$ for all $\kappa\in[0,1)$.
				\end{lem}
				\begin{proof}
					Fix $\kappa \in [0,1)$. Let $q_\kappa = \mathrm{VaR}_{\kappa}(Y)$ and $\xi_{\kappa} = \frac{F_Y(q_{\kappa})-\kappa}{p_Y(q_{\kappa})}$. Note that $\xi_{\kappa}\in[0,1]$ for every $\kappa\in[0,1)$. Let  $\Delta_{\kappa} = \mathcal{C}^{\mathrm{TVaR}}_{\kappa}(X_w,Y) - \mathcal{C}^{\mathrm{TVaR}}_{\kappa}(X_v,Y)$. We do not know whether $\mathrm{E}[X_v\mathbb{1}_{\{Y=q_{\kappa}\}}]\geq\mathrm{E}[X_w\mathbb{1}_{\{Y=q_{\kappa}\}}]$ or vice versa; thus, we prove $\Delta_{\kappa}\geq 0$ under both possibilities. If $\mathrm{E}[X_v\mathbb{1}_{\{Y=q_{\kappa}\}}]\geq\mathrm{E}[X_w\mathbb{1}_{\{Y=q_{\kappa}\}}]$, we have
					\begin{align}
						\Delta_{\kappa}&= \frac{1}{1-\kappa}\left(\mathrm{E}[X_w\mathbb{1}_{\{Y\geq \mathrm{q_{\kappa}+1}\}}] - \mathrm{E}[X_v\mathbb{1}_{\{Y\geq \mathrm{q_{\kappa}+1}\}}] + \xi_{\kappa}\left(\mathrm{E}[X_w\mathbb{1}_{\{Y=q_{\kappa}\}}]-\mathrm{E}[X_v\mathbb{1}_{\{Y=q_{\kappa}\}}]\right) \right)\notag\\
						&\geq  \frac{1}{1-\kappa}\left(\mathrm{E}[X_w\mathbb{1}_{\{Y\geq \mathrm{q_{\kappa}+1}\}}] - \mathrm{E}[X_v\mathbb{1}_{\{Y\geq \mathrm{q_{\kappa}+1}\}}] + 1\times\left(\mathrm{E}[X_w\mathbb{1}_{\{Y=q_{\kappa}\}}]-\mathrm{E}[X_v\mathbb{1}_{\{Y=q_{\kappa}\}}]\right)\right)\notag\\
						&= \frac{1}{1-\kappa}\left(\mathrm{E}[X_w\mathbb{1}_{\{Y\geq \mathrm{q_{\kappa}}\}}] - \mathrm{E}[X_v\mathbb{1}_{\{Y\geq \mathrm{q_{\kappa}}\}}]\right), \notag
					\end{align}
					which is positive given (\ref{eq:LemmaContrib-1}). Conversely, if $\mathrm{E}[X_v\mathbb{1}_{\{Y=q_{\kappa}\}}]\leq\mathrm{E}[X_w\mathbb{1}_{\{Y=q_{\kappa}\}}]$, we have
					\begin{align}
	\Delta_{\kappa}
						&\geq  \frac{1}{1-\kappa}\left(\mathrm{E}[X_w\mathbb{1}_{\{Y\geq \mathrm{q_{\kappa}+1}\}}] - \mathrm{E}[X_v\mathbb{1}_{\{Y\geq \mathrm{q_{\kappa}+1}\}}] + 0\times\left(\mathrm{E}[X_w\mathbb{1}_{\{Y=q_{\kappa}\}}]-\mathrm{E}[X_v\mathbb{1}_{\{Y=q_{\kappa}\}}]\right)\right)\notag\\
						&= \frac{1}{1-\kappa}\left(\mathrm{E}[X_w\mathbb{1}_{\{Y\geq \mathrm{q_{\kappa}+1}\}}] - \mathrm{E}[X_v\mathbb{1}_{\{Y\geq \mathrm{q_{\kappa}+1}\}}]\right), \notag
					\end{align}
					which is also positive given (\ref{eq:LemmaContrib-1}). Therefore, the inequality $\Delta_{\kappa}$ holds for all $\kappa\in[0,1)$. 
				\end{proof}

                \section{Hasse diagrams for the poset defined in Section \ref{sect:SHA}}
				\label{sect:ShapesBonanza}

                Let $\Omega_{\mathcal{T}}^{(d)}$ be the set of all distinct $n_d$ trees of $d$ nodes $\mathcal{V} = \{1,\dots,d\}$. Note that $n_d = 2, 3, 6, 11, 23, 47$ for $d = 4,5,6,7,8,9$ respectively (see Enumeration of trees in \cite{knuth1997art}, pp 386-388, or sequence A000055 from \href{https://oeis.org/A000055}{OEIS} (\cite{oeistree}) for details about the computation of $n_d$).
                In Appendix 3 of \cite{harary2018graph}, one finds
                the depictions of all the elements within the set $\Omega_{\mathcal{T}}^{(d)}$ for each $d \in \{4,\dots, 9\}$.  The partial order introduced in Section \ref{sect:SHA} allows to establish the relations among the $n_d$ trees within the finite set $\Omega_{\mathcal{T}}^{(d)}$ for each $d \in \{4,\dots, 9\}$. In Figure \ref{fig:ShapesBonanza2} (page 22), we depict the Hasse diagrams of the posets $(\Omega_{\mathcal{T}}^{(d)},\preceq_{sha})$ for every $d\in \{4,\dots, 9\}$. The implied orientation of the order is upward, \textit{i.e.}, of two linked trees, the uppermost is greater than the other based on $\preceq_{sha}$. 
                Notably, these Hasse diagrams prove to be useful to illustrate the minimal and maximal elements in the sense of this partial order of $(\Omega_{\mathcal{T}}^{(d)},\preceq_{sha})$ for each $d\in \{4,\dots, 9\}$. Also, for $d\in \{4,\dots, 8\}$, we observe that $(\Omega_{\mathcal{T}}^{(d)},\preceq_{sha})$ is a lattice.

				\begin{figure}[bt]
    \centering

						};
						\draw[very thick] (S7) -- (S6);
						\draw[very thick] (S6) -- (E2);
						\draw[very thick] (E2) -- (E1);
						\draw[very thick] (E2) -- (D5);
						\draw[very thick] (E1) -- (D4);
						\draw[very thick] (D5) -- (D4);
						\draw[very thick] (D5) -- (S5);
						\draw[very thick] (D4) -- (M3);
						\draw[very thick] (D4) -- (D3);
						\draw[very thick] (S5) -- (D3);
						\draw[very thick] (M3) -- (C5);	
						\draw[very thick] (D3) -- (D2);	
						\draw[very thick] (D3) -- (C5);	
						\draw[very thick] (D3) -- (D1);	
						\draw[very thick] (D2) -- (S4);	
						\draw[very thick] (C5) -- (S4);	
						\draw[very thick] (C5) -- (M2);	
						\draw[very thick] (D1) -- (M2);	
						\draw[very thick] (S4) -- (C2);	
						\draw[very thick] (M2) -- (C4);	
						\draw[very thick] (C4) -- (C3);	
						\draw[very thick] (C4) -- (M1);	
						\draw[very thick] (C3) -- (C2);	
						\draw[very thick] (M1) -- (A2);	
						\draw[very thick] (C2) -- (C1);	
						\draw[very thick] (C2) -- (A2);	
						\draw[very thick] (C1) -- (A1);	
						\draw[very thick] (A2) -- (A1);	
						\draw[very thick] (A1) -- (S3);	
						\draw[very thick] (S3) -- (S2);	
						
				%
				%

						};
						\draw[thick] (S8) -- (S7);	
						\draw[thick] (S7) -- (F1);
						\draw[thick] (F1) -- (F2);
						\draw[thick] (F1) -- (E8);
						\draw[thick] (F2) -- (E7);
						\draw[thick] (E8) -- (E7);
						\draw[thick] (E8) -- (S6);
						\draw[thick] (E7) -- (Q4);
						\draw[thick] (E7) -- (E6);
						\draw[thick] (S6) -- (E6);
						\draw[thick] (Q4) -- (D9);
						\draw[thick] (Q4) -- (Q3);
						\draw[thick] (E6) -- (E3);
						\draw[thick] (E6) -- (D9);
						\draw[thick] (E6) -- (E5);
						\draw[thick] (E3) -- (E4);
						\draw[thick] (D9) -- (D8);
						\draw[thick] (D9) -- (S5);
						\draw[thick] (E5) -- (E1);
						\draw[thick] (E5) -- (S5);
						\draw[thick] (E5) -- (E2);
						\draw[thick] (E4) -- (E1);
						\draw[thick] (Q3) -- (K);
						\draw[thick] (Q3) -- (D8);
						\draw[thick] (E1) -- (D8);
						\draw[thick] (S5) -- (D5);
						\draw[thick] (E2) -- (D7);
						\draw[thick] (D8) -- (Q2);
						\draw[thick] (K) -- (N4);
						\draw[thick] (Q2) -- (Q1);
						\draw[thick] (Q2) -- (C8);
						\draw[thick] (Q2) -- (D5);
						\draw[thick] (D6) -- (D8);
						
						\draw[thick] (Q1) -- (D2);
						\draw[thick] (D6) -- (N4);
						\draw[thick] (D6) -- (D7);
						\draw[thick] (D6) -- (D5);
						\draw[thick] (N4) -- (C8);
						\draw[thick] (N4) -- (N2);
						\draw[thick] (D7) -- (C8);
						\draw[thick] (D7) -- (D2);
						\draw[thick] (D7) -- (D4);
						\draw[thick] (D5) -- (D2);
						\draw[thick] (D5) -- (D4);
						\draw[thick] (C8) -- (N3);
						\draw[thick] (C8) -- (C7);
						\draw[thick] (D2) -- (N2);
						\draw[thick] (D2) -- (D1);
						\draw[thick] (D4) -- (D1);
						\draw[thick] (D4) -- (D3);
						\draw[thick] (N2) -- (N3);
						\draw[thick] (D5) -- (C7);
						\draw[thick] (N3) -- (N1);
						\draw[thick] (N3) -- (C6);
						\draw[thick] (C7) -- (N1);
						\draw[thick] (C7) -- (C6);
						\draw[thick] (C7) -- (S4);
						\draw[thick] (D1) -- (N1);
						\draw[thick] (D1) -- (C6);
						\draw[thick] (D3) -- (S4);
						\draw[thick] (N1) -- (C5);
						\draw[thick] (N1) -- (M2);
						\draw[thick] (C6) -- (C5);
						\draw[thick] (C6) -- (M2);
						\draw[thick] (C5) -- (C4);
						\draw[thick] (M2) -- (M1);
						\draw[thick] (M2) -- (A2);
						\draw[thick] (C4) -- (M1);
						\draw[thick] (C4) -- (C3);
						\draw[thick] (M1) -- (A2);
						\draw[thick] (C3) -- (C2);
						\draw[thick] (S4) -- (C2);
						\draw[thick] (C2) -- (A2);
						\draw[thick] (C2) -- (C1);
						\draw[thick] (A2) -- (A1);
						\draw[thick] (C1) -- (A1);
						\draw[thick] (A1) -- (S3);
						\draw[thick] (S3) -- (S2);
						
					\end{tikzpicture}
					\caption{Hasse diagrams of $(\Omega_{\mathcal{T}}^{(d)},\preceq_{sha})$, for $d = 4,5,6,7,8, 9$}
					\label{fig:ShapesBonanza2}
				\end{figure}
			
\end{appendix}

\section*{Acknowledgments}
This work was partially supported by the Natural Sciences and Engineering Research Council of Canada (Cossette: 04273; Marceau: 05605; Côté), by the Fonds de recherche du Québec -- Nature et technologie (Côté: 335878), and by the Chaire en actuariat de l'Université Laval (Cossette, Côté, Marceau: FO502320).

\bibliographystyle{apalike} 
\bibliography{bibPoissonAR1}

\end{document}